\documentclass[a4paper]{amsart}

\usepackage{pdfsync}
\usepackage[english]{babel}
\usepackage[utf8]{inputenc}
\usepackage{amsmath,amsthm,amsfonts,amssymb}
\usepackage{accents,ulem,mathrsfs,setspace} %mathscr command

\usepackage[colorlinks=true,naturalnames=true,urlcolor=blue,linkcolor=blue,citecolor=blue]{hyperref}
\usepackage[dvipsnames]{xcolor}
\usepackage{color,graphicx,pictexwd,dcpic}
\usepackage{enumitem,multirow}
\usepackage{subcaption}
\usepackage{tikz}
\usetikzlibrary{positioning,shapes,shadows,automata,arrows,decorations,cd, trees}
\usetikzlibrary{lindenmayersystems}
\pgfdeclarelindenmayersystem{chaotic cantor set}{
  \symbol{R}{\pgflsystemdrawforward}
  \symbol{B}{\pgflsystemdrawforward}
  \symbol{v}{\pgflsystemmoveforward}
  \symbol{V}{\pgflsystemmoveforward}
  \rule{R -> RVB}
  \rule{B -> RvRvR}
  \rule{v -> vv}
  \rule{V -> VVV}
}
\pgfdeclarelindenmayersystem{cantor set}{
  \symbol{N}{\pgflsystemdrawforward}
  \symbol{v}{\pgflsystemmoveforward}
  \rule{B -> NvvN}
  \rule{N -> NvN}
  \rule{v -> vv}
}
\pgfdeclarelindenmayersystem{self-similar tree}{  
    \rule{A ->[-l] [+R]}
    \rule{l -> l [s+ [-l] [+R]]}
    \rule{L -> L [s+ [-l] [m] [+r]]}
    \rule{m -> m [s [-l] [+R]]}
    \rule{M -> M [s [-l] [m][+r]]}
    \rule{r -> r [s- [-l] [+R]]}
    \rule{R -> R [s- [-l] [m][+r]]}
            
    \symbol{L}{\pgflsystemdrawforward}  % draw the L branch
    \symbol{M}{\pgflsystemdrawforward}  % draw the M branch
    \symbol{R}{\pgflsystemdrawforward}  % draw the R branch
    \symbol{l}{\pgflsystemdrawforward}  % draw the L branch
    \symbol{m}{\pgflsystemdrawforward}  % draw the M branch
    \symbol{r}{\pgflsystemdrawforward}  % draw the R branch
    \symbol{s}{\pgflsystemstep =0.4 \pgflsystemstep}
}
\pgfdeclarelindenmayersystem{binary tree}{  
    \rule{T ->[-L] [+R]}
    \rule{L -> L [s+ [-L] [+R] ]}
    \rule{R -> R [s- [-L] [+R]]}
            
    \symbol{L}{\pgflsystemdrawforward}  % draw the L branch
    \symbol{M}{\pgflsystemdrawforward}  % draw the M branch
    \symbol{R}{\pgflsystemdrawforward}  % draw the R branch
    \symbol{l}{\pgflsystemdrawforward}  % draw the L branch
    \symbol{m}{\pgflsystemdrawforward}  % draw the M branch
    \symbol{r}{\pgflsystemdrawforward}  % draw the R branch
    \symbol{s}{\pgflsystemstep =0.4 \pgflsystemstep}
}

\theoremstyle{plain}
\newtheorem{thm}{Theorem}[section]
\newtheorem{hook}[thm]{Hooking Lemma}
\newtheorem{cor}[thm]{Corollary}
\newtheorem{lem}[thm]{Lemma}
\newtheorem*{lem*}{Lemma}
\newtheorem{prop}[thm]{Proposition}
\newtheorem*{quest*}{Open Question}

\theoremstyle{definition}
\newtheorem{defn}[thm]{Definition}
\newtheorem*{defn*}{Definition}
\newtheorem*{asp}{Assumption}
\theoremstyle{remark}
\newtheorem{oss}[thm]{Remark}
\newtheorem{exm}[thm]{Example}

\counterwithout{figure}{section}

\DeclareMathOperator{\dist}{d}
\DeclareMathOperator{\dG}{d_{\Gamma}}
\DeclareMathOperator{\dH}{d_{\mathit{h}}}
\DeclareMathOperator{\dF}{d_{\mathsf{F}}}
\DeclareMathOperator{\dB}{d_{\mathsf{B}}}
\DeclareMathOperator{\Haus}{d_{\mathcal{H}}}
\DeclareMathOperator{\diam}{diam}
\DeclareMathOperator{\T}{T^{*}}
\DeclareMathOperator{\Tend}{\mathcal{E}_{T}}
\DeclareMathOperator{\End}{\mathcal{E}}

\begin{document}
	\title{Rationality of the Gromov Boundary of Hyperbolic Groups }
	\author{Davide Perego}
	\date{}
 \thanks{The author is member of the Gruppo Nazionale per le Strutture Algebriche, Geometriche e le loro Applicazioni (GNSAGA) of the Istituto Nazionale di Alta Matematica (INdAM)}

\address{Dipartimento di Matematica e Applicazioni, Universit\`a degli Studi di Milano-Bicocca, Ed. U5, Via R.Cozzi 55, 20125 Milano, Italy, EU}
\email{\href{mailto:d.perego25@campus.unimib.it}{d.perego25@campus.unimib.it}}
	\maketitle

\begin{abstract}
In \cite{BBM}, Belk, Bleak and Matucci proved that hyperbolic groups can be seen as subgroups of the rational group. In order to do so, they associated a tree of atoms to each hyperbolic group. Not so many connections between this tree and the literature on hyperbolic groups were known. In this paper, we prove an atom-version of the fellow traveler property and exponential divergence, together with other similar results. These leads to several consequences: a bound from above of the topological dimension of the Gromov boundary, the definition of an augmented tree which is quasi-isometric to the Cayley graph and a synchronous recognizer which described the equivalence relation given by the quotient map defined from the end of the tree onto the Gromov boundary.
\end{abstract}

\section*{Introduction}
The Gromov boundary of a hyperbolic group is an object which has been widely studied in the past decades. Examples of this past and ongoing interest can be found in the survey \cite{KB}. A particular effort has been made to detect ``recursive'' presentations of such boundary: in their works \cite{CP,CP2}, Coornaert and Papadopoulos show how it can be seen as a \textit{semi-Markovian space} when the group is torsion free, while Pawlik \cite{P} provides a way to describe it as a \textit{Markov-compacta} and completes the work on semi-Markovian presentations in the general case, and Barrett \cite{B} gives an algorithm to determine if the boundary is a circle and investigates other topological properties. Also the well studied tool of \textit{subdivision rules} plays a role in this context, see e.g.\ \cite{R,R2}.\

The concept of \textit{rationality} that we follow can be found in the work \cite{GNS} of Grigorchuk, Nekrashevych and Sushchanski\v{\i}. The idea is to describe sets and hence relations, and functions, by using \textit{finite state machines}. One of the main goals is to define homeomorphisms of the \textit{Cantor set}  $\{0,1\}^{\omega}$ via asynchronous machines (one bit, i.e.\ $0$ or $1$, as input and a finite string written using $\{0,1\}$ as output at each step of the computation), these are \textit{rational functions}. On the other hand, synchronous machines, which for us have just inputs, at each step can read exactly one bit, are used to define \textit{rational sets} and \textit{rational relations}.\

In \cite{BBM} Belk, Bleak and Matucci associate a \textit{self-similar tree} called the \textit{tree of atoms} $\mathcal{A}(\Gamma)$ to any hyperbolic graph $\Gamma$, and they proceed to prove that the action of a hyperbolic group on the boundary of such a tree is rational, that is any element of the group can be regarded as a finite state machine that has a boundary point as input and its image according to the action as output. They also show that the boundary $\partial \mathcal{A}(\Gamma) $ projects onto the Gromov boundary $\partial \Gamma$ of $\Gamma$ (exploiting Webster and Winchester's work \cite{WW}). Since any self-similar tree defines a language, i.e.\ a subset of $\Sigma^{\omega}$ where $\Sigma$ is a finite set of symbols, and the language is also rational, then the projection induces a coding of any element of the Gromov boundary. Here we mean that to any boundary point we associate some (possibly more than one) elements of $\Sigma^{\omega}$. It is natural to ask whether the equivalence relation given by the projection is a rational relation.\\
In this paper, we tried to answer these and other questions about the relation between the tree of atoms and the Gromov boundary.\\
In order to fix the notations, we recall the main tools in metric geometry and geometric group theory we need in Section~\ref{1}. Section~\ref{2} contains the first original results of the paper, which regard the relation of atoms with cones and balls and in some cases are an improvement of what is pointed out in \cite{BBM}. Furthermore, we introduce \textit{tips of atoms} (Definition~\ref{tips}),
which turn out to be useful in our study. \

In Section~\ref{3} we show how infinite sequences of atoms behave like geodesic rays in hyperbolic graphs. The most useful result for the rest of the discussion is the following.\\

\textbf{Theorem~\ref{Conj1}}
\textit{Let $\Gamma$ be a hyperbolic graph and let $ \mathcal{A}(\Gamma)$ be its tree of atoms. Let $u= ( u_{k} )_{k=1}^{\infty}$ and $v= ( v_{k} )_{k=1}^{\infty}$ be two elements of $\partial \mathcal{A}(\Gamma)$. Then there exists a constant $C$ and a family of distances $\{\dG^{k}\}_{k=1}^{\infty}$ each defined on a level of the tree such that the sequences $u$ and $v$ are mapped in the same element of the Gromov boundary $\partial \Gamma$ if and only if $\dG^{k}(u_{k},v_{k} ) \leq C$ for all $k \geq 1$.}\\

Roughly speaking, this is an analog to the fellow traveler property of geodesic rays. In Section~\ref{4}, we prove an atom-version of the exponential divergence for geodesics (see Proposition~\ref{exp_prop_tdb}) and we define the Gromov product for atoms and for infinite sequences of atoms, providing an explicit relation between the latter and the Gromov product of elements of $\partial \Gamma$ (see Lemma~\ref{Gromov_atoms} and the discussion before it).
Moreover, we bound the fibers of the projection $\partial \mathcal{A}(\Gamma) \twoheadrightarrow \partial \Gamma$ (Theorem~\ref{finite_to_one})
and,  consequently, we provide another way to bound the topological dimension of the Gromov boundary using Theorem~\ref{Hurewicz}. Section~\ref{5} contains a generalization of Theorem~\ref{Conj1} (see Theorem~\ref{SConj1}) and a first approximation of the Gromov boundary using atoms in the sense of the \textit{weak Gromov-Hausdorff convergence} (see Definition~\ref{wGH}).\

One can construct, starting from the tree of atoms and the main results of Section~\ref{3}, the \textit{set of tips} $(\mathbf{T},\Haus)$,
where $\Haus$ is the Hausdorff metric, and the \textit{graph of atoms} $\Gamma_{\mathcal{A}}$ endowed with the standard metric on graphs (Definition~\ref{graph_of_atoms}). In particular, the graph is an \textit{augmented tree} in the sense of \cite{K}. In Section~\ref{6} we provide a quasi-isometry between the Cayley graph of a hyperbolic group and the set of tips (see Proposition~\ref{quasi_iso_haus} for both the definition of the set and the quasi-isometry). 
Furthermore,\\

\textbf{Theorem~\ref{quasi_isometry}}. \textit{Let $G$ be a hyperbolic group and let $\Gamma$ be its Cayley graph. Then the graph of atoms $\Gamma_{\mathcal{A}}$ and $\Gamma$ are quasi-isometric.}\\

In Section~\ref{7}, after briefly recalling language theoretic notions, we present the machine which describes the equivalence relation given by the projection $\partial \mathcal{A}(\Gamma) \twoheadrightarrow \partial \Gamma$. The language is based on the \textit{rigid structure} of the tree of atoms, which is a particular self-similar structure that assigns to each edge in the tree an element of $\Sigma$. The construction of the machine uses, again, results from Section~\ref{3}. The whole section can be summarized obtaining the following \\

\textbf{Theorem~\ref{rational_gluing_main_thm}.} 
\textit{The quotient map $\partial \mathcal{A}(\Gamma) \twoheadrightarrow \partial \Gamma$ defines a rational equivalence relation.}\\

We point out that being a semi-Markovian space implies the existence of such a map with such a property, generally the two notions do not coincide and our case seems to fail being semi-Markovian .\

Finally, in Section~\ref{8} we provide an example of a group with an Apollonian gasket as Gromov boundary. The first part investigates its graph of atoms from the point of view of approximation of the Gromov boundary via finite graphs. The second part is devoted to the partial description of its gluing machine. 

\section{Background} \label{1}

\subsection{Metric Geometry}
We fix some convenient notations and recall some useful facts about metric geometry.

\begin{defn}
Let $X$ be a set. A function $\dist: X \times X \rightarrow \mathbb{R} \cup \{ \infty \}$ such that $\dist(x,y)\geq 0$ or $\dist(x,y)=\infty$ is called \textbf{distance} if the following conditions hold:

\begin{itemize}
	\item[(a)] $\dist(x,x)=0$ for all $x \in X$;
	\item[(b)] $\dist(x,y)=\dist(y,x)$ for all $x,y \in X$.
\end{itemize}
\end{defn}

\noindent In some cases, we add further conditions
\begin{itemize}
	\item[($\overline{a}$)] if $\dist(x,y)=0$ with $x,y \in X$, then $x=y$;
	\item[(c)] $\dist(x,y)\leq \dist(x,z)+\dist(z,y)$ for all $x,y,z \in X$.
\end{itemize}

In particular, we call \textbf{semi-metric} a distance which satisfies $(\overline{a})$ and \textbf{pseudometric} a distance for which holds (c).

\begin{defn} \label{metric}
A \textbf{metric} is a semi-metric that is also a pseudometric.
\end{defn}

In metric geometry there are some techniques that allow a distance to become a metric changing the underlying set in a reasonable sense. We are interested in the following two operations (for more details see \cite[I.1.24]{BH} and \cite[Proposition 1.1.5]{BBI}).\

\textit{First Move.} \label{standard_pseudometric} Take a distance $\dist$ on a space $X$ and define $\underline{\dist}$ in the following way: consider all the possible finite sequences of elements that start from $x$ and end in $y$ and taking the minimum of the sums of the distances between two consecutive elements of the sequence. We get that $\underline{\dist}$ is a pseudo-metric. We do not modify the space $X$ in this case.\

\textit{Second Move.} Take a pseudo-metric $\dist$ on a space $X$ and consider the quotient $X/ \dist$ where elements are equivalence classes of the following relation: $x \sim_{\dist} y$ if and only if $\dist(x,y)=0$. This leads to a metric space, which is a quotient of $X$. \\

Given a metric space $(X,\dist)$, we set 
$$B_{k}(x)= \{ y \in X \mid \dist(x,y) \leq k \}$$
to be the \textbf{ball} of radius $k$ centered in $x$. For simplicity, we deal with pointed spaces $(X, x_{0})$ and hence $B_{k}$ will be the ball centered in $x_{0}$. Closed balls will be helpful as we will deal with graphs.\

Given a metric space $(X,\dist)$ it can be possible to turn it into a length space  defining a new metric $\widehat{\dist}$ that satisfies the condition above and such that the lengths of paths are defined using the original metric $\dist$.\

Assume now that $(X,\dist)$ is a length space and $Y$ is a subset of $X$. In general the metric space $(Y,\dist_{\mid_{Y}})$ is not a length space. But, as presented above, we can consider $(Y, \widehat{\dist}_{\mid_{Y}})$ and we call $\widehat{\dist}_{\mid_{Y}}$ the \textbf{intrinsic metric} on $Y$ with respect to $\dist$.\

\begin{oss}
\label{length}
Let $(X,\dist)$ be a length space and let $Y$ be a subset endowed with the intrinsic metric $\widehat{\dist}$. Then the following hold.
\begin{itemize}
\item[-]If $y_{1}, y_{2} \in Y$, then $\dist(y_{1},y_{2}) \leq \widehat{\dist}(y_{1},y_{2})$.
\item[-]Let $y_{1}, y_{2} \in Y$. Suppose that there exists a geodesic $[y_{1},y_{2}]_{X}$ (with respect to $\dist$) which is fully contained in $Y$. Then $\dist(y_{1},y_{2}) = \widehat{\dist}(y_{1},y_{2})$.
\end{itemize}
\end{oss}

\subsection{Graphs and Groups}
A \textbf{graph} $\Gamma=(V,E)$ is a simple undirected one which is locally finite and connected. By an abuse of notation we will often refer to $\Gamma$, and we will write $x \in \Gamma$, to mean the set of vertices.\

We endow a graph $\Gamma$ with the usual metric $\dG$ defined on vertices. We see the graph as a length space due to its quasi-isometric relation with its geometric realization $|\Gamma|$. So that a geodesic in $\Gamma$ is a sequence of vertices quasi-isometric to a geodesic in $|\Gamma|$.
Having this in mind, we can also define spheres (centered in a distinguished point $x_{0}$). Let $n \in \mathbb{N}$. Then the $n$-\textbf{sphere} $S_{n}$ is the collection of all vertices $x$ such that there exists a geodesic between $x$ and $x_{0}$ of length $n$. \\

We make the following hyphothesis on groups, so that their Cayley graphs satisfy our requirements on graphs.

\begin{asp}
A group is always a finitely generated group and a set of generators is always symmetric, namely if $s$ belongs to a set a generators $S$ then also $s^{-1} \in S$, and it does not contain $1_{G}$.
\end{asp}

\subsection{Hyperbolic Groups and their Boundary}
We recall that there are two definitions of hyperbolic graphs, one based on the thinness condition on triangles and the other based on the following 
\begin{equation}
\label{hypineq}
\tag{$\mathfrak{H}$}(x \mid z) \geq  \min \{ (x \mid y), (y \mid z) \} - \tilde{\delta}.
\end{equation}
We will use both of them, hence we denote with $(x \mid y)$ the Gromov product between $x$ and $y$ with respect to $x_{0}$, $\delta$ the constant that bounds triangle thinness and $\tilde{\delta}$ as appears above.\\

The following two results, that are considered folklore in the theory, are mentioned because they are helpful to our discussion: 
\begin{prop}
\label{generalized_fellow_traveler}
Let $\Gamma$ be a hyperbolic graph and let $[z,x]$ and $[z,y]$ be two geodesics such that $\dG(z,x)=t_{x}$ and $\dG(z,y)=t_{y}$. Put $t_{max}= \max \{ t_{x}, t_{y} \}$ and extended the shorter one to $[0,t_{max}]$ by the constant map. Then 
$$ \dG([z,x](t),[z,y](t)) \leq 2\dG(x,y)+4\delta$$ for all $0 \leq t \leq t_{max}$.
\end{prop}
For the proof see e.g. \cite[Lemma H.1.15]{BH}.

\begin{prop}[Exponential divergence]
\label{exp_div_geodesics}
Let $\Gamma$ be a hyperbolic graph. There exist three constants $E$, $E_{1}$ and $E_{2}$ such that for any two geodesics $[z,x]$ and $[z,y]$ and given $t,\tilde{t}$ such that $\tilde{t}+t \leq \min \{ \dG(z,x), \dG(z,y) \}$, if $\dG([z,x](\tilde{t}),[z,y](\tilde{t})) > E$ and $c$ is a rectifiable path fully contained in $\Gamma-B_{\tilde{t}+t}(z)$ from $[z,x](\tilde{t}+t)$ to $[z,y](\tilde{t}+t)$, then $\ell(c) > E_{1}e^{E_{2}t}$.\\
Moreover, $E_{1}$ and $E_{2}$ only depends on $\delta$. 
\end{prop}

We also recall that the Gromov boundary $\partial \Gamma$ of a hyperbolic graph can be seen as a quotient of both Gromov sequences $\{x_{k}\}_{k=1}^{\infty}$ and geodesic rays $\gamma:[0, \infty [ \rightarrow \Gamma$.

It is worth noticing that adding the boundary to the graph $\Gamma$, you get a compact space (the fact that $\Gamma \cup \partial \Gamma$ is a topological space will be stated later on). Actually every compact space can be seen as a boundary of a hyperbolic space. But since we are focused on groups, there are fewer possibilities. Indeed, the Gromov boundary of a non-elementary hyperbolic group is a compact metrizable space without isolated points. For all these facts see \cite[Section 2]{KB} to have further details. Since there are different versions of it in literature, we recall the definition of Gromov product on $\partial \Gamma$
\begin{defn}\label{Gromov-product-boundary}
Let $\Gamma$ be a hyperbolic graph and let $x_{\infty},y_{\infty} \in \partial \Gamma$. Then the Gromov product between $x_{\infty}$ and $y_{\infty}$ is
$$(x_{\infty} \mid y_{\infty}):= \inf \liminf_{k} (x_{k} \mid y_{k})$$
where the infimum is taken over all the Gromov sequences $\{x_{k}\}_{k=1}^{\infty}$ and $\{y_{k}\}_{k=1}^{\infty}$ that converge respectively to $x_{\infty}$ and $y_{\infty}$.
\end{defn} 
There are some considerations that are worth pointing out before continue. For a complete treatment on the argument, we recommend to see \cite[Section 5]{V}.
\begin{oss}\label{Vaisala_Gromov_product}\
\begin{enumerate}[label=(\alph*),ref=\ref{Vaisala_Gromov_product}(\alph*)]
\item In the same way, we can consider the product between an element $x_{\infty} \in \partial \Gamma$ and an element $y \in \Gamma$ by taking $y_{n}=y$, for every $n$. Note that $\liminf_{n} (x_{\infty} \mid y_{n})$ for some Gromov sequence $\{y_{n}\}_{n=1}^{\infty}$ is infinite if and only if the sequence converge to $x_{\infty}$.
\item The hyperbolic inequality (\ref{hypineq}) can be extended to $\Gamma \cup \partial \Gamma$.
\item \label{Vaisala_Gromov_product_c} To define the Gromov product on the boundary, one can consider also the $\inf \limsup$ and even $\sup \liminf$ or  $\sup \limsup$. We choose the smallest one, but they are all related. Indeed, they all lie in a $2\tilde{\delta}$-interval where $\tilde{\delta}$ is the hyperbolic constant involved in  (\ref{hypineq}) (see \cite[Definition 5.7]{V} for details).
\end{enumerate}
\end{oss}

To stress the connection between trees and hyperbolic spaces, we give this result that will be useful later. The proof can be found in Lemma 3.7 of \cite{GMS}.
\begin{lem}
\label{tripod}
Let $\Gamma$ be a hyperbolic graph with a distinguished point $x_{0}$. Let $\gamma$ and $\eta$ be two geodesic rays of $\Gamma$. Then there exists a quasi-isometry, with $L_{1}=1$ and $L_{2}=5\delta$ and $\delta$ the hyperbolic constant, between $\gamma \cup \eta$ and the tripod consisting of the rays glued together along an initial segment of length $( [\gamma] \mid [\eta])$ .
\end{lem}
And we recall the definition of metric on the Gromov boundary together with some properties and notations:
\begin{defn}\label{visual_metric}
Let $\beta > 1$ be a fixed constant. The \textbf{visual metric} on the completion $\Gamma \cup \partial \Gamma$ of a hyperbolic graph is defined as follow
$$\curlyvee(x_{*},y_{*}):= \beta^{-(x_{*} \mid y_{*})} \text { with } x_{*},y_{*}  \in \Gamma \cup \partial \Gamma,$$
with the convention that $\beta^{-\infty}=0$.
\end{defn}

\begin{lem}\label{visual_notation}
Let $\Gamma$ be a hyperbolic graph. Then
\begin{itemize}
\item[(a)]the function $\curlyvee$ defines a topology on $\Gamma \cup \partial \Gamma$;
\item[(b)]the function $\curlyvee$ is a semi-metric when restricted to $\partial \Gamma$;
\item[(c)]using the \textit{First Move} explained just after Definition \ref{metric} and with an abuse of notation, we get a new function that we call again $\curlyvee$ and that is a metric on $\partial \Gamma$;
\item[(d)]taking $\curlyvee$ as in the previous point, there exists a costant $\mathcal{B}$, depending only on $\delta$, such that for all $1<\beta\leq \mathcal{B}$ we have
$$ \dfrac{1}{2} \beta^{-(x_{*} \mid y_{*})} \leq \curlyvee(x_{*},y_{*}) \leq \beta^{-(x_{*} \mid y_{*})}.$$
\end{itemize}
\end{lem}
\begin{proof}
See Proposition 5.16 in \cite{V}.
\end{proof}
Please note that point (d) partially explain the ambiguity of point (c) and the abuse of notation in the Definition, indeed from now on we will assume $\beta$ will satisfy the constraint just introduced.\\

The last notations we need to introduce are about \textit{ends}.

\begin{defn}[Topological Ends]
Let $\Gamma$ be a graph with a distinguished point $x_{0}$. A sequence $\{ \mathcal{C}_{k} \}_{k=1}^{\infty}$ such that $\mathcal{C}_{k+1} \subseteq \mathcal{C}_{k}$ and $\mathcal{C}_{k}$ is a connected component of $\Gamma-B_{k-1}$ is called an \textbf{end} of $\Gamma$. We denote the collection of all ends with $\Tend(\Gamma)$.
\end{defn}

\begin{defn}[Graph Ends]
Let $\Gamma$ be a graph with a distinguished point $x_{0}$.We put an equivalence relation on geodesic rays in this way: two geodesic rays $\gamma_{1}$ and $\gamma_{2}$ are equivalent if $\gamma_{1}([k,\infty[)$ and $\gamma_{2}([k,\infty[)$ belong to the same connected component of $\Gamma - B_{k-1}$ for all $k$. Given a geodesic ray $\gamma$, we denote its equivalence class with $end(\gamma)$ and call this an \textbf{end} of $\Gamma$ and denote the collection of all ends with $\End(\Gamma)$.
\end{defn}

In our case, i.e.\ hyperbolic graphs, it is not difficult to see that the second definition induces a surjective map $\partial \Gamma \twoheadrightarrow \End(\Gamma)$  such that $[\gamma] \mapsto end(\gamma)$. This map gives a correspondence between graph ends and connected components of the Gromov boundary. More in general, the two definitions are equivalent, in the sense that there exists a bijection between $\End(\Gamma)$ and $\Tend(\Gamma)$ whenever $\Gamma$ is a graph (see \cite{DK}). \\

\subsection{Tree associated to Hyperbolic Graphs} \label{Trees} In this subsection we provide an overview of the \textit{tree of atoms}. It is a way to associate a rooted tree to $\Gamma$, together with a quotient map from its boundary onto the Gromov boundary, first introduced in \cite{BBM}. In literature, there are more canonical way to do so (see e.g. \cite{CP}).\\

We start by considering the abelian group  $\mathcal{F}=\mathbb{Z}^{\Gamma}$ of functions from  the vertices of $\Gamma$ (which we will identify with $\Gamma$) to $\mathbb{Z}$. This space is endowed with the product topology (which is also the compact-open topology, since the set of vertices of $\Gamma$ is discrete). In particular, if we take the quotient space $\overline{\mathcal{F}}$ of $\mathcal{F}$ over the subgroup of constant functions, it inherits the quotient topology.\\
We consider the function $d_{x}: \Gamma \rightarrow \mathbb{Z}$ defined as $d_{x}(-):= \dG(x,-)$ with $x \in \Gamma$. The class $\overline{d}_{x}$ in $\mathcal{F}$ has a global minimum in $x$. Hence there is a canonical embedding $\iota: \Gamma \rightarrow  \overline{\mathcal{F}}$  since $\overline{d}_{x}$ is an isolated point in $\Gamma \iota$. 

\begin{defn}
Let $\Gamma$ be a hyperbolic graph. The \textbf{horofunction boundary} $\partial_{h} \Gamma$ of $\Gamma$ is the set of limit points of $\Gamma \iota$.
\end{defn}
\noindent
An element of $\partial_{h} \Gamma$ is called \textbf{horofunction}. Using the definition, we say that a sequence of vertices $\{ x_{k} \}_{k=1}^{\infty}$ converges to a horofunction $u$ if and only if the sequence $\{ \overline{d}_{x_{k}} \}_{k=1}^{\infty}$ converges to $u$ (in the compact-open sense). Horofunctions were initially introduced by Gromov in \cite{BGS}, he refers to them as the ``metric boundary'' and since then, they were widely studied. For a different, but related notion see Definition~\ref{CP_function}.\\
Among horofunctions there are also the so-called \textbf{Busemann points}. Namely, a horofunction is a Busemann point if it is the limit (in the sense explained above) of a geodesic ray. This notion is a first step towards the next result, but please note that in the general hyperbolic case there are horofunctions that are not Busemann points (see e.g. \cite{WW2}).\

This boundary can always be constructed, instead of the Gromov one, which requires hyperbolicity. Moreover, when the two are well-defined, we can retrieve $\partial \Gamma$ starting from horofunctions. Indeed, we have the following 
\begin{prop}
Let $\Gamma$ be a hyperbolic graph. Then there exists a continuous surjective map $\pi_{h}: \partial_{h} \Gamma \twoheadrightarrow \partial \Gamma$.
\end{prop}
\begin{proof}
See \cite[Section 4]{WW}.
\end{proof}
\noindent
In particular, if a sequence converges to a point in $\partial_{h} \Gamma$ then it goes to infinity in the sense of Gromov.\

There is another way to represent horofunctions, it is the so called \textbf{tree of atoms} (\cite[Definition 3.4]{BBM}). The idea is to construct a suitable collection of partitions of $\Gamma$ (seen as the set of vertices) and then to endow it with a tree structure.\\
Let $x$ be an element of a hyperbolic graph $\Gamma$. We consider the function
\begin{equation}
\label{horo_representative}
\tag{$\mathfrak{F}$}
f_{x}(-):=\dG(-,x)-\dG(x_{0},x): \Gamma \rightarrow \mathbb{Z}
\end{equation}
for all $x \in \Gamma$.\\
Now we fix $k$ to be a non-negative integer. The $k$-partition comes from the following equivalence relation: two vertices $x$ and $y$ are equivalent if and only if $f_{x}$ and $f_{y}$ agree on the ball $B_{k}$ of radius $k$ centered in $x_{0}$ (this means that $\overline{d}_{x}=\overline{d}_{y}$). We call the equivalence classes that contain an infinite number of vertices $k$\textbf{-level atoms} and we denote the collection of such classes with $\mathcal{A}_{k}(\Gamma)$. When $\Gamma$ will be clear, we will drop it in the notation.\\
It can be shown that each partition is finite and it is a refinement of the previous one. Indeed, given $x \in \Gamma$ and $k \in \mathbb{N}$, there are only finitely many possibilities for the restriction of $f_{x}$ to $B_{k}$. Moreover, since we are dealing with restriction, if $f_{x}$ and $f_{y}$ agree on $B_{k+1}$, then they agree on $B_{k}$.  In particular, if we consider atoms, they have a structure of an infinite tree $$\mathcal{A}(\Gamma):= \coprod_{k=1}^{ \infty} \mathcal{A}_{k}(\Gamma).$$
For further details about this construction see \cite[Subsection 3.1]{BBM}.\

\begin{exm}
Consider the $1$-skeleton of the hyperbolic tiling depicted in Figure \ref{atoms_tiling}.(\subref{atoms_tiling_A}). The first level consists of ten atoms: in Figure \ref{atoms_tiling} the subdivision given by the red lines gives all ten of them and a finite region in which $x_{0}$ is the only element. The second level can be described as follows: every $1$-level atom has three children, given by the intersection between the atom and the brown lines.
\end{exm}

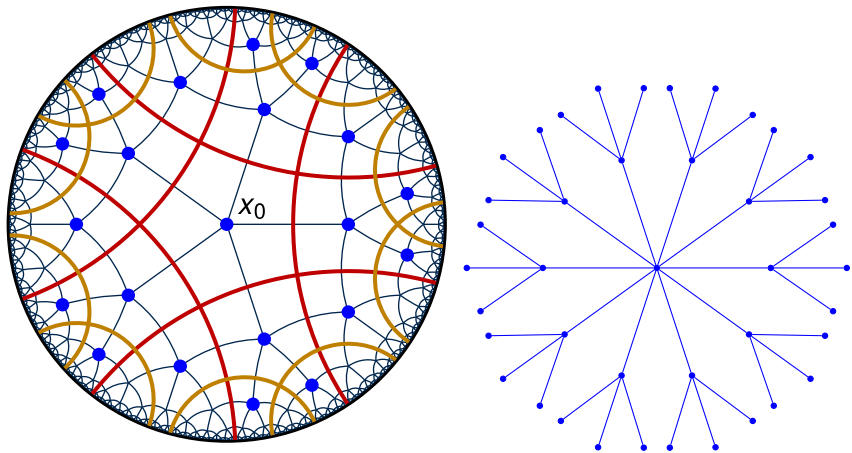
\begin{figure}
    \centering
\begin{subfigure}[t]{0.45\textwidth}
\includegraphics{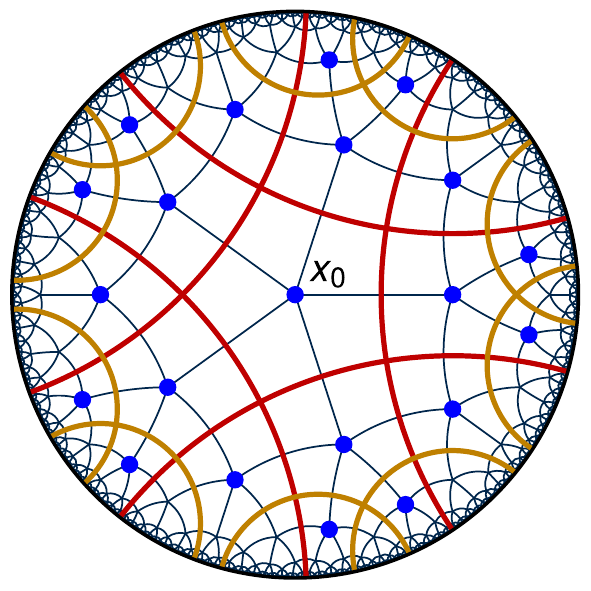}
\caption{The atoms seen as subsets in the graph together with the ball of radius 2 (in blue).}
\label{atoms_tiling_A}
\end{subfigure}
\begin{subfigure}[t]{0.45\textwidth}
\centering
\begin{tikzpicture}[
  grow cyclic,
  rotate=18,
  level distance=1.5cm,
  level/.style={
    level distance/.expanded=\ifnum#1>1 \tikzleveldistance/1.5\else\tikzleveldistance\fi,
    nodes/.expanded={fill}
  },
  level 1/.style={sibling angle=36},
  level 2/.style={sibling angle=35},
  level 3/.style={sibling angle=24},
  level 4/.style={sibling angle=36},
  nodes={circle,draw,inner sep=+0pt, minimum size=2pt},
  ]
\path[rotate=0, blue]
  node[fill] {}
  child foreach \Ilevel in {1,...,10} {
    node {}
    child foreach \IIlevel in {1,...,3} { 
      node {}
      }
    };
\end{tikzpicture}
\caption{The first two levels of the tree.}
\end{subfigure}
    \caption[Example of atoms]{The first two levels of the atoms for the $1$-skeleton of the uniform tiling of the hyperbolic plane given by five squares meeting in each vertex.}
    \label{atoms_tiling}
\end{figure}

A first property which says something about the asymptotic behavior of atoms is the following
\begin{prop}[\cite{BBM}{, Proposition 3.5}]
\label{complement-atom}
Every $k$-level atom is contained in $\Gamma-B_{k-1}$.
\end{prop}
But the key aspect of this structure is the following
\begin{thm}[\cite{BBM}{, Theorem 3.6}]\label{horofunction-boundary-tree-atoms}
Let $\Gamma$ be a hyperbolic graph. Then the boundary of $\mathcal{A}(\Gamma)$ is homeomorphic to $\partial_{h} \Gamma$.
\end{thm}

This means that we can represent horofunctions via infinite nested sequences of atoms, namely if $u$ is an element of $\partial_{h} \Gamma$, then there exists a unique nested sequence $(u_{k})_{k=1}^{\infty}$ such that $u_{k}$ is a $k$-level atom and the horofunction corresponds to the sequence by virtue of the homeomorphism. In symbols we will simply write $u=(u_{k})_{k=1}^{\infty}$. We will frequently refer to this representation as the \textbf{atom-coding of the horofunction}.
\begin{oss}
Let $u=(u_{k})_{k=1}^{\infty}$ be a horofunction. If we denote with $f_{u}$ the representative of $u$ such that $f_{u}(x_{0})=0$. Then $f_{u_{\mid_{B_{k}}}}=f_{u_{k}}$ with $f_{u_{k}}$ defined as the restriction on $B_{k}$ of $f_{x}$ for all $x \in u_{k}$, where $f_{x}$ is the function introduced in (\ref{horo_representative}).
\end{oss}

\section{Old and New Results about Atoms} \label{2}
This section is devoted to describing some further properties of atoms. In particular we will present a first relation with the Gromov boundary and others with cones and balls which will be useful to establish a full connection with $\partial \Gamma$. We will also defined Gromov products on atoms and horofunctions in a different (and improper, but very useful) way. Unless specified, the results contained in this section are to be considered new.\

Before starting, we clarify a notation that will be used from now on. If $x \in \Gamma$ and $B \subseteq \Gamma$, then with $\dG(x,B)$ we mean the mininum over all elements of $B$ of their distances with respect to $\dG$ from $x$.\\
Also, in this section we deal with two different definitions of \textit{type} in two different situations. This occurs because usually in the literature the collection of types refers to a partition of the vertices. 

We begin defining some collections of points.
\begin{defn}
Let $\Gamma$ be a hyperbolic graph with a distinguished point $x_{0}$ and let $x$ an element in $\Gamma-B_{k}$ for some $k \in \mathbb{N}$. 
\begin{itemize}
\item[(a)]A \textbf{nearest neighbor} for $x$ is a vertex $\overline{x}$ in $B_{k}$ such that $\dG(B_{k},x)=\dG(\overline{x},x)$.
\item[(b)]A point $p \in B_{k}$ is \textbf{visible} for $x$ if $[p,x] \cap B_{k}=\{p\}$  for every geodesic $[p,x]$ from $p$ to $x$.
\item[(c)]A point $p \in S_{n}$ with $n \leq k$ is said to be $n$-proximal (or simply \textbf{proximal} when the $n$ is clear) to $x$ if there exists a geodesic $[x_{0},p]$ such that for every geodesic $[x_{0},x]$ we have
$$\dG(p_{i},x_{i}) \leq 4\delta+2 \text{ for } 1\leq i \leq n$$
with $p_{i}$ the $i$-vertex of $[x_{0},p]$ and $x_{i}$ the $i$-vertex of $[x_{0},x]$.
\end{itemize}

\end{defn}
\noindent We denote with $N(x,B_{k})$ the collection of nearest neighbors in $B_{k}$ of $x$; and with $V(x,B_{k})$ and $P(x,S_{n})$ the collections of visible points and proximal points respectively. 
\begin{lem}\label{visible_proximal}
Let $\Gamma$ be a hyperbolic graph and let $x\in a$ for some $k$-level atom $a$. Then the following properties hold.
\begin{enumerate}[label=(\alph*),ref=\ref{visible_proximal}.\alph*]
\item \label{visible_proximal_a} Every nearest neighbor of $x$ is visible and every visible point is proximal, in short $N(x,B_{k}) \subseteq V(x,B_{k}) \subseteq P(x,S_{k})$. 
\item \label{visible_proximal_b} A point $p$ is $k$-proximal if and only if there exists a $(k-1)$-proximal point at distance $1$ from $p$ and $\dG(p,q) \leq 4\delta+2$ for all $q \in V(x,B_{k})$. In particular, the diameter of $P(x,S_{k})$ is bounded by $8\delta+4$.
\item \label{visible_proximal_c} If $p$ is a nearest neighbor in $B_{k}$ for $x \in a$, then it is a nearest neighbor for all elements in $a$. The same statement is true for $V(x,B_{k})$ and $P(x,S_{k})$.
\end{enumerate}
\end{lem}
\noindent An immediate consequence is that $N(a,B_{k})$, $V(a,B_{k})$ and $P(a,S_{k})$ are well-defined.\

\begin{proof}
The proofs of all these statements can be found in \cite{BBM}: the first part of (a) is straightforward, for the second see Proposition 3.19, for (b) see Proposition 3.21 and Corollary 3.22, while for (c) we need to put together Proposition 3.15 and Corollary 3.24.
\end{proof}

We note that if $\overline{x}$ belongs to $N(x,B_{k})$ for some $x \in \Gamma-B_{k}$, then there exists a geodesic $[x_{0},x]$ via $\overline{x}$, i.e. $\overline{x} \in [x_{0},x]$. In fact, this can be seen as an alternative definition. \\

Now we briefly recall \textit{cones} and \textit{cone types}, as we explicitly need them in our proofs.
\begin{defn}
Let $\Gamma$ be a hyperbolic graph with a distinguished point $x_{0}$. If $p$ is a vertex of $\Gamma$, then we define its \textbf{cone} to be
$$C(p):=\{ x \in \Gamma \mid \dG(x_{0},x)=\dG(x_{0},p)+\dG(p,x)\}.$$
\end{defn}
\noindent Equivalently, a point $x$ is in the cone of $p$ if and only if there is a geodesic $[x_{0},x]$ such that $p \in [x_{0},x]$.\\

For the case $\Gamma=\Gamma(G,S)$ and $x_{0}=id$ with $G$ some hyperbolic group, we can define the so-called \textit{cone types}.

\begin{defn}
Let $G$ be a hyperbolic group and let $S$ one of its generating set. If $g \in G$ then its \textbf{cone type} is the collection
$$\{  h\in G \mid  \ell(hg)=\ell(h)+\ell(g) \}$$
where $g$ is intended as a geodesic between $id$ and $g$ and $\ell$ is the length.
\end{defn}
In particular, if $g_{1}$ and $g_{2}$ have the same cone type then the map $g_{1}^{-1}g_{2}$ is an isometry of cones.\\

We now recall a fundamental result about cone types in hyperbolic groups, that is due to Cannon~\cite{C}

\begin{prop} \label{cone_type_finite}
Let $G$ be a hyperbolic group. Then the number of cone types of its Cayley graph is finite. 
\end{prop}
\begin{proof}
See e.g. Proposition 7.5.4 in \cite{L}.
\end{proof}

We now come back to our purpose, that is to link atoms with the Gromov boundary and so to the Gromov product. A property which will describe the behavior of the product over an atom is the following
\begin{oss}
If $p \in S_{k}$, i.e. $\dG(x_{0},p)=k$, then the Gromov product restricted to its cone is more than or equal to $k$. Indeed, taken $x,y \in C(p)$ we have
$$(x \mid y)= \dfrac{1}{2}\left[ k+\dG(p,x)+k+\dG(p,y) -\dG(x,y) \right]$$
and a triangular inequality yields the claim.
\end{oss}

One can see that there is a weak connection between cones and atoms, namely
\begin{oss}
Let $x$ be an element of $S_{k}$. Then $C(x)$ is a finite disjoint union of suitable $k$-atoms. Indeed, we take $y \in C(x)$ and we denote $a_{y}$ the $k$-atom that contains $y$. Then $C(x) \subseteq \cup_{y \in C(x)} a_{y}$. Since $x$ is a nearest neighbor for $y$, then $x$ is a nearest neighbor for every $z \in a_{y}$ too. Hence, for all $z \in a_{y}$, we have $z \in C(x)$. So $ \cup_{y \in C(x)} a_{y} \subseteq C(x)$.
\end{oss}
\begin{prop}
\label{cone_intersect}
Let $\Gamma$ be a hyperbolic graph with a distinguished point $x_{0}$ and let $a$ be a $k$-level atom. Then
$$a \subseteq   \displaystyle \bigcap_{p \in N(a, B_{k})} C(p) - \bigcup_{q \in S_{k}-N(a,B_{k})} C(q).$$
\end{prop}
\begin{proof}
The fact that $a$ lies in the intersection of its nearest neighbors cones follows immediately from the definition.\

Now suppose that $x \in a$ and $x \in C(q)$ with $q \in S_{k}-N(a,B_{k})$. This means that there exists $p \in B_{k}$ such that $\dG(q,x) > \dG(p,x)$. Hence $\dG(x_{0},x)=\dG(x_{0},q)+\dG(q,x) \geq \dG(x_{0},p)+\dG(q,x) > \dG(x_{0},p)+\dG(p,x) \geq \dG(x_{0},x) $, and this proves the claim. 
\end{proof}

The goal of the Proposition should be to fully characterize atoms in geometric terms (i.e. via cones). Even though this question is still open, we can bind atoms with the so called $N$\textit{-types} introduced by Cannon (see \cite{C} and \cite{CP} for other applications). If $\Gamma$ is the Cayley graph of an hyperbolic group, we say that two element $x$ and $y$ \textbf{have the same $N$-type} if $\dG(xz,id)-\dG(x,id)=\dG(yz,id)-\dG(y,id)$ for all $z \in B_{N}(id)$. In fact, they are a useful tool to prove Proposition \ref{cone_type_finite}. 
\begin{oss}
Let $\Gamma$ be the Cayley graph of a hyperbolic group. Then $x,y \in \Gamma$ belong to the same $N$-level atom if and only if $x^{-1}$ and $y^{-1}$ are of the same $N$-type. This is straightforward once we notice that $\dG(x,id)=\dG(x^{-1},id)$ and $\dG(xz,id)=\dG(z,x^{-1})$.
\end{oss}
The problem is that, again, the connection between $N$-types and cones is yet to be fully understood.\\

A topic on which we can say something is the topology of atoms in $\partial \Gamma$. More precisely, exploiting the correspondence given by $\pi_{h}$, we can define the \textbf{shadows} of a $k$-level atom $a$ as
$$\partial a:= \{u \in \partial_{h} \Gamma \mid  f_{u}=f_{a_{k}} \text{ on } B_{k}\}.$$. We know that they are closed subset of $\partial_{h} \Gamma$ and since $\pi_{h}$ is a closed map (it is continuous from a compact space to a metrizable space) we get that $\partial a \pi_{h}$ is closed in $\partial \Gamma$.\\
In this context, the diameter of $\partial a \pi_{h}$ with respect to the visual metric is bounded above by $\beta^{-k}$ with $a \in \mathcal{A}_{k}$.\\

We return for a moment to our (otherwise implicit) group action of $G$ on $\Gamma$. In particular, we use it to induce a notion of morphisms between subtrees of $\mathcal{A}(\Gamma)$.
\begin{defn}
Let $G$ be a group that acts geometrically on a hyperbolic graph $\Gamma$ with a distinguished point $x_{0}$. Let $a_{n} \in \mathcal{A}_{n}$ and $a_{m} \in \mathcal{A}_{m}$. We say that an element $g \in G$ \textbf{induces a morphism} between $a_{n}$ and $a_{m}$ if 
\begin{itemize}
\item[-]$a_{n}g=a_{m}$,
\item[-]$(a_{n} \cap B_{n+k})g= a_{m} \cap B_{m+k}$ for all $k \geq 0$,
\item[-]for each $k>0$ and each atom $\tilde{a}_{n+k} \in \mathcal{A}_{n+k}$ contained in $a_{n}$, there exists an atom $\tilde{a}_{m+k} \in \mathcal{A}_{n+k} $ contained in $a_{m}$ such that $\tilde{a}_{n+k}g=\tilde{a}_{m+k}$.
\end{itemize}
\end{defn}
\noindent The second condition can be restated in this terms
\begin{equation}
\label{hypinrat3.8.1}
\tag{$\mathfrak{L}$}
\dG(x_{0},gx)-m=\dG(x_{0},x)-n \ \ \forall x \in a_{n}.
\end{equation}
\noindent
The third condition fits into the context of subtrees of $\mathcal{A}(\Gamma)$. In fact, we can see the atoms that are contained in a fixed one $a$ as a subtree $\mathcal{A}(\Gamma)_{a}$ rooted in $a$ and the condition admits the existence of an isomorphism between two such subtrees given by $g$. With this in mind, we give the following
\begin{defn}
Let $G$ be a group that acts geometrically on a hyperbolic graph $\Gamma$ with a distinguished point $x_{0}$. Two atoms of $\Gamma$ are of the same \textbf{type} if there exists an  isomorphism (given by an element of $G$) between them.
\end{defn}

As for cones, the following important property holds (see \cite{BBM} for a complete proof and Section \ref{7} for further details). 
\begin{thm}
\label{finite-types}
Let $G$ be a group that acts geometrically on a hyperbolic graph $\Gamma$ with a distinguished point $x_{0}$. Then the number of different types of atoms in $\mathcal{A}(\Gamma)$ is finite.
\end{thm}

Using these considerations, we get a tool, useful in the next section, about atoms and their balls:
\begin{hook}
\label{Conj3}
There exists a constant $\lambda_{a}$ that bounds the distances between
atoms and their balls.
\end{hook}
The symbol $\lambda_{a}$ will be used from now on to denote the \textbf{hooking constant}.
\begin{proof}
The Condition (\ref{hypinrat3.8.1}) says that the distance between an atom and its ball depends only on the type of the atom. Since the graph is hyperbolic we know it has a finite number of types (by Lemma \ref{finite-types}). This two facts combined together allow us to consider the maximum over all the types of the distances and to get the constant.
\end{proof}
\begin{quest*}
Is there a way to express $\lambda_{a}$ with respect to $\delta$?
\end{quest*}

The last collection of vertices we define are tips.

\begin{defn}\label{tips}
We call the \textbf{tip} of an atom $a \in \mathcal{A}_{k}$ the collection $T(a):= \{ x \in a \mid \dG(B_{k},x)=\dG(B_{k},a)\}$, or, equivalently, the first non-empty intersection $a \cap S_{k+i}$ with $i \geq 0$.
\end{defn}

\begin{prop}
\label{diam}
Let $a$ be an atom of $\Gamma$. Then $$\diam T(a):= \max_{x,y \in T(a)} \dG(x,y)$$ is bounded by $2\lambda_{a}$.
\end{prop}
\begin{proof}
It suffices to construct a triangle made by geodesics with two elements of $T(a)$ and a nearest neighbor of $a$ as vertices, namely if $\widehat{x},\widehat{y} \in T(a)$ and $\overline{x} \in N(a,B_{k})$ then $\dG(\widehat{x},\widehat{y}) \leq \dG(\widehat{x},\overline{x})+\dG(\overline{x},\widehat{y})$.
The claim follows by the \hyperref[Conj3]{Hooking Lemma}.
\end{proof}

To understand how small a tip could be, we give the following
\begin{lem}
Let $a_{k} \in \mathcal{A}_{k}$. Suppose $B_{k} \cap a_{k} \neq \emptyset$, then $B_{k} \cap a_{k}$ consists in one point. Which means that $T(a_{k})=N(a_{k},B_{k})$ consists in one point.  
\end{lem}
\begin{proof}
The fact that $B_{k} \cap a_{k}=T(a_{k})=N(a_{k},B_{k})$ is straightforward.\\
It suffices to show that $T(a_{k})$ is a point: taken $x$ and $y$ that belong to $T(a_{k})$, and hence to $S_{k}$, by definition of atom we get
$$\dG(x,y)-k=\dG(x,x)-k=-k$$
that is $\dG(x,y)=0$ and the claim holds.
\end{proof}

\begin{exm}
We consider again the uniform tiling of the hyperbolic plane made of squares such that each vertex has degree 5. Recall that we have ten $1$-level atoms. They divides equally in two different types. One type has the property described above. Indeed, each element of the sphere $S_{1}$ is the tip of one of the five atoms. See Figure \ref{4-5_tiling} for clearence.
\begin{figure}\centering

\includegraphics[]{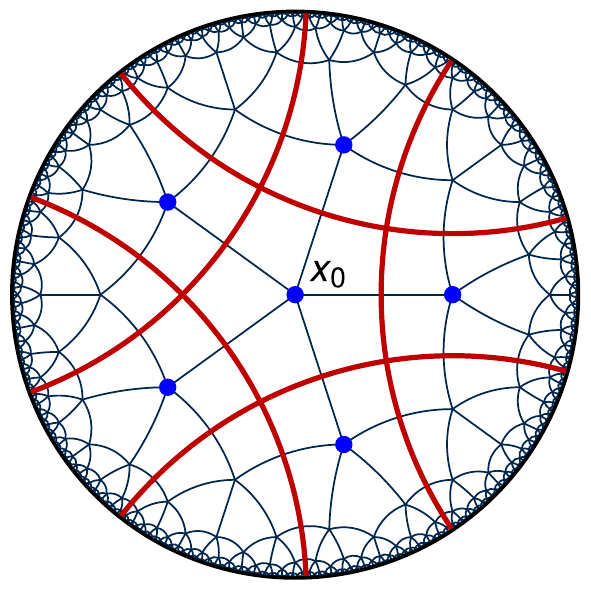}
\caption[Example of sharp tips]{In blue the ball of radius 1. Each element of the sphere coincides with the tip of a $1$-level atom.}
\label{4-5_tiling}
\end{figure}
\end{exm}

\begin{oss} \label{equal_tips}
It is worth pointing out that two different atoms $a$ and $b$ such that $b$ is a children of $a$ can have the same tip. This could happen in one of the following situations
\begin{itemize}
    \item[--] the atom $b$ is the only child of $a$ and it is isometric to it. For an example see narrow type atoms in Section~\ref{8}.
    \item[--] the atom $a$ splits in different children, but one of them has the same tip. See type $C$ and $D$ atoms in Example~\ref{type_automaton_2D}.
\end{itemize}
\end{oss}

A first application of tips is that the two projections $\pi_{h}$ and $\partial \Gamma \twoheadrightarrow \Tend(\Gamma)$ are compatible.
\begin{oss}
We induce the map $\widehat{\iota}: \partial_{h} \Gamma \twoheadrightarrow \Tend(\Gamma)$ from the family of surjective maps $\iota_{k}$ from $\mathcal{A}_{k}$ to  the collection of connected components of $\Gamma - B_{k-1}$ defined in the following way: $a_{k} \mapsto \mathcal{C}_{k}$ if $a_{k} \subseteq \mathcal{C}_{k}$. Note that any $k$-level atom $a$ is fully contained in a connected component of $\Gamma-B_{k-1}$ since, if $x,y \in a $, in order to get a path between the two of them, one can consider a geodesic going from $x$ to a nearest neighbor $\overline{x}$ and then a geodesic from $\overline{x}$ to $y$. 
Now, taking a horofunction $u=(u_{k})_{k=1}^{\infty}$ its image is $\mathcal{C}_{u}=(\mathcal{C}_{k}^{u})_{k=1}^{\infty}$ with $u_{k} \subseteq \mathcal{C}_{k}^{u}$ and $u_{k+1} \subseteq u_{k} \subseteq \mathcal{C}_{k}^{u}$, so $\mathcal{C}_{k+1}^{u}\subseteq \mathcal{C}_{k}^{u}$. Moreover, the diagram below commutes.

\[
\begin{tikzcd}[ampersand replacement=\&]
\partial_{h} \Gamma \arrow[d, "\pi_{h}",twoheadrightarrow] \arrow[dr, "\widehat{\iota}", twoheadrightarrow] \\
\partial \Gamma \arrow[r, twoheadrightarrow]
\&  \Tend(\Gamma)
\end{tikzcd}
\]

To prove this, let $[\gamma_{u}]$ be a point in the Gromov boundary such that $u\pi_{h}=[\gamma_{u}]$. Let $\mathcal{C}_{\gamma} \in \End(\Gamma)$ such that $end(\gamma)=\mathcal{C}_{\gamma}$ and $\mathcal{C}_{\gamma}$ corresponds to some sequence $(\mathcal{C}_{k})_{k=1}^{\infty} \in \Tend(\Gamma)$. We want $\mathcal{C}_{\gamma} \simeq \mathcal{C}_{u}$.\\
Now $\mathcal{C}_{n}$ is such that $\gamma([0, \infty[)\subseteq \mathcal{C}_{n}$ and there exists a sequence $\{\widehat{x}_{k}\}_{k=1}^{\infty}$ in $\Gamma$ such that $\widehat{x}_{k} \in T(u_{k})$, hence going to infinity in the sense of Gromov, that converges to $u$ and such that $(\widehat{x}_{k} \mid \gamma(k)) \rightarrow \infty$.So $\widehat{x}_{k}$ and $\gamma(k)$ need to be in the same connected component. Otherwise 
$$\dG(\widehat{x}_{k},x_{0})+\dG(\gamma(k),x_{0})-\dG(\widehat{x}_{k},\gamma(k)) \leq \lambda_{a}+2\overline{k} \ \ \ \ \ \ \  \forall k > \overline{k}.$$
Indeed, $\dG(\widehat{x}_{k},x_{0}) \leq k+\lambda_{a}$ for the \hyperref[Conj3]{Hooking Lemma}, $\dG(\gamma(k),x_{0})=k$ since $\gamma$ is a geodesic ray and $\dG(\widehat{x}_{k},\gamma(k)) \geq 2k-2\overline{k}$ where $\overline{k} \in \mathbb{N}$ is such that $\mathcal{C}_{k}$ is the first component not containing $\widehat{x}_{\overline{k}}$.
It follows that $\mathcal{C}_{n}=\mathcal{C}_{n}^{u}$ and hence the claim. 
\end{oss}

The conclusion of this section is devoted to describing a different Gromov product for atoms. The defintion of Gromov product we gave can be applied on $\mathcal{A}$ as a tree, we will denote it with $( \cdot \mid \cdot )_{\mathcal{A}}$, it will be useful in the next sections and it has an associated visual metric $\curlyvee_{\mathcal{A}}$. What we want to do here, it is to define a Gromov product between atoms of the same level and one between horofunctions (recall Theorem~\ref{horofunction-boundary-tree-atoms}  and the atom-coding) that keeps track of what is happening in the underlying hyperbolic graph. Namely, we put

$$(a_{k} \mid b_{k})_{k}:=\max_{\widehat{x} \in T(a_{k}),\widehat{y} \in T(b_{k})}(\widehat{x} \mid \widehat{y}) \text{ and } (u \mid v):= \liminf_{k} (u_{k} \mid v_{k})_{k}$$

\noindent We will always drop the $k$ in the notation, as it will be clear from the context.\

We need to check that this notion somehow agree with the one regarding points of the Gromov boundary.

\begin{lem}\label{Gromov_atoms}
Let $u$ and $v$ be two horofunctions of $\Gamma$. Then
$$(u\pi_{h} \mid v\pi_{h}) \leq (u \mid v) \leq (u\pi_{h} \mid v\pi_{h})+2\tilde{\delta}.$$
\end{lem}
\begin{proof}
By Definition~\ref{Gromov-product-boundary}, we have $(u\pi_{h} \mid v\pi_{h}) \leq (u \mid v)$. For the second inequality, we start by saying that
$$(u \mid v) \leq \max_{x_{k} \in T(u_{k}), y_{k} \in T(v_{k})} \liminf_{k } (x_{k} | y_{k}),$$
here we mean that the maximum has to be taken over all the possibile couples of sequences of vertices such that $x_{k} \in T(u_{k})$ and $y_{k} \in T(v_{k})$.\\
Now 
$$ \max_{x_{k} \in T(u_{k}), y_{k} \in T(v_{k})} \liminf_{k } (x_{k} | y_{k}) \leq \sup_{\{x_{k}\} \in u\pi_{h} , \{y_{k}\} \in v \pi_{h}} \liminf_{k } (x_{k} | y_{k}),$$ 
hence combining the two inequalities and using Remark \ref{Vaisala_Gromov_product_c}, we get $(u \mid v) \leq (u\pi_{h} \mid v\pi_{h}) +2\tilde{\delta}$.
\end{proof}

\section{Gluing Relation via Atoms}\label{3}
Now that we developed the two points of view for the horofunction boundary of a hyperbolic graph $\Gamma$ and some tools regarding the tree of atoms, we want to find a connection between them to understand the gluing relation given by the quotient map $\pi_{h}: \partial_{h} \Gamma \twoheadrightarrow \partial \Gamma$. More precisely, the goal is a way to determine when two horofunctions are glued by looking at the tree of atoms. This leads to a coarse version of the so called fellow traveler property.\

In order to do that, we introduce the following 
\begin{defn}
Let $\Gamma$ be a hyperbolic graph and let $\mathcal{A}(\Gamma)$ be its tree of atoms. For all $k$, we choose $\dist^{k}$ a semi-metric defined on the $k$-level of $\mathcal{A}(\Gamma)$. We say that $\dist^{k}$ \textbf{represents the gluing} if there exists a constant $C$ such that given two horofunctions $u= ( u_{k} )_{k=1}^{\infty}$ and $v= ( v_{k} )_{k=1}^{\infty}$, then 
$$ u \pi_{h}= v \pi_{h} \ \ \ \Leftrightarrow \ \ \ \forall k \geq 0 \ \dist^{k}(u_{k},v_{k}) \leq C.$$ 
\end{defn}

\noindent In this way, we see that, under such a family of semi-metrics, the horofunctions which glue have a similar, or maybe we should say generalized, behavior of two geodesics that glue on $\partial \Gamma$.\

We are going to introduce three semi-metrics and to prove that they represent the gluing.\

The first one we consider comes from a natural way to think about distances between atoms. Let $a$ and $b$ two $k$-level atoms, then 
$$\dG^{k}(a,b):=  \min_{x \in a, y \in b} \dG(x,y).$$
Since $\dG^{k}$ is defined between atoms of the same level $k$, we will always refer to $\dG^{k}$ simply as $\dG$ when $k$ will be already specified by $a$ and $b$. Whenever we will deal with atoms of different levels and hence with $\dG^{k}$ and $\dG^{k+1}$, we will make it clear. Moreover, in statements and discussions we may consider the collection of semi-metrics $\{ \dG^{k} \}_{_{k=1}}^{^{\infty}}$ and we will simply say \textit{the semi-metric} $\dG$ (note that $\dG$ defined on $\Gamma$ is instead a metric).\\
Since we are interested in horofunctions, given $u= ( u_{k} )_{k=1}^{\infty}$ and $v= ( v_{k} )_{k=1}^{\infty}$, we will consider $\{ \dG(u_{k},v_{k}) \}_{_{k=1}}^{^{\infty}}$. The first thing we note is that the sequence of distances is non-decreasing. Indeed, we take $x_{k+1} \in u_{k+1}$ and $y_{k+1} \in v_{k+1}$ such that $\dG(u_{k+1},v_{k+1})=\dG(x_{k+1},y_{k+1})$. But $x_{k+1} \in u_{k}$ and $y_{k+1} \in v_{k}$, hence by definition $$\dG(u_{k},v_{k}) \leq \dG(x_{k+1},y_{k+1}).$$
Despite the fact that we are looking to $\partial_{h} \Gamma$ from a geometric viewpoint, we can adopt a more analytic distance 
$$\dH(u,v):=\Vert f_{u}-f_{v} \Vert_{\infty}.$$ In particular, we have the following result which is clearly related to our goal and will be useful for our proof.
\begin{prop}
\label{Prop4.4}
Let $\Gamma$ be a hyperbolic graph. Let $u$ and $v$ be two horofunctions of $\Gamma$  and let $f_{u} \in u$ and $f_{v} \in v$ be two representatives such that $f_{u}(x_{0})=f_{v}(x_{0})=0$. Then $u$ and $v$ are glued on $\partial \Gamma$ if and only if there exists $\lambda_{h}$ (independent from $u$ and $v$) such that 
$$d_{h}(u,v):= \Vert f_{u}-f_{v} \Vert_{\infty} < \lambda_{h}. $$ 
\end{prop}
We will always denote this constant with $\lambda_{h}$. 
\begin{proof}
See \cite[Proposition 4.4]{WW}
\end{proof}
\noindent Note that if  $u= ( u_{k} )_{k=1}^{\infty}$ is a horofunction and $f_{u}$ is the representative such that $f_{u}(x_{0})=0$. When we consider its restriction to $B_{k}$, we get  
$$f_{u_{k}}:=\dG(-,u_{k})-\dG(x_{0},u_{k})$$
and $\dH(u_{k},v_{k})=\Vert f_{u_{k}}-f_{v_{k}} \Vert_{\infty}$.\

With the following results we link the geometric distance with the analytic one. 
\begin{lem}
\label{Conj2}
For all $k \in \mathbb{N}$ and for all pairs of $k$-level atoms $u_{k}$ and $v_{k}$ we have
$$ \frac{1}{2} \dH(u_{k},v_{k}) \leq \dG(u_{k},v_{k}) \leq \dH(u_{k},v_{k})+2\lambda_{a}.$$
\end{lem}
\begin{proof}
The first part follows by two triangle inequalities. Indeed, for all $p \in \Gamma$ we have
$$|\dG(u_{k},p)-\dG(v_{k},p)| \leq \dG(u_{k},v_{k}),$$
and since we also have
\begin{align*}|\dG(u_{k},p)-\dG(u_{k},x_{0}) & -\dG(v_{k},p)+\dG(v_{k},x_{0})| \leq \\  &|\dG(u_{k},p)-\dG(v_{k},p)|+|\dG(v_{k},x_{0})-\dG(u_{k},x_{0})|,\end{align*}
with $p \in B_{k}$,
by applying the first inequality with two times,
we get $\dH(u_{k},v_{k}) \leq 2\dG(u_{k},v_{k})$.\

For the second part, we start by taking $x \in u_{k}$ and $ \overline{x} \in B_{k}$ (resp. $y \in v_{k}$ and $ \overline{y} \in B_{k}$) such that $$\dG(B_{k}, u_{k})=\dG(\overline{x},x) \ \ (\text{resp. } \dG(B_{k}, v_{k})=\dG(\overline{y},y)  ).$$
By definition of $f_{v_{k}}$ and $\overline{y}$ we know that $f_{v_{k}}(\overline{y})=-k$. This implies that $f_{u_{k}}(\overline{y})+k \leq d_{h}(u_{k},v_{k})$. Equivalently, we have $\dG(\overline{y},x)-\dG(x_{0},x)+k \leq d_{h}(u_{k},v_{k})$ and by using the definition of $\overline{x}$, we get 
\begin{align*}
 \dG(\overline{y},x)-[\dG(x_{0},\overline{x})+\dG(\overline{x},x)]+k & = \\  \dG(\overline{y},x)-[k+\dG(\overline{x},x)]+k & \leq d_{h}(u_{k},v_{k}). \end{align*}
\noindent
Which means that $\dG(\overline{y},x) \leq \dG(x,\overline{x})+d_{h}(u_{k},v_{k})$.\

Now 
$$\dG(u_{k},v_{k}) \leq \dG(x,y) \leq \dG(x, \overline{y})+\dG(\overline{y}, y),$$
by the discussion in the previous paragraph we get 
$$\dG(u_{k},v_{k}) \leq \dG(x, \overline{x})+d_{h}(u_{k},v_{k}) +\dG(\overline{y}, y)$$
and by the \hyperref[Conj3]{Hooking Lemma} we have the thesis. 
\end{proof}

By combining the previous facts, we get the main result about $\dG$.

\begin{thm}
\label{Conj1}
Let $\Gamma$ be a hyperbolic graph and let $\pi_{h}: \partial_{h} \Gamma \twoheadrightarrow \partial \Gamma$ be the projection of the horofunction boundary onto the Gromov boundary. Let $u= ( u_{k} )_{k=1}^{\infty}$ and $v= ( v_{k} )_{k=1}^{\infty}$ be two horofunctions expressed via their infinite sequences of infinite atoms. The following are equivalent.
\begin{itemize}
\item[(A)]The horofunctions are glued on $\partial \Gamma$ i.e.\ $u \pi_{h}=z_{\infty}=v \pi_{h}$ for some $z_{\infty} \in \partial \Gamma$.
\item[(B)]There exists a constant $C$ and two sequences of vertices $\{ x_{k} \}_{k=1}^{\infty}$ and $\{ y_{k} \}_{k=1}^{\infty}$ such that $x_{k} \in u_{k}$ and $y_{k} \in v_{k}$, and $\dG(x_{k},y_{k}) \leq C$ for all $k \geq 1$. 
\end{itemize}
\end{thm}
From now on $\lambda$ will be the \textbf{gluing constant},  that is the smallest constant such that the claim holds. Note that, since $\dG(x_{k},y_{k})$ is a natural number, $\lambda$ is well-defined and is a natural number.
\begin{proof}
Suppose that (B) holds. The fact that $x_{k} \in u_{k}$ for all $k \geq 1$ implies that $\{ x_{k} \}_{k=1}^{\infty}$ converges to the horofunction $u$. It follows that $\{ x_{k} \}_{k=1}^{\infty}$ goes to infinity in the sense of Gromov. Analogously we have that $\{ y_{k} \}_{k=1}^{\infty}$ converges to $v$ and goes to infinity in the sense of Gromov. A fortiori, the distances $\dG(x_{0},x_{k})$ and $\dG(x_{0},y_{k})$ go to infinity as $k \rightarrow \infty$. By hypothesis
$$(x_{k} \mid y_{k}) \geq \dfrac{1}{2} \left[ \dG(x_{k},x_{0})+\dG(y_{k},x_{0})-C \right]$$ and (A) follows.\

All that is left is to combine the previous results to get (B) starting from (A). By Lemma~\ref{Conj2} we know that $\dG(u_{k},v_{k}) \leq d_{h}(u_{k},v_{k})+2\lambda_{a}$
and by Proposition~\ref{Prop4.4} we have that $d_{h}(u,v) < \lambda_{h}$, hence $d_{h}(u_{k},v_{k}) < \lambda_{h}$. To conclude, we have that $\dG(u_{k},v_{k}) < \lambda_{h}+2\lambda_{a}$ for all $k \geq 1$, so there exist $x_{k} \in u_{k}$ and $y_{k} \in v_{k}$ such that $\dG(x_{k},y_{k}) < C$ with $C=\lambda_{h}+2\lambda_{a}$.
\end{proof}

A straightforward consequence is what we were looking for.

\begin{cor}
The semi-metric $\dG$ represents the gluing.
\end{cor}

The second distance we introduce is aimed to illustrate the fact that atoms which glue are near in an asymptotic way i.e.\ the distance occurs to be less than the gluing constant $\lambda$ in an infinite number of points which lie in the atoms.

\begin{defn}
We define $$\dF^{k}(a,b):= \sup_{A \subset a,  B \subset b} \min_{\substack{x \in a-A \\ y \in b-B}} \dG(x,y)$$
with $a$ and $b$ atoms of the $k$-level and the $A$ and $B$ finite sets.
\end{defn}
\noindent Again, we will write $\dF$ when $k$ is clear or to mean the collection of $ \{ \dF^{k} \}_{_{k=1}}^{^{\infty}}$. Note that if $\dF(a,b)$ is finite, then the supremum is actually a maximum.\

The aim, the definition and the discussion that follow show how we can think $\dF$ as a limit of $\dG$.\

A particular benefit of this distance is that, for a fixed level, it can be calculated by looking at the level below. In fact, $\dF(a,b)$ is the minimum over all the children of $a$ and $b$ of the distances between such children.\

We start a comparison with $\dG$ by stating some properties.

\begin{itemize}
\item[(1)]The distance $\dF$ is a semi-metric.
\item[(2)]If $a_{k+1} \subseteq a_{k}$ and $b_{k+1} \subseteq b_{k}$, then $\dF(a_{k},b_{k}) \leq \dF(a_{k+1},b_{k+1})$. This follows immediately from the definition (as for $\dG$) or using the discussion we made above.
\item[(3)] The distance $\dG$ is less than or equal to the distance $\dF$. Indeed, if we put $A=B= \emptyset$ in the definition of $\dF$ we recover $\dG$.
\end{itemize}

In general the converse of Property (3) is not true: it can happen that two atoms are $\dG$-adjacent but not $\dF$-adjacent. But something more specific can be stated.

\begin{lem} \label{gamma_implies_f}
Let $u= ( u_{k} )_{k=1}^{\infty}$ and $v= ( v_{k} )_{k=1}^{\infty}$ two horofunctions. If $\dG(u_{k},v_{k}) \leq C$ for all $k \geq 1$, then $\dF(u_{k},v_{k}) \leq C$ for all $k \geq 1$.
\end{lem}
\begin{proof}
We fix $k$. We know that for all $n \geq k$ there exists a $\widehat{k} \geq k$ 
such that $u_{\widehat{k}} \subseteq u_{k}-B_{n}$ and $v_{\widehat{k}} \subseteq v_{k}-B_{n}$. So that $$\min_{\substack{x \in u_{k}-B_{n} \\ y \in v_{k}-B_{n}}} \dG(x,y) \leq \min_{\substack{x \in u_{\widehat{k}} \\ y \in v_{\widehat{k}}}} \dG(x,y)=\dG(u_{\widehat{k}},v_{\widehat{k}}) \leq C.$$
It follows that $\displaystyle \sup_{n \geq k} \min_{\substack{x \in u_{k}-B_{n} \\ y \in v_{k}-B_{n}}} \dG(x,y) \leq C$.  \

Taking a finite subset $A$ of $u_{k}$ and a finite subset $B$ of $v_{k}$, there exists an $n$ such that $A \cup B \subseteq B_{n}$. Hence $u_{k}-A \supseteq u_{k}-B_{n}$ and $v_{k}-B \supseteq v_{k}-B_{n}$, that implies
$$\min_{\substack{x \in u_{k}-A \\ y \in v_{k}-B}} \dG(x,y) \leq \min_{\substack{x \in u_{k}-B_{n} \\ y \in v_{k}-B_{n}}} \dG(x,y).$$
Finally, we get $\dF(u_{k},v_{k}) \leq \displaystyle \sup_{n \geq k} \min_{\substack{x \in u_{k}-B_{n} \\ y \in v_{k}-B_{n}}} \dG(x,y)$ and the claim follows.
\end{proof}

As an immediate consequence of Theorem~\ref{Conj1} and Lemma~\ref{gamma_implies_f} we have
\begin{cor}
The semi-metric $\dF$ represents the gluing.
\end{cor}

\begin{oss}\

\begin{itemize}
\item[--]In fact, we have proven that if there exists a level $k$ such that $\dF(u_{k},v_{k}) \leq C$, then there exist two horofunctions $u$ and $v$ with $u_{k}$ and $v_{k}$ respectively in their sequences such that $u \pi_{h}= v \pi_{h}$.
\item[--]If $u \pi_{h} \neq v \pi_{h}$, then $\dG(u_{k},v_{k}) \rightarrow \infty$. So there exists an index $\overline{k}$ such that $\dF(u_{\overline{k}},v_{\overline{k}})$ is finite and $\dF(u_{\overline{k}+1},v_{\overline{k}+1})= \infty$.
\end{itemize}
\end{oss}

We are interested in a distance which can represent the gluing, but also that increases exponentially when there is no gluing. It can easily be seen that $\dG$ is too slow. Indeed, if we suppose that $a$ and $b$ are two $k$-level atoms, then $\dG^{k}(a,b) \leq \dG(x,x_{0})+\dG(x_{0},y) \leq 2k+2\lambda_{a}$ for some $x \in a$ and $y \in b$.
On the other hand, $\dF$ is too fast, as stated in the second point of the previous Remark. With this purpose in mind, we give the following

\begin{defn}
Let $k$ be an integer greater or equal to $1$. We consider $\Gamma-B_{k-1}$ as the induced subgraph with respect to the subset of vertices $\Gamma-B_{k-1}$. We can put a metric on $\Gamma-B_{k-1}$ which is the standard metric on the graph and we call it $\dB^{k}$. 
\end{defn}
\noindent The notation is subject to the same convention used before with $\dG$ and $\dF$.\\
Please note that $\dB^{k}$ is finite if and only if $\Gamma-B_{k-1}$ is connected and that, according to Proposition~\ref{complement-atom}, it holds $\mathcal{A}_{k} \subseteq \Gamma-B_{k-1}$.\\
In the same way we defined the semi-metric $\dG$, we put
$$\dB(a,b)= \min_{x \in a, y \in b} \dB(x,y)$$ with $a,b \in \mathcal{A}_{k}$.\\
We note that $\dB$ is an intrinsic metric with respect to $\dG$, and hence we know that $\dG \leq \dB$ (see Remark~\ref{length}), that is $\dG(a,b) \leq \dB(a,b)$ for all $a,b \in \mathcal{A}_{k}$ and for all $k$. It follows that $\dB$ is faster than $\dG$ as a semi-metric (when defined on atoms). We are going to prove a technical lemma, which will be useful to prove that $\dB$ represents the gluing.

\begin{lem}
\label{fullcont_trick}
Let $a,b \in \mathcal{A}_{k}$ such that $\dF(a,b) \leq C$. Then there exists $n \geq k$ such that there exist $\tilde{x} \in a-B_{n}$ and $\tilde{y} \in b-B_{n}$ with the following properties: $\displaystyle \dG(\tilde{x}, \tilde{y})=\min_{\substack{x \in a-B_{n} \\ y \in b-B_{n}}} \dG(x,y)$ and there exists a geodesic $[\tilde{x},\tilde{y}]_{\Gamma}$ (with respect to $\dG$) fully contained in $\Gamma-B_{k-1}$.
\end{lem} 
\begin{proof}
By hypothesis, we know that for all $n \geq k$ it holds we have $\displaystyle \min_{\substack{x \in a-B_{n} \\ y \in b-B_{n}}} \dG(x,y) \leq C$. We take $n \geq \dfrac{C}{2}+k-1$ and $\tilde{x}, \tilde{y} \in \Gamma$ such that $$l:=\dG(\tilde{x}, \tilde{y})=\min_{\substack{x \in a-B_{n} \\ y \in b-B_{n}}} \dG(x,y).$$\\
Suppose that there exists $[\tilde{x},\tilde{y}]_{\Gamma}= \{z_{i} \}_{_{i=0}}^{^{l}}$ and there exists a $j \in \{ 0, \ldots, l \}$ such that $z_{j} \in B_{k-1}$. Now 
$$C \geq \dG(\tilde{x}, \tilde{y})=\dG(\tilde{x}, z_{j})+\dG(z_{j}, \tilde{y})> C$$ 
and the claim follows.
\end{proof}

Combining Lemma~\ref{fullcont_trick} with Remark~\ref{length}, in the case that $\dF(a,b) \leq C$, we get $$\dB(a,b) \leq \dB(\tilde{x}, \tilde{y}) = \dG(\tilde{x}, \tilde{y}) \leq \dF(a,b).$$ Note that $\dF(a,b)= \infty$ the inequality is true. So we can state the following
\begin{prop}The semi-metric $\dB$ is slower than $\dF$, namely $\dB \leq \dF$, which means $\dB(a,b) \leq\dF(a,b)$ for all $a,b \in \mathcal{A}_{k}$ and for all $k$.
\end{prop}
If we put together the previous Proposition and the discussion we made about intrinsic metrics, we get the following
\begin{cor}
The semi-metric $\dB$ represents the gluing.
\end{cor}

\section{Distances on tips and consequences} \label{4}
In this section we will combine distances on atoms and tips to answer the question on exponential behaviors of atom-codings that arose at the end of the previous section, more precisely we will find a distance with the following property:
\begin{defn}
A distance $\dist$ on atom-codings has the \textbf{exponential property} if taking two horofunctions $u=(u_{k})_{k=1}^{\infty}$ and $v=(v_{k})_{k=1}^{\infty}$ one of the following holds
\begin{itemize}
\item[(1)]$u\pi_{h}=v\pi_{h}$ and $\dist(u_{k},v_{k})$ is bounded above by some constant (depending only on $\Gamma$) for all $k$ (that is $\dist$ represents the gluing);
\item[(2)]$u\pi_{h} \neq v\pi_{h}$ and $\dist(u_{k},v_{k}) \geq E_{1}e^{E_{2}(k-j)}$ for some constant $E_{1}$ and $E_{2}$ and for all $k \geq j$ with $j$ sufficiently large (we say that $\dist$ \textbf{diverges exponentially}).
\end{itemize}
\end{defn} 
Note that this definition resembles the property of geodesics in hyperbolic spaces (see  Proposition \ref{exp_div_geodesics})
and that it will be discussed again later on using more powerful tools to relate it with the Gromov product. \\
We will also prove that the Hausdorff distance can represent the gluing and we will bound the fibers of $\pi_{h}$.\\
We start by introducing a useful notation. Given a semi-metric $\dist$ on some level of atoms $\mathcal{A}_{k}$, we will denote by
$$\T \dist(a,b):=\dist(T(a),T(b)) \ \ \ \ \text{with } a,b \in \mathcal{A}_{k}.$$
where we recall that $T(a)$ and $T(b)$ denote the tips of $a$ and $b$.

\begin{oss} \label{Conj1_tips}
We point out that we already know something important about $\T \dG$. Indeed, if we look at the proof of Theorem~\ref{Conj1}, we see that we actually prove that \textit{(A) implies (B)}
for $\T \dG$, while for the other direction we can easily exploit the same technique in the proof to get the result.
\end{oss}

So we have the following
\begin{cor}
The semi-metric $\T \dG$ represents the gluing.
\end{cor}

In fact, we can prove a stronger result about the relation between $\dG$ and $\T\dG$.

\begin{prop}
\label{tips_non-tips_relation}
Let $a,b \in \mathcal{A}_{k}$. Then 
$$\dG(a,b) \leq \T\dG(a,b) \leq 2\dG(a,b)+4\delta+2\lambda_{a}.$$
\end{prop}
\begin{proof}
We only need to prove the second inequality. Let $x \in a$ and $y \in b$ such that $\dG(a,b)=\dG(x,y)$. By Proposition~\ref{cone_intersect}, we have that $x \in C(\overline{x})$ for some $\overline{x} \in N(a,B_{k})$. The same is true for $y$ and $\overline{y} \in N(b,B_{k})$. Now we apply Proposition~\ref{generalized_fellow_traveler} to the geodesics $[x_{0},x]$ and $[x_{0},y]$ respectively passing through the two nearest neighbors at the time $k$, so that $\dG(\overline{x},\overline{y}) \leq 2(\dG(x,y)+2\delta)$. To finish the proof, we use the the \hyperref[Conj3]{Hooking Lemma} together with a triangle inequality
$$\dG(\widehat{x},\widehat{y}) \leq \dG(\widehat{x},\overline{x})+
\dG(\overline{x},\overline{y}) + \dG(\overline{y}, \widehat{y})
\leq 2\dG(x,y)+4\delta+2\lambda_{a}, $$
which leads to $\dG(\widehat{x},\widehat{y}) \leq 2\dG(x,y)+4\delta+2\lambda_{a}$ for any two vertices $\widehat{x} \in T(a)$ and $\widehat{y} \in T(b)$.

\end{proof}

With a bit of work, we can say something about $\T \dB$ too.
\begin{prop}
The semi-metric $\T \dB$ represents the gluing.
\end{prop}
\begin{proof}
Since $\dG \leq \dB$, it follows that $\T \dG \leq \T \dB$, or more precisely  $\T \dG(a,b) \leq \T \dB(a,b)$ for all $a,b \in \mathcal{A}_{k}$ and for all $k$. It remains to prove that: if $u=(u_{k})_{k=1}^{\infty}$ and $v=(v_{k})_{k=1}^{\infty}$ are two horofunctions and there exists $C_{1}$ such that $\T \dG(u_{k}, v_{k}) \leq C_{1}$ for all $k \geq 1$, then there exists $C_{2}$ such that $\T \dB(u_{k},v_{k}) \leq C_{2}$ for all $k \geq 1$.\\

Let $\widehat{x} \in T(u_{k})$ and $\widehat{y} \in T(v_{k})$. We know that if $u_{\tilde{k}} \subseteq u_{k}$ and $v_{\tilde{k}} \subseteq v_{k}$ are such that $u_{\tilde{k}} \cup v_{\tilde{k}} \subseteq \Gamma-B_{\frac{C_{1}}{2}+k-1}$, then $\T \dB^{k}(u_{\tilde{k}},v_{\tilde{k}})=\T \dG(u_{\tilde{k}},v_{\tilde{k}})$ (this is the same argument from Lemma~\ref{fullcont_trick}).\\
We take $u_{\tilde{k}}$ and $v_{\tilde{k}}$ to be the first occurrences of atoms contained in $\Gamma-B_{\frac{C_{1}}{2}+k-1}$ (i.e.\ the one with minimal distance from $x_{0}$ contained in $u_{k}$ and $v_{k}$ respectively). We set $\tilde{x} \in T(u_{\tilde{k}})$ and $\tilde{y} \in T(v_{\tilde{k}})$ such that $\T \dG(u_{\tilde{k}},v_{\tilde{k}})=\dG(\tilde{x},\tilde{y})$.\\
Note that by the \hyperref[Conj3]{Hooking Lemma} we know that $T(u_{\tilde{k}}) \cup T(v_{\tilde{k}}) \subseteq \Gamma-B_{\frac{C_{1}}{2}+k-1+\lambda_{a}}$.\\

Since we want to study $\dB(\widehat{x},\widehat{y})$, we start by applying a triangle inequality and the main hyphotesis:

$$\dB(\widehat{x},\widehat{y}) \leq \dB^{k}(\widehat{x},\tilde{x})+\dB^{k}(\tilde{x},\tilde{y})+\dB^{k}(\tilde{y},\widehat{y}) \leq \dB^{k}(\widehat{x},\tilde{x})+C_{1}+\dB^{k}(\tilde{y},\widehat{y}).$$

Now $\tilde{x} \in \left(\Gamma-B_{\frac{C_{1}}{2}+k-1+\lambda_{a}}\right) \cap C(p)$ 
for every $p \in N(u_{k},B_{k})$. We have that there exists $[p,\tilde{x}]_{\Gamma} \subseteq \Gamma-B_{k-1}$ and its length is less or equal than $\frac{C_{1}}{2}-1+\lambda_{a}$. Exploiting the \hyperref[Conj3]{Hooking Lemma} again, we get
$$\dB^{k}(\widehat{x},\tilde{x}) \leq \dB^{k}(x,p)+\dB^{k}(p, \tilde{x})=\dG(x,p)+\dG(p,\tilde{x}) \leq \frac{C_{1}}{2}-1+2\lambda_{a}.$$
In an analogous way, we get  $\dB^{k}(\widehat{y},\tilde{y}) \leq \frac{C_{1}}{2}-1+2\lambda_{a}$; and hence the claim.
\end{proof}
It remains to prove the following
\begin{prop}\label{exp_prop_tdb}
The semi-metric $\T \dB$ diverges exponentially.
\end{prop}
\begin{proof}
Let $u,v \in \partial_{h} \Gamma$ such that $u\pi_{h} \neq v\pi_{h}$. Since $\dG(u_{k},v_{k})$ is unbounded as $k$ tends to infinity, we can choose $j$ such that $\dG(u_{j},v_{j})$ is arbitrarily large. In fact, taken $\overline{x} \in N(u_{j},B_{j})$ and $\overline{y} \in N(v_{j},B_{j})$, we know by the \hyperref[Conj3]{Hooking Lemma} that 
$$\dG(\overline{x},\overline{y}) \geq \dG(u_{j},v_{j})-2\lambda_{a}$$
hence we can take $\dG(\overline{x},\overline{y})$ arbitrarily large.\\
Now for all $k \geq j$ we have that $x_{k} \in u_{k}$ belongs to $C(\overline{x})$, and the same holds for $y_{k} \in v_{k}$ and $C(\overline{y})$. We construct a pair of geodesics $[x_{0},x_{k}]$ and $[x_{0},y_{k}]$ passing through $\overline{x}$ and $\overline{y}$ respectively and by Proposition~\ref{exp_div_geodesics}, we know that their fellow travelers diverge exponentially. We consider $x$ to be nearest neighbor of $x_{k}$ such that it belongs to $[x_{0},x_{k}] \cap S_{k}$ and the same for $y \in [x_{0},y_{k}] \cap S_{k}$, so it holds that 
$$\dB^{k}(x,y) \geq E_{1} e^{E_{2}(k-j)}$$
for some constants $E_{1}$ and $E_{2}$. To conclude, we take $x_{k} \in T(u_{k})$ and $y_{k} \in T(v_{k})$ (with $\dB^{k}(x_{k},y_{k})=\T \dB^{k}(u_{k},v_{k})$) and we have
\begin{align*}
\T \dB^{k}(u_{k},v_{k}) =\dB^{k}(x_{k},y_{k}) & \geq \dB^{k}(x,y)-\dB^{k}(x,x_{k})-\dB^{k}(y,y_{k})\\
& \geq  E_{1} e^{E_{2}(k-j)}-\dB^{k}(x,x_{k})-\dB^{k}(y,y_{k}).
\end{align*}
By construction, the geodesic $[x,x_{k}] \subseteq \Gamma-B_{k-1}$ and hence $\dG(x,x_{k})=\dB^{k}(x,x_{k})$. Analogously, we get $\dG(y,y_{k})=\dB^{k}(y,y_{k})$.\\
Again, by virtue of the \hyperref[Conj3]{Hooking Lemma} combined with the previous inequality, we finally obtain
$$\T \dB^{k}(u_{k},v_{k}) \geq E_{1} e^{E_{2}(k-j)} -2\lambda_{a}.$$
\end{proof}

\begin{cor}
The semi-metric $\T \dB$ has the exponential property.
\end{cor}

There are two other important consequences that we want to briefly discuss, that comes from the tips approach. The first one involves a well-known metric.\\

\begin{prop}\label{quasi_iso_like_hausdorff}
The Hausdorff metric $\T \Haus$ represents the gluing.
\end{prop}
\begin{proof}
Let $a,b \in \mathcal{A}_{k}$. Since $\T \dG$ is the minimum over all pairs of elements $(\widehat{x},\widehat{y})$ that belong to the tips $T(a)$ and $T(b)$ respectively, it is straightforward that 
$$\T \dG(a,b) \leq \min_{\widehat{y} \in T(b)} \dG(\widehat{x},\widehat{y}) \leq \max_{\widehat{x} \in T(a)} \min_{\widehat{y} \in T(b)} \dG(\widehat{x},\widehat{y})$$
and that the same holds in the other way (with $a$ and $b$ switched). So $\T \dG (a,b) \leq \T \Haus(a,b)$.\\

For the other direction, we need to use Proposition~\ref{diam}, that leads to
$$\max_{\widehat{x} \in T(a)} \min_{\widehat{y} \in T(b)} \dG(\widehat{x},\widehat{y}) \leq \T \dG(a,b) + 2\lambda_{a}$$
and hence $\T \Haus(a,b) \leq \T \dG(a,b)+2\lambda_{a}$ as before.
\end{proof}

The last consequence is about the fibers of $\pi_{h}$. We have already discussed some properties, in particular Theorem~\ref{Conj1} tells us when two elements belong to the same fiber. But now, we can prove the following
\begin{thm}\label{finite_to_one}
The map $\pi_{h}: \partial_{h} \Gamma \twoheadrightarrow \partial \Gamma$ is finite-to-one. Moreover, the number of elements in a fiber is bounded by a constant that depends only on $\lambda_{a}$ and $\delta$.
\end{thm}
\begin{proof}
We recall our assumption on the graph $\Gamma$: it is locally finite and hence the balls are finite. We consider the atom-coding of a horofunction $u=(u_{k})_{k=1}^{\infty}$ and we know that the tip $T(u_{k})$ has a finite diameter due to Proposition~\ref{diam} for all $k$. In particular, it is bounded above by $2\lambda_{a}$. We choose an element in $T(u_{k})$ and we consider a ball $B$ of radius $2(\lambda_{a}+2\delta)+2\lambda_{a}$ centered at that element. By Theorem~\ref{Conj1} and Remark~\ref{Conj1_tips},
we get that two horofunctions map onto the same point in the Gromov boundary if the distance of their tips at each level is less or equal than $2(\lambda_{a}+\delta)$. So by definition of the ball $B$, the $k$-level tip of every horofunction that is contained in the same fiber of $u$ must intersect $B$. Since it holds for all $k$ and the constants do not depend on $k$, we have the claim.
\end{proof}
Despite the differences, the way of thinking of the part of the section shares many points with \cite{CP}, in particular this proof uses the same technique provided there.\\

It is worth noting that this result can be achieved in a different way. In their work \cite{CP,CP2}, Coornaert and Papadopoulos use a different notion of horofunction (this defintion was provided by Gromov in \cite{G}). 
\begin{defn} \label{CP_function}
Let $\Gamma$ a hyperbolic graph and let $x_{0}$ be a distinguished vertex in $\Gamma$. A map $f:|\Gamma| \rightarrow \mathbb{R}$ with $f(x_{0})=0$ is called a $CP$-function if it satisfies the following two conditions:
\begin{itemize}
\item[(1)]There exists $\epsilon >0$ such that 
$$f(l\gamma(t)) \leq (1-t)f(l\gamma(0))+tf(l\gamma(1))+\epsilon$$ 
with $\gamma:[0,l] \rightarrow |\Gamma|$ geodesic and $l:[0,1] \rightarrow [0,l]$ that maps $t \in [0,1]$ to $lt \in [0,l]$;
\item[(2)]$f(x)=\tilde{t}+\dG(x,f^{-1}(\tilde{t}))$ for every $x \in |\Gamma|$ and every $\tilde{t} \in ]-\infty, f(x)]$.
\end{itemize}
\end{defn}
They then managed to prove that the space of $CP$-functions that assume only integer values on the vertices of $\Gamma$ projects onto $\partial \Gamma$ and the quotient map is finite-to-one (see \cite[Proposition 4.5]{CP2}). 
So all we need to conclude that the fibers are finite is the following
\begin{prop}[\cite{mathoverflow}]
Let $\Gamma$ be a hyperbolic graph. Then a horofunction is a $CP$-function.
\end{prop}
Note that the converse is false. Since there are hyperbolic graphs such that the two notions do not coincide (see \cite{mathoverflow} for details). \\

Theorem~\ref{finite_to_one} gives us a tool to bound the \textit{topological dimension} of $\partial \Gamma$. 

\begin{thm}[Hurewicz] \label{Hurewicz}
Let $X$ be a compact metrizable space that is a continuous image of a Cantor set. If the fibers of the map are bounded above by an integer $n>0$, then the topological dimension of $X$ is less than $n-1$.
\end{thm}
\noindent See \cite[Chapter XIX]{Ku}.\\
Applications of this theorem are common in geometric group theory, see e.g.\ the bound for limit sets of contracting self-similar groups in \cite[Proposition 5.7]{N2} and the bound for hyperbolic graphs in two different versions, namely Proposition 3.7 and Proposition 4.2 followed by Corollary 5.2 in \cite{CP}.

\section{Using geodesics and geodesic rays} \label{5}
This section is based on the following strategy: first we will associate quasi-geodesic rays to atom-codings, then finite geodesics and in the end also geodesic rays. This approach allows us to put together and generalize the previous sections and gives a first approximation of the Gromov boundary via atoms. To start, we need to iterate the hooking lemma, that is
\begin{prop}
\label{iterated_hooking}
Let $u_{k} \supseteq u_{k+1}$ two atoms respectively of level $k$ and $k+1$. Then
$$\max_{\substack{x \in T(u_{k}) \\ y \in T(u_{k+1})}} \dG(x,y) \leq 2\lambda_{a}+1.$$ 
\end{prop}
\begin{proof}
Let $\widehat{x}_{k+1} \in T(u_{k+1})$. By the \hyperref[Conj3]{Hooking Lemma}, we know that $\dG(x_{0},\widehat{x}_{k+1})$ is less than or equal to $k+1+\lambda_{a}$. 
And so it is the length of a geodesic starting from $x_{0}$ and passing through any nearest neighbor $\overline{x} \in S_{k}$ of $u_{k}$ (recall that we can take any of them due to Lemma~\ref{visible_proximal_c}).
We know that such a geodesic exists because $\widehat{x}_{k+1} \in u_{k}$ and by Proposition~\ref{cone_intersect}. But this means that $\dG(\overline{x},\widehat{x}_{k+1}) \leq \lambda_{a}+1$ and again by the \hyperref[Conj3]{Hooking Lemma}, we have $\dG(\overline{x},\widehat{x}_{k}) \leq \lambda_{a}$ with $\widehat{x}_{k}$ any element in the tip of $u_{k}$. Combining these two facts in a triangle inequality we get the claim.
\end{proof}

Note that if we look at the minimum, namely 
$$\min_{\substack{x \in T(u_{k}) \\ y \in T(u_{k+1})}} \dG(x,y), $$
then maybe the estimate provided is naive. 
Indeed, it can happen that two elements in two consecutive tips are one the successor of the other. 
\begin{exm}
Looking at Figure~\ref{atoms_tiling}.(\subref{atoms_tiling_A}), we get an example of two consecutive tips. Indeed, every element in the sphere of radius 2 which is an element in the tip of a $2$-level atom is a successor  of an element in the tip of the corresponding $1$-level atom.
\end{exm}
Despite this aspect, we are able to construct a quasi-geodesic ray

\begin{prop}
Let $u=(u_{k})_{k=1}^{\infty}$ be an element of $\partial_{h} \Gamma$ described by its atom-coding. Then any sequence of points $\{\widehat{x}_{k}\}_{k=1}^{\infty}$ such that $\widehat{x}_{k} \in T(u_{k})$ has a subsequence that is a quasi-geodesic ray.
\end{prop}
\begin{proof}
Take $\{\widehat{x}_{k}\}_{k=1}^{\infty}$ as in the statement. It may occur that some $\widehat{x}_{k}$ are equal so, for each $k$, we remove all the redundant copies of the same $\widehat{x}_{k}$ to get a new sequence $\{\widehat{x}_{n}\}_{n=1}^{\infty}$. We claim that $\{\widehat{x}_{n}\}_{n=1}^{\infty}$ is a quasi-isometric embedding of $\mathbb{N}$ in $\Gamma$.\\
By Proposition~\ref{iterated_hooking} we know that $\dG(\widehat{x}_{n},\widehat{x}_{n+1}) \leq D$ for some constant $D$ depending only on $\lambda_{a}$. To conclude, we know that if $\widehat{x}_{n}, \widehat{x}_{m} \in \{\widehat{x}_{n}\}_{n=1}^{\infty}$  and $n \leq m$, then $\dG(\widehat{x}_{n},\widehat{x}_{m}) \leq D(m-n)$ by iterations of the triangle inequality.\\
On the other hand, $\widehat{x}_{n} \in B_{n-1+\lambda_{a}}$ by virtue of the \hyperref[Conj3]{Hooking Lemma} and $\widehat{x}_{n} \in \Gamma-B_{m-1}$ by Proposition~\ref{complement-atom}, hence $\dG(\widehat{x}_{n},\widehat{x}_{m}) \geq m-n -\lambda_{a}$.
\end{proof}
An interesting fact to remark is that if such a quasi-geodesic ray is a geodesic ray, then by definition the horofunction is a Busemann point.\\

\begin{oss}
Note that two horofunctions $u$ and $v$ are mapped into the same point in $\partial \Gamma$ if and only if the Hausdorff distance between the two geodesic rays constructed using the previous proposition is bounded (see \cite[Lemma H.3.1]{BH} for a different, but equivalent, definition of $\partial \Gamma$ that leads to this fact).
\end{oss}

The technique used in the following Remark will not only improve the structure of the quasi-geodesic ray, but it will also be the key ingredient for most of the incoming proofs:
\begin{oss}
\label{proximal_trick}
Let $a_{n} \in \mathcal{A}_{n}$ and take $p_{n}$ a proximal point of $a_{n}$. By Lemma~\ref{visible_proximal_b}, we can consider a combinatorial geodesic $[p_{0},p_{n}]=(p_{0},p_{1}, \ldots, p_{n})$ such that $p_{0}=x_{0}$ and $p_{i} \in P(x,S_{i})$ for some (and hence all) points $x \in a_{n}$ and for all $i \leq n$.
But we can say more, we actually find a geodesic such that $p_{i} \in P(a_{i}, S_{i})$ with $a_{i} \in \mathcal{A}_{i}$ and $a_{i} \supseteq a_{n}$ for all $i \leq n$, as $x \in a_{i}$.
\end{oss}

Following Definition~\ref{visual_metric} and the discussion right after it, we look at $\partial \mathcal{A}$ as a metric space with respect to $\curlyvee_{A}$, which is nothing more than the standard visual metric on the boundary of a rooted tree. Explicitly, we define $\curlyvee_{A}(u,v):=\beta^{-k}$ where $k$ is such that $u_{k}=v_{k}$ (this implies $u_{i}=v_{i}$ for every $i \le k$)
and $u_{k+1} \neq v_{k+1}$. As for other tree structures that are related to $\partial \Gamma$ (see e.g. \cite[Proposition 2.3]{CP}), we have the following
\begin{prop}
\label{lipschitz}
The map $\pi_{h}: ( \partial \mathcal{A}, \curlyvee_{A}) \twoheadrightarrow (\partial \Gamma, \curlyvee)$ is Lipschitz. 
\end{prop}
\begin{proof}
Let $u=(u_{n})_{n=1}^{\infty}$ and $v=(v_{n})_{n=1}^{\infty}$ be two horofunctions with $\curlyvee_{A}(u,v)=\beta^{-k}$. Pick $\{\widehat{x}_{n}\}_{n=1}^{\infty}$ and $\{\widehat{y}_{n}\}_{n=1}^{\infty}$ such that $\widehat{x}_{n} \in T(u_{n})$ and $\widehat{y}_{n} \in T(v_{n})$ for all $n$ and with $\widehat{x}_{n}=\widehat{y}_{n}$ for all $n \leq k$, now Remark~\ref{Vaisala_Gromov_product_c}  leads to  
$$2\tilde{\delta}+(u\pi_{h} \mid v\pi_{h}) \geq \sup (\sup_{n\geq 0} \inf_{m \geq n}(x_{m} \mid y_{m})) \geq \sup_{n\geq 0} \inf_{m \geq n}(\widehat{x}_{m} \mid \widehat{y}_{m}) $$
where the first $\sup$ is over all the sequences $\{x_{n}\}_{n=1}^{\infty}$ and $\{y_{n}\}_{n=1}^{\infty}$ such that $x_{n}$ converges to $u\pi_{h}$ and  $y_{n}$ converges to $v \pi_{h}$.\\
Since $\widehat{x}_{m}=\widehat{y}_{m}$ for all $m \le k$, then $(\widehat{x}_{m} \mid \widehat{y}_{m})=d(\widehat{x}_{m},x_{0}) \geq m$, while if $m > k$ then we have $\dG(\widehat{x}_{m},\widehat{y}_{m}) \leq \dG(\widehat{x}_{m},\widehat{x}_{k})+\dG(\widehat{y}_{m},\widehat{y}_{k})$ and 
$$\dG(\widehat{x}_{m},\widehat{x}_{k}) \leq \dG(\widehat{x}_{m},\overline{x}_{m})+\dG(\overline{x}_{m},p_{m})+\dG(p_{m},p_{k})+\dG(p_{k},\overline{x}_{k})+\dG(\widehat{x}_{k},\overline{x}_{k})$$
with $\overline{x}_{m}$ and $\overline{x}_{k}$ nearest neighbors of $\widehat{x}_{m}$ and $\widehat{x}_{k}$ respectively; and $p_{m}$ and $p_{k}$ proximal points of $\widehat{x}_{m}$ and $\widehat{x}_{k}$ on the same geodesic (see Remark~\ref{proximal_trick}). So that $\dG(\widehat{x}_{i},\overline{x}_{i}) \leq \lambda_{a}$ by the \hyperref[Conj3]{Hooking Lemma} and $\dG(\overline{x}_{i},p_{i}) \leq 4\delta+2$ by Lemma~\ref{visible_proximal_b} for $i=k,m$. Moreover, $\dG(p_{m},p_{k})=m-k$ and so we can conclude that 
$$\dG(\widehat{x}_{m},\widehat{x}_{k}) \leq m-k+D \text{ with } D=2(\lambda_{a}+4\delta+2).$$
The same holds for $\widehat{y}_{m}$ and $\widehat{y}_{k}$.\\
Returning to the Gromov product, we can say that
$$\dG(\widehat{x}_{m},x_{0})+\dG(\widehat{y}_{m},x_{0})-\dG(\widehat{x}_{m},\widehat{y}_{m}) \geq 2m-2(m-k+D) =2(k-D)$$
with $D$ a constant only depending on $\lambda_{a}$ and $\delta$. Hence $\curlyvee(u \pi_{h},v\pi_{h}) \leq \tilde{D}\beta^{-k}$ with $\tilde{D}=\beta^{D+2\tilde{\delta}}$ as desired.
\end{proof}

\begin{quest*}
Is there a connection between the visual metric and the uniform metric on $\partial_{h} \Gamma$ so that we can say the map $(\partial_{h} \Gamma,|\cdot |_{\infty}) \twoheadrightarrow (\partial \Gamma, \curlyvee)$ is Lipschitz?
\end{quest*}

Starting from an atom-coding, we want to find a geodesic ray that represents the same element of $\partial \Gamma$ as the horofunction associated to the atom-coding, more formally
\begin{prop}
\label{proximal_ray_prop}
Let $u$ be a horofunction. Then there exists a geodesic ray $\gamma_{u}$ such that $\gamma_{u}(0)=x_{0}$ and $\gamma_{u}(k)$ is proximal to $u_{k}$ for all $k \in \mathbb{N}$. Furthermore $[\gamma_{u}]_{\partial \Gamma}=u \pi_{h}$.
\end{prop}
Note that we are not saying that every horofunction is a Busemann point, this is false in general for hyperbolic graphs, as we already mentioned. What is true is that the horofunction $u$ is a Busemann point if and only if it coincides with the horofunction defined by $\gamma_{u}$.\\

The proof of the Proposition involves the following well-known result, together with the proximal points technique mentioned before.

\begin{thm}[Arzelà-Ascoli]
\label{AA}
Let $X$ be a proper geodesic space with a distinguished point $x_{0}$. Let $\{ \gamma_{k} \}_{k \in \mathbb{N}}$ be a sequence of functions $\gamma_{k}:[0, \infty[ \rightarrow X$ such that 
\begin{itemize}
\item[(a)]$\gamma_{k}(0)=x_{0}$ for all $k \in \mathbb{N}$,
\item[(b)]$\gamma_{k}$ is a geodesic on $[0,k]$.
\end{itemize}
Then there exists a subsequence $\{\tilde{\gamma}_{n}\}_{n \in \mathbb{N}}$ that converges uniformly on compacts to a geodesic ray $\gamma:[0,\infty[ \rightarrow X$.
\end{thm}

\begin{proof}[Proof(Proposition~\ref{proximal_ray_prop}).]
Let $(u_{k})_{k=1}^{\infty}$ be the sequence of atoms representing $u$. For each $k$, we exploit Remark~\ref{proximal_trick} to get a geodesic $\gamma_{k}$ such that $\gamma_{k}(n) \in P(u_{n},S_{n})$ for all $n \leq k$. Then we apply Theorem~\ref{AA} to the sequence $\{\gamma_{k}\}_{k \in \mathbb{N}}^{\infty}$ by saying that $\gamma_{k}(t):=\gamma_{k}(k)$ for all $t \geq k$ and we obtain a geodesic ray $\gamma$. We claim that every $n$-vertex of $\gamma$ (i.e.\ $\gamma(n)$ with $n \in \mathbb{N}$) is a proximal point of $u_{k}$. Indeed, we are dealing with uniform convergence on compacts, that is
$$\lim_{k} \max_{t \in [0,n]}\dG(\tilde{\gamma}_{k}(t),\gamma(t))=0$$
or in other words
$$\forall \epsilon >0 \ \ \exists j \text{ s.t. } \forall k \geq j \text{ and } \forall t \in [0,n], \ \ \dG(\tilde{\gamma}_{k}(t), \gamma(t))< \epsilon.$$
This means that any vertex $\gamma(n)$ is arbitrarily close to a proximal point, but since we are in a graph, taking $\epsilon$ small enough means that the proximal point and $\gamma(n)$ are in fact the same vertex.\\

To prove that $\gamma$ represents the same point of $u\pi_{h}$, we argue as before: we use the \hyperref[Conj3]{Hooking Lemma} to bound the distance between an element $\widehat{x}_{k} \in \T(u_{k})$ and a nearest neighbor $\overline{x}_{k}$ of $u_{k}$; then we apply Lemma~\ref{visible_proximal_b} to say that $\overline{x}_{k}$ and $\gamma(k)$ are $(4\delta+2)$-near. And we conclude with a triangle inequality that yields $\dG(\gamma(k), \widehat{x}_{k}) \leq 4\delta+2+\lambda_{a}$.
\end{proof}

\begin{defn}
\label{proximal_ray_def}
Let $u$ be a horofunction. We call the geodesic ray $\gamma_{u}$ defined by Proposition~\ref{proximal_ray_prop} a \textbf{proximal ray} with respect to $u$.
\end{defn}

We know that, in some sense, the Gromov product of two points in $\partial \Gamma$ measures how long two geodesic rays representing these points fellow travel. In the following result, we will provide an atom-coding version of this fact.
\begin{thm}
\label{SConj1}
Let $u=(u_{k})_{k=1}^{\infty}$ and $v=(v_{k})_{k=1}^{\infty}$ be two horofunctions coded by atoms. Then the following holds.
\begin{enumerate}[label=(\alph*),ref=\ref{SConj1}(\alph*)]
\item If $\T\dG(u_{i},v_{i}) \leq C$ for all $i \leq k$, then $(u \pi_{h} \mid v \pi_{h}) \geq k-C'$,\label{SConj1_a}
\item if $(u \pi_{h} \mid v \pi_{h}) \geq k$, then  $\T \dG(u_{i},v_{i}) \leq C''$ for all $i \leq k$, \label{SConj1_b}
\end{enumerate} 
where $C$, $C'$ and $C''$ are constants. Furthermore, if $C$
 depends only on $\lambda_{a}$ and $\delta$, so do $C'$ and $C''$.
\end{thm}
In particular, we obtain a new proof of Theorem~\ref{Conj1} when the hypotheses hold for each $k$.
\begin{proof}
For the first assertion, we proceed as in the proof that $\pi_{h}$ is Lipschitz (see Proposition~\ref{lipschitz}). When $l \leq k$, we have $(u_{l} \mid v_{l} ) \geq l-C$ by hyphotesis. If $l > k$, we consider the following triangle inequality
$$\dG(\widehat{x}_{l},\widehat{y}_{l}) \leq \dG(\widehat{x}_{l},\widehat{x}_{k})+\dG(\widehat{x}_{k},\widehat{y}_{k})+\dG(\widehat{y}_{k},\widehat{y}_{l})$$
with $\widehat{x}_{i} \in T(u_{i})$ and $\widehat{y}_{i} \in T(v_{i})$ for $i=l,k$.\\
Then we use the argument in Remark~\ref{proximal_trick} to get two geodesics of proximal points and to give an estimate of $\dG(\widehat{x}_{l},\widehat{x}_{k})$ and $\dG(\widehat{y}_{l},\widehat{y}_{k})$, so that we have
$$\T \dG(u_{l},v_{l}) \leq 2(l-k)+4(\lambda_{a}+4\delta+2)+\dG(\widehat{x}_{k},\widehat{y}_{k}).$$
To conclude, we use again the hypothesis that $\dG(\widehat{x}_{k},\widehat{y}_{k}) \leq C$ and hence $ \T \dG(u_{l},v_{l}) \leq
 2(l-k)+4(\lambda_{a}+4\delta+2)+C$. In this way, we found a lower bound for each Gromov product of the type $(u_{l} \mid v_{l})$ (we implicitly use the \hyperref[Conj3]{Hooking Lemma} and the claim follows).\\
What is left to point out is that $(u \mid v) \leq (u \pi_{h} \mid v \pi_{h}) +2\tilde{\delta}$ due to Remark~\ref{Vaisala_Gromov_product_c}.
\\

For the second assertion, we set $\gamma_{u}$ and $\gamma_{v}$ to be to proximal rays with respect to $u$ and $v$ (see Definition~\ref{proximal_ray_def} and Proposition~\ref{proximal_ray_prop}). We consider
$$\T \dG(u_{k},v_{k}) \leq \dG(T(u_{k}),\gamma_{u}(k))+\dG(\gamma_{u}(k),\gamma_{v}(k))+\dG(\gamma_{v}(k),T(v_{k})).$$
The distance $\dG(T(u_{k}),\gamma_{u}(k))$ is bounded by $4\delta+2+\lambda_{a}$ for the same argument as before, that is taking a nearest neighbor and exploiting Lemma~\ref{visible_proximal_b} together with the \hyperref[Conj3]{Hooking Lemma}. The same occurs to $\dG(T(v_{k}),\gamma_{v}(k))$. The distance $\dG(\gamma_{u}(k),\gamma_{v}(k))$ is bounded by virtue of Lemma~\ref{tripod}. More clearly if $(u \pi_{h} \mid v \pi_{h}) \geq k$ it means that the distance between $\dG(\gamma_{u}(i),\gamma_{v}(i))$ for all $i \leq k$ is bounded by $5\delta$ due to the quasi-isometry.
\end{proof}

This version of the theorem gives a chance of characterizing the Gromov boundary via horofunctions through a metric viewpoint.

\begin{cor}
The function $\beta^{-(\cdot \mid \cdot )}$ on $\partial_{h} \Gamma$ is a distance and if $\curlyvee_{h}$ is the pseudo-metric computed using the First Move, then the quotient $\partial_{h} \Gamma / \curlyvee_{h}$ defined by the Second Move is $\partial \Gamma$.
\end{cor}
\noindent We recall that the two moves are explained after the Definition~\ref{metric} at the very beginning of the dissertation.
\begin{proof}
We start by proving that the function is a distance. Symmetry is obvious. Taken $u \in \partial_{h} \Gamma$, then 
$$\max_{x_{k} \in T(u_{k}), y_{k} \in T(v_{k})} \liminf_{k } (x_{k} | y_{k})= \infty$$
and so is $(u \mid u)$, by virtue of the proof of Lemma~\ref{Gromov_atoms}. Hence $\beta^{-(u,u)}=0$.\ 
Using the First Move, we get the pseudo-metric $\curlyvee$.\
It remains to prove that the metric quotient is the Gromov boundary. We point out that two horofunctions $u \neq v$ can satisfy $\beta^{-(u,v)}=0$ and indeed this happens if they glue (they are in the same fiber of $\pi_{h}$). We need to show that this is the only possible case, which means that if $\beta^{-(u,v)}=0$, then $u$ and $v$ glue. But now $(u \mid v)= \infty$, 
and so there exists a couple of Gromov sequences converging to $u$ and $v$ such that their Gromov product is infinite, hence these two sequences are the same element in $\partial \Gamma$. 
\end{proof}

In literature, there are many examples of metric spaces (or similar structures) that in some way converge to the Gromov boundary of a hyperbolic group. We cite as an example, the work of Pawlik \cite{P} and Lemma 3.8 of \cite{GMS} which says that spheres with center in a distinguished point $x_{0}$ and endowed with the visual metric \textit{weakly converge to $\partial \Gamma$ in the sense of Gromov-Hausdorff}. We can look at the tips as a coarse version of spheres and  so our aim now is to provide a tip-version of this convergence.\\
In the following discussion, we will adopt the notation  $\curlyvee_{k}(u_{k},v_{k})=\beta^{-(u_{k}\mid v_{k})}$ in which $u_{k}$ and $v_{k}$ are atoms of the same level. We recall that the Gromov product is the one defined right before Lemma~\ref{Gromov_atoms} and that $\curlyvee_{k}$ is not a metric (not even a distance), but still plays an important role in the theory.  With this in mind, we will consider the weak Gromov-Hausdorff limit of $(\mathcal{A}_{k}, \curlyvee_{k})$ as $k$ goes to infinity even if they are not metric spaces. \\

A formal definition for the limit we discussed is the following

\begin{defn}\label{wGH}
Let $\{ \Gamma_{i}, \dist_{i} \}_{i \in \mathbb{N}}$ be a sequence of graphs endowed with the standard metric. We say that the graph $(\Gamma,\dist)$ is the \textbf{weak Gromov-Hausdorff limit}, or that the sequence weakly converges in the sense of Gromov-Hausdorff, if for all $i \in \mathbb{N}$ there exists a quasi-isometry $\phi_{i}: \Gamma \rightarrow \Gamma_{i}$ with $L_{1}^{i}$ not depending on $i$ and $L_{2}^{i}$ that goes to zero as $i$ tends to infinity, where $L_{1}^{i}$ is the multiplicative constant and $L_{2}^{i}$ is the additive constant of the quasi-isometric embedding. 
\end{defn}
\noindent Note that this is a coarse version of the standard notion of Gromov-Hausdorff convergence in metric geometry (see e.g.\ \cite{BBI}).\\

Before proving the result, we need a technical lemma that links the Gromov product between two atoms with the one between two proximal points:

\begin{lem}
\label{CGP}
Let $u_{k}$ and $v_{k}$ be two $k$-level atoms. Let $p_{k} \in P(u_{k},S_{k})$ and $q_{k} \in P(v_{k}, S_{k})$. Then 
$$|\T \dG(u_{k},v_{k})-\dG(p_{k},q_{k})|\leq 8\delta+4+2\lambda_{a} \text{ and } |(u_{k}|v_{k})-(p_{k}|q_{k})|\leq 4\delta+2+2\lambda_{a} .$$
\end{lem}
\begin{proof}
The first part is the usual consequence of the \hyperref[Conj3]{Hooking Lemma} together with Lemma~\ref{visible_proximal_b} applied to $$\T \dG(u_{k},v_{k}) \leq \dG(T(u_{k}),p_{k})+ \dG(p_{k},q_{k})+\dG(q_{k},T(v_{k}))$$ and the triangle inequality where proximal points and atoms are switched.\\

For the second part, we use the first part as follows
$$2(u_{k} \mid v_{k}) \geq 2k-\dG(p_{k},q_{k})-8\delta-4-2\lambda_{a},$$
and
$$2(u_{k} \mid v_{k}) \leq 2k+2\lambda_{a}-\dG(p_{k},q_{k})+8\delta+4+2\lambda_{a}.$$
All is left is to notice that $2k-\dG(p_{k},q_{k})=2(p_{k} \mid q_{k}) $.
\end{proof} 

\begin{prop}
The metric space $(\partial \Gamma, \curlyvee)$ is the weak Gromov-Hausdorff limit of the sequence $(\mathcal{A}_{k}, \curlyvee_{k})$ as $k \rightarrow \infty$.
\end{prop}
\begin{proof}
First, we introduce the map we claim induces the quasi-isometry. 
\begin{align*}
\Phi: \partial \Gamma \rightarrow T(\mathcal{A}_{k}) \\
     x_{\infty} \mapsto u \mapsto T(u_{k})
\end{align*}

We consider an element $x_{\infty} \in \partial \Gamma$ and a section $\mathcal{S}_{h}: \partial \Gamma \rightarrow \partial_{h} \Gamma$ of the projection $\pi_{h}$. Now $x_{\infty} \mathcal{S}_{h}=u$ and $u_{k}$ is the $k$-level atom of the atom-coding. So $x_{\infty} \Phi:=T(u_{k})$.\\

\textit{Quasi-dense image.} Let $a \in \mathcal{A}_{k}$. Take $x_{\infty} \in \partial a \pi_{h}$
and evaluate $T^{*} \dG(x_{\infty}\Phi,a)$. Since there exists a horofunction that passes through $a$ and projects onto $x_{\infty}$, we know that such a horofunction and $x_{\infty}\mathcal{S}_{h}$ identify on $\partial \Gamma$. Hence by Theorem~\ref{Conj1}, the tips of their $k$-th terms of the atom-codings (which are $a$ and $x_{\infty}\Phi$) stay within $2(\lambda_{a}+\delta)$ .\\

\textit{Quasi-isometric embedding.} We will proceed by cases.\\
If $(x_{\infty}\mid y_{\infty}) \geq k$, we apply Theorem~\ref{SConj1_b} and we have $\T \dG(u_{i},v_{i}) \leq C$ for all $i \leq k$ and hence $(u_{k}|v_{k}) \geq k-C$ by applying Theorem~\ref{SConj1_a}. To conclude then that $|\curlyvee_{k}(u_{k},v_{k})-\curlyvee(x_{\infty},y_{\infty})|\leq \curlyvee_{k}(u_{k},v_{k}) \leq \beta^{-k+C}$.\\
If $(x_{\infty}\mid y_{\infty}) \leq k$, then we consider two proximal rays $\gamma \in x_{\infty}$ and $\eta \in y_{\infty}$ and by virtue of Lemma~\ref{tripod}, we have $$|(\gamma(k)|\eta(k))-(x_{\infty}|y_{\infty})| \leq \dfrac{5}{2} \delta.$$ 
All that is left to do is combine it with Lemma~\ref{CGP} and get 
\begin{align*}|(u_{k}|v_{k})-(x_{\infty}|y_{\infty})| \leq & |(u_{k}|v_{k})-(\gamma(k)|\eta(k))|+ \\
& |(\gamma(k)|\eta(k))-(x_{\infty}|y_{\infty})| \leq 4\delta+2+2\lambda_{a}+\dfrac{5}{2} \delta.\end{align*}
Hence $D^{-1}\curlyvee(x_{\infty},y_{\infty}) \leq \curlyvee_{k}(u_{k},v_{k})\leq D \curlyvee(x_{\infty},y_{\infty})$ for a suitable constant $D$ as desired.
\end{proof}

\section{Quasi-isometries} \label{6}
We now continue our parallelism between the graph $\Gamma$, its spheres, its geodesic rays and the atom-coding tree $\mathcal{A}$, the tips and the horofunctions; we will now present a couple of quasi-isometries between the set of tips and the graph $\Gamma$.\\

The following result will help us restricting our attention to elements with infinite cones in both the quasi-isometries we are going to describe.
\begin{lem}
\label{bounded_finite_cone}
Let $\Gamma$ be a hyperbolic graph quasi-isometric to the Cayley graph of some hyperbolic group. Then there exists a constant $\lambda_{\infty}$ such that every element $x \in \Gamma$ with a finite cone is in a ball of radius $\lambda_{\infty}$ centered at an element with an infinite cone.
\end{lem}
From now on $\lambda_{\infty}$ will be such a constant.
\begin{proof}
Let us start by determining $\lambda_{\infty}$. Since two finite cones with the same type have the same number of points, we can consider $\lambda_{\infty}$ to be the maximum of the cardinalities among all types of finite cones (they are finite by Proposition~\ref{cone_type_finite}). This allows us to find a predecessor $x^{c}$ of $x$ (i.e.\ an element that belongs to a geodesic between $x_{0}$ and $x$), such that $\dG(x^{c},x) \leq \lambda_{\infty}$ that has an infinite cone. Indeed, suppose that every element that belongs to a geodesic $[x^{c},x]$ with $x^{c}$ a predecessor and $\dG(x^{c},x) = \lambda_{\infty}$ has a finite cone. This means that the geodesic is fully contained in the cone $C(x^{c})$ but exceeds the number of possible elements in the cone. Hence we have a contradiction. This implies that the cone $C(x^{c})$ has to be infinite.
\end{proof}

The following is useful for proving the quasi-density in both cases.
\begin{lem}
\label{almost_all_visual}
Let $\Gamma$ be a hyperbolic graph. If $S_{n}^{\infty}$ is the subset of all elements in $S_{n}$ such that their cones are infinite, then $S_{n}^{\infty} \subseteq \displaystyle \bigcup_{a \in \mathcal{A}_{n}} V(a,B_{n})$.
\end{lem}
\begin{proof}
Let $x'$ be an element of $S_{n}^{\infty}$. Now $C(x')$ is infinite and there are finitely many atoms of level $n$, so there exists an atom  $a \in \mathcal{A}_{n}$ such that $C(x') \cap a \neq \emptyset$. We take $y \in C(x') \cap a$ so that $[x',y] \cap B_{n}=x'$ and by definition $x' \in V(y,B_{n})$. By the property of visible points (see Lemma~\ref{visible_proximal_c}), we have $V(y,B_{n})=V(a,B_{n})$.
\end{proof}

Throughout Section~\ref{3}, we were dealing with many distances 
(almost all of them were not metrics) and we studied the connections between them. We then proved two quasi-isometry like results (Proposition~\ref{tips_non-tips_relation} and Proposition~\ref{quasi_iso_like_hausdorff}). As a first step, we now want to formally prove what these results naturally suggest.

\begin{prop}\label{quasi_iso_haus}
Let $(\mathbf{T},\Haus)$ be the set of tips without repetitions (that means that if two atoms share the same tip, we count it once) endowed with the usual Hausdorff metric. Then $\mathbf{T}$ is quasi-isometric to $\Gamma$. 
\end{prop}
\noindent We recall the discussion made in Remark~\ref{equal_tips} to better understand what ``without repetitions'' means.
\begin{proof}
The map $\mathcal{S}_{\mathbf{T}}: \mathbf{T} \rightarrow \Gamma$ we want to show is a quasi isometry is defined as $T(a) \mathcal{S}_{\mathbf{T}}:=\widehat{x}$ with $\widehat{x}$ some fixed element in $ T(a)$. \

\textit{Quasi-dense image.} Let $x \in \Gamma$. We consider $n$ such that $x \in B_{n-1}$. \\
First, suppose that $ x \in S_{n-1}^{\infty}$, then there exists at least one of its successor $x'$ that belongs to $S_{n}^{\infty}$. By Lemma~\ref{almost_all_visual}, we know that  $x'$ is a visible point for some atom $a \in \mathcal{A}_{n}$. Now we take $\overline{x} \in N(a, B_{n})$ and we know that $\dG(x',\overline{x}) \leq 4\delta+2$ by a property of visible points (by combining Lemma~\ref{visible_proximal_a} with Lemma~\ref{visible_proximal_b}). So $\dG(x,\overline{x}) \leq 4\delta+3$ and by the \hyperref[Conj3]{Hooking Lemma} we get $\dG(x,\widehat{x}) \leq 4\delta+3+\lambda_{a}$ for any point $\widehat{x} \in T(a)$.\\
Now, suppose that the cone of $x$ is finite. Combining Lemma~\ref{bounded_finite_cone} with the first case, we get $\dG(x,T(a)) \leq \lambda_{\infty}+4\delta+3\lambda_{a}$.\

\textit{Quasi-isometric embedding.} We argue as in the proof of Proposition~\ref{quasi_iso_like_hausdorff}.\\
Indeed, we already know that $\Haus(T(a),T(b)) \leq \displaystyle \min_{\widehat{x} \in T(a), \widehat{y} \in T(b)} \dG(\widehat{x}, \widehat{y})+2\lambda_{a}$ and by Proposition~\ref{diam} we get $$\Haus(T(a),T(b)) \leq \dG(T(a)\mathcal{S}_{\mathbf{T}},T(b)\mathcal{S}_{\mathbf{T}})+4\lambda_{a}.$$ 

\noindent On the other side, Proposition~\ref{quasi_iso_like_hausdorff} provides $$\displaystyle \min_{\widehat{x} \in T(a), \widehat{y} \in T(b)} \dG(\widehat{x}, \widehat{y}) \leq \Haus(T(a),T(b)).$$ 
\noindent All we have to do is combine it with $$\dG(T(a)\mathcal{S}_{\mathbf{T}},T(b)\mathcal{S}_{\mathbf{T}}) \leq \displaystyle \min_{\widehat{x} \in T(a), \widehat{y} \in T(b)} \dG(\widehat{x}, \widehat{y})+\lambda_{a}$$ coming from Proposition~\ref{diam}. 
\end{proof}

The second step of the section is inspired by the work of Kaimanovich (\cite{K}) on fractals and of Nekrashevych (\cite{N}, \cite{BGN}) on limit spaces of contracting self-similar groups. See also \cite{LW} for an application in dynamical systems.\\

The main object of the discussion is:
\begin{defn}\label{graph_of_atoms}
Let $\Gamma$ be a hyperbolic graph and $\mathcal{A}$ its tree of atoms. We define $\Gamma_{\mathcal{A}}$ and we call it the \textbf{graph of atoms} in the following way:
\begin{description}
\item[Vertices] all elements of $\mathcal{A}$;
\item[Vertical Edges] given two vertices $a_{n} \in \mathcal{A}_{n}$ and $a_{n+1} \in \mathcal{A}_{n+1}$, there exists an edge if and only if $a_{n} \supseteq a_{n+1}$;
\item[Horizontal Edges] given two vertices $a_{n}, b_{n} \in \mathcal{A}_{n}$ there exists an edge if and only if $\dG(a_{n}, b_{n}) \leq 2(\lambda_{\infty}+4\delta+\lambda_{a})+7$ and define this as $\lambda_{e}$.
\end{description}
\end{defn}
As before, $\lambda_{e}$ will be such a constant. Its peculiar definition will be clarified during the proof of the quasi-isometry result.

\begin{defn}
We denote by $T(\Gamma_{\mathcal{A}})$ the same construction as before, but using the distance $\T\dG$ and the constant $ 2\lambda_{e}+4\delta+2\lambda_{a}$ for horizontal edges. We call it the \textbf{graph of tips}. Please note that the vertices of the graph are still atoms.
\end{defn}

Despite the choice of the constant looking strange, it is related to the fact that $\T\dG \leq 2\dG +4\delta+2\lambda_{a}$ (see Proposition~\ref{tips_non-tips_relation}) and will be fully explained in the following.
\begin{oss}\label{properties_graph_atoms}
Some straightforward properties of $\Gamma_{\mathcal{A}}$ are the following
\begin{enumerate}[label=(\alph*),ref=(\alph*)]
\item the tree $\mathcal{A}$ is a spanning tree for the graph;
\item the vertices of the $n$-sphere $(\Gamma_{\mathcal{A}})_{n}$ are in bijection with $\mathcal{A}_{n}$;
\item the projection $\pi_{n}: (\Gamma_{\mathcal{A}})_{n} \twoheadrightarrow (\Gamma_{\mathcal{A}})_{n-1}$ is well-defined;
\item the graph is locally finite, indeed $\mathcal{A}$ is locally finite and horizontal edges starting from a vertex are finite due to the same argument that proves the fibers of $\pi_{h}$ are finite (see Theorem~\ref{finite_to_one}).
\item \label{properties_graph_atoms_5} the graph of atoms is a subgraph of the graph of tips, since $\dG(a,b) \leq \lambda_{e}$ implies $\T\dG(a,b) \leq 2\lambda_{e} +4\delta+2\lambda_{a}$. In particular, we have $\dist_{T(\Gamma_{\mathcal{A}})}(a,b) \leq \dist_{\Gamma_{\mathcal{A}}}(a,b)$.   
\item $\Gamma_{\mathcal{A}}$ is an \textit{augmented tree} in the sense of \cite{K}. Roughly speaking, an augmented tree is a graph constructed starting from a tree where we add edges between some vertices on the same level with the condition that if $x$ and $y$ are two vertices that share such an edge and $\tilde{x} \in [x_{0},x]$, $\tilde{y} \in [x_{0},y]$ are at the same level, then $\tilde{x}=\tilde{y}$ or they share an edge  (recall that $x_{0}$  is the root). 
\end{enumerate}
\end{oss}

\begin{thm}
\label{quasi_isometry}
Let $\Gamma$ be a hyperbolic graph. Then the graph of atoms $\Gamma_{\mathcal{A}}$ is quasi-isometric to $\Gamma$. In particular, it is hyperbolic and its boundary is homeomorphic to $\partial \Gamma$.
\end{thm}

\begin{proof}
Our quasi-isometry candidate map $\mathcal{S}_{\mathcal{A}}: \Gamma_{\mathcal{A}} \rightarrow \Gamma$ is defined as $a \mathcal{S}_{\mathcal{A}}:= \widehat{x}$ with $\widehat{x} $ a fixed element in $T(a)$.\

\textit{Quasi-dense image.} This argument is exactly the same as Proposition~\ref{quasi_iso_haus} as both maps are defined from tips to element in $\Gamma$.

\textit{Quasi-isometric embedding.} Let $a \mathcal{S}_{\mathcal{A}}=\widehat{x}$ and $b\mathcal{S}_{\mathcal{A}}=\widehat{y}$.\\
We start by proving that $\dG(\widehat{x},\widehat{y}) \leq M \dist_{\Gamma_{\mathcal{A}}}(a,b)$ for some constant $M$.

Let $a=a_{0},a_{1}, \ldots, a_{l}=b$ be a geodesic between $a$ and $b$ in $T(\Gamma_{\mathcal{A}})$. If $\{ a_{i},a_{i+1} \}$ is vertical, then we recall the bound on consecutive tips (see Proposition~\ref{iterated_hooking}). If  $\{ a_{i},a_{i+1} \}$ is horizontal, by definition we have a bound between two of their elements. We need to pay attention: we have a bound for every distance outside the atoms, but we also need a bound for what happens inside the tips, so that an element involved in the bound for the edge $\{a_{i-1},a_{i}\}$ is at a reasonable distance from an element involved in $\{a_{i}, a_{i+1} \}$. This internal bound follows from the fact that a tip has diameter at most $2\lambda_{a}$ (see Proposition~\ref{diam}). Let $D$ be the maximum between the two external bounds, more explicitly if $D'$ is the bound provided for vertical edges by Proposition~\ref{iterated_hooking} and $D''$ is the bound coming from the definition of horizontal edge, then $D= \max \{ D',D''\}$.
If we put all together, we have

$$\dG(\widehat{x},\widehat{y}) \leq \sum_{i=0}^{l-1} \dG(\widehat{x}_{i},\widehat{x}_{i+1})+ \sum_{i=1}^{l-1} \diam T(a_{i}) \leq  Dl+2\lambda_{a}(l-1) \leq M\dist_{T(\Gamma_{\mathcal{A}})}(a,b)$$

with $\widehat{x}_{0}=\widehat{x}$, $\widehat{x}_{l}=\widehat{y}$ and $\widehat{x}_{i} \in T(a_{i})$  one of the two elements of the tips involved in the bound for the left and for the right edges. And $M=D+2\lambda_{a}$.

By part~\ref{properties_graph_atoms_5} of Remark~\ref{properties_graph_atoms}, we get the claim. \\
We now prove the other part, namely $\dist_{\Gamma_{\mathcal{A}}}(a,b) \leq W \dG(\widehat{x},\widehat{y})$ for some constant $W$. We proceed in the same way as before. We take a geodesic in $\Gamma$, explicitly $y_{0}=\widehat{x}, y_{1},y_{2}, \ldots, \widehat{y}=y_{l}$, between $\widehat{x}$ and $\widehat{y}$. By the quasi-density, we know that  for every point $y_{i}$ there exists a $n_{i}$-atom $a_{i}$ such that $\dG(y_{i},a_{i}) \leq \lambda_{\infty}+4\delta+3+\lambda_{a}$ and $\max \{0, \dG(x_{0},y_{i})-\lambda_{\infty} \} \leq n_{i} \leq \dG(x_{0},y_{i})$. This is due to the fact that either $y_{i}$ has a infinite cone, hence it is a visible point and the atom $a_{i}$ is at level $\dG(x_{0},y_{i})$ (see Lemma~\ref{almost_all_visual}) or $y_{i}$ has a finite cone, but there exists another element $y \in [x_{0},y_{i}]$ at a distance at most $\lambda_{\infty}$ (see Lemma~\ref{bounded_finite_cone}) that has an infinite cone and the associated atom $a_{i}$ is at level $\dG(x_{0},y)$.
Note that, in this way, two consecutive atoms are at a distance $\dG$ less than $\lambda_{e}$. \\
We want to prove that two consecutive atoms $a_{i}$ and $a_{i+1}$ have a distance in $\Gamma_{\mathcal{A}}$ bounded by some constant. So if they are at the same level, they are adjacent by the definition of horizontal edges. If they are on two different levels $n$ and $m$, then $|n-m| \leq \lambda_{\infty}+1$. Indeed, we combine the fact that  two consecutive vertices $y_{i}$ and $y_{i+1}$ are such that $|\dG(x_{0},y_{i})-\dG(x_{0},y_{i+1})| \leq 1$ (they are two consecutive points of a geodesic) and $\max \{0, \dG(x_{0},y_{i})-\lambda_{\infty} \} \leq n_{i} \leq \dG(x_{0},y_{i})$.
Now, we can assume without loss of generality that $m<n$. We denote with $a_{i}^{m}$ the $m$-atom such that $a_{i} \subseteq a_{i}^{m}$ and we have
$$\dG(a_{i}^{m},a_{i+1})=\min_{x \in a_{i}^{m}, y \in a_{i+1}} \dG(x,y) \leq \min_{z \in a_{i},y \in a_{i+1}} \dG(z,y) \leq \lambda_{e}.$$
This means that $a_{i}^{m}$ and $a_{i+1}$ are adjacent and so $\dist_{\Gamma_{\mathcal{A}}}(a_{i},a_{i+1}) \leq \dG(a_{i},a_{i}^{m}) + \dG(a_{i}^{m},a_{i+1}) \leq \lambda_{\infty}+1+1= \lambda_{\infty}+2$.
To conclude, for each edge of the geodesic in $\Gamma$, we have constructed a geodesic in $\Gamma_{\mathcal{A}}$ of length at most $\lambda_{\infty}+2$, hence $\dist_{\Gamma_{\mathcal{A}}}(a,b) \leq W \dG(\widehat{x},\widehat{y})$ with $W=\lambda_{\infty}+2$.
\end{proof}

\section{Rational gluing of horofunctions} \label{7}
The goal of this section is to construct a machine that can tell if two elements of $\partial_{h} G$, represented by their atom-codings, are in the same $\pi_{h}$-fiber.\\

We start by setting $\Sigma$ as our finite \textbf{alphabet}, that is a finite collection of symbols $\sigma \in \Sigma$. Since we need more alphabets at the same time, we will use also $\tilde{\Sigma}$, $\Xi$ and $\tilde{\Xi}$. From an alphabet $\Sigma$, we can construct two different objects: the collection of all finite strings $\Sigma^{*}$ and the collection of all infinite strings $\Sigma^{\omega}$. A language is a subcollection of $\Sigma^{*}$ or of $\Sigma^{\omega}$. We will usually deal with infinite strings, the reason is that $\Sigma^{\omega}$ is a Cantor set.\\

We also need to set some notations for machines. We recall that a \textit{partial function} is a binary relation between two sets that associates to every element of the first set at most one element of the second. 

\begin{defn}
A \textbf{synchronous deterministic finite state automaton} is a quadruple $(\Sigma, \Theta, \rightarrow, \theta_{0})$ with $\Theta$ a finite set called \textit{states}, $\theta_{0} \in \Theta $ called \textit{initial state} and a partial function between $\Theta \times \Sigma $ and $\Theta$ which is called \textit{transition function}.
\end{defn}

We recall that deterministic means $\rightarrow$ is an actual (partial) function, or that there cannot be two different transitions starting from a state and processing the same element of the alphabet, and it is synchronous because it processes one element of $\Sigma$ at each step. We will drop all the adjectives and we will simply call it an automata, this because it will be the only machine of this type.\\
We say that an automata \textbf{recognizes} a language when a string belong to the language if and only if it is processed by the automata. 

\begin{defn}\label{rational_subset}
If a language of infinite strings is recognized by an automata, it is called a \textbf{rational subset} of $\Sigma^{\omega}$.
\end{defn}

\begin{defn}
An \textbf{asynchronous deterministic finite state transducer} is a quintuple $(\Sigma, \Xi \cup \{ \varepsilon \}, \Theta, \rightarrow, out,  \theta_{0})$ with $\Theta$ a finite set called \textit{states}, $\theta_{0} \in \Theta $ called \textit{initial state}, a \textit{transition function} defined as $(\theta_{1},\sigma) \rightarrow \theta_{2}$ with $\theta_{1},\theta_{2} \in \Theta$ and $\sigma \in \Sigma$ and the \textit{output function} defined as $(\theta, \sigma)out=\xi_{1}\ldots\xi_{n}$ with $\theta \in \Theta$, $\sigma \in \Sigma$ and $\xi_{i} \in \Xi$ or $\xi_{i}= \varepsilon$ the \textbf{empty string}.
\end{defn}
\noindent As before, we will drop all the adjectives and we will simply call it a transducer.

We see the set of finite strings $\Sigma^{*}$ as a tree in the usual way (see Figure \ref{tree_language}). It follows immediately that the boundary of the tree is $\Sigma^{\omega}$ and hence the latter is a Cantor set.

\begin{figure}
\centering
\begin{tikzpicture}[thick]

   \fill[color=blue] (0,0) circle (0.4ex);
   \draw[  blue] (0,0) -- (190:2.5) node[midway, above] {\textcolor{blue}{$\sigma_{1}$}} ;

   \draw[  blue] (0,0) -- (270:1.7) node[midway, left] {\textcolor{blue}{$\sigma_{2}$}};
   \draw[  blue] (0,0) -- (350:2.5) node[midway, above] {\textcolor{blue}{$\sigma_{3}$}};
   
   \fill[color=blue] (190:2.5) circle (0.35ex);
   \draw[  blue] (190:2.5) -- +(210:1.25)node[midway, above] {\footnotesize\textcolor{blue}{$\sigma_{1}$}};
   \draw[  blue] (190:2.5) -- +(270:0.75)node[near end, left] {\footnotesize\textcolor{blue}{$\sigma_{2}$}};
   \draw[  blue] (190:2.5) -- +(330:1.25)node[midway, above] {\footnotesize\textcolor{blue}{$\sigma_{3}$}};
   
   \fill[color=blue] (190:2.5)++(210:1.25) circle (0.3ex);
   \draw[  blue] (190:2.5)++(210:1.25) -- +(324:0.5) node[above] {\footnotesize\textcolor{blue}{$\sigma_{3}$}};
   \draw[  blue] (190:2.5)++(210:1.25) -- +(270:0.5)node[below] {\footnotesize\textcolor{blue}{$\sigma_{2}$}};
   \draw[  blue] (190:2.5)++(210:1.25) -- +(216:0.5) node[above] {\footnotesize\textcolor{blue}{$\sigma_{1}$}};
   \fill[color=blue] (190:2.5)++(210:1.25)++(324:0.5) circle (0.3ex);
   \fill[color=blue] (190:2.5)++(210:1.25)++(270:0.5) circle (0.3ex);
   \fill[color=blue] (190:2.5)++(210:1.25)++(216:0.5) circle (0.3ex);
   
   \fill[color=blue] (190:2.5)++(270:0.75) circle (0.3ex);
   \draw[  blue] (190:2.5)++(270:0.75) -- +(324:0.5) node[below] {\footnotesize\textcolor{blue}{$\sigma_{3}$}};
   \draw[  blue] (190:2.5)++(270:0.75) -- +(270:0.5)node[below] {\footnotesize\textcolor{blue}{$\sigma_{2}$}};
   \draw[  blue] (190:2.5)++(270:0.75) -- +(216:0.5)node[below] {\footnotesize\textcolor{blue}{$\sigma_{1}$}};
   \fill[color=blue] (190:2.5)++(270:0.75)++(324:0.5) circle (0.3ex);
   \fill[color=blue] (190:2.5)++(270:0.75)++(270:0.5) circle (0.3ex);
   \fill[color=blue] (190:2.5)++(270:0.75)++(216:0.5) circle (0.3ex);
  
  \fill[color=blue] (190:2.5)++(330:1.25) circle (0.3ex);
  \draw[  blue] (190:2.5)++(330:1.25) -- +(324:0.5)node[above] {\footnotesize\textcolor{blue}{$\sigma_{3}$}};
   \draw[  blue] (190:2.5)++(330:1.25) -- +(270:0.5)node[below] {\footnotesize\textcolor{blue}{$\sigma_{3}$}};
   \draw[  blue] (190:2.5)++(330:1.25) -- +(216:0.5)node[above] {\footnotesize\textcolor{blue}{$\sigma_{1}$}};
   \fill[color=blue] (190:2.5)++(330:1.25)++(324:0.5) circle (0.3ex);
   \fill[color=blue] (190:2.5)++(330:1.25)++(270:0.5) circle (0.3ex);
   \fill[color=blue] (190:2.5)++(330:1.25)++(216:0.5) circle (0.3ex);

  \fill[color=blue] (350:2.5) circle (0.35ex);
   \draw[  blue] (350:2.5) -- +(210:1.25)node[midway, above] {\footnotesize\textcolor{blue}{$\sigma_{1}$}};
   \draw[  blue] (350:2.5) -- +(270:0.75)node[near end, left] {\footnotesize\textcolor{blue}{$\sigma_{2}$}};
   \draw[  blue] (350:2.5) -- +(330:1.25)node[midway, above] {\footnotesize\textcolor{blue}{$\sigma_{3}$}};
   
   \fill[color=blue] (350:2.5)++(210:1.25) circle (0.3ex);
   \draw[  blue] (350:2.5)++(210:1.25) -- +(324:0.5)node[above] {\footnotesize\textcolor{blue}{$\sigma_{3}$}};
   \draw[  blue] (350:2.5)++(210:1.25) -- +(270:0.5)node[below] {\footnotesize\textcolor{blue}{$\sigma_{2}$}};
   \draw[  blue] (350:2.5)++(210:1.25) -- +(216:0.5)node[above] {\footnotesize\textcolor{blue}{$\sigma_{1}$}};
   \fill[color=blue] (350:2.5)++(210:1.25)++(324:0.5) circle (0.3ex);
   \fill[color=blue] (350:2.5)++(210:1.25)++(270:0.5) circle (0.3ex);
   \fill[color=blue] (350:2.5)++(210:1.25)++(216:0.5) circle (0.3ex);
   
   \fill[color=blue] (350:2.5)++(270:0.75) circle (0.3ex);
   \draw[  blue] (350:2.5)++(270:0.75) -- +(324:0.5)node[below] {\footnotesize\textcolor{blue}{$\sigma_{3}$}};
   \draw[  blue] (350:2.5)++(270:0.75) -- +(270:0.5)node[below] {\footnotesize\textcolor{blue}{$\sigma_{2}$}};
   \draw[  blue] (350:2.5)++(270:0.75) -- +(216:0.5)node[below] {\footnotesize\textcolor{blue}{$\sigma_{1}$}};
   \fill[color=blue] (350:2.5)++(270:0.75)++(324:0.5) circle (0.3ex);
   \fill[color=blue] (350:2.5)++(270:0.75)++(270:0.5) circle (0.3ex);
   \fill[color=blue] (350:2.5)++(270:0.75)++(216:0.5) circle (0.3ex);
  
  \fill[color=blue] (350:2.5)++(330:1.25) circle (0.3ex);
  \draw[  blue] (350:2.5)++(330:1.25) -- +(324:0.5)node[above] {\footnotesize\textcolor{blue}{$\sigma_{3}$}};
   \draw[  blue] (350:2.5)++(330:1.25) -- +(270:0.5)node[below] {\footnotesize\textcolor{blue}{$\sigma_{2}$}};
   \draw[  blue] (350:2.5)++(330:1.25) -- +(216:0.5)node[above] {\footnotesize\textcolor{blue}{$\sigma_{1}$}};
   \fill[color=blue] (350:2.5)++(330:1.25)++(324:0.5) circle (0.3ex);
   \fill[color=blue] (350:2.5)++(330:1.25)++(270:0.5) circle (0.3ex);
   \fill[color=blue] (350:2.5)++(330:1.25)++(216:0.5) circle (0.3ex);

   \fill[color=blue] (270:1.7) circle (0.35ex);
   \draw[  blue] (270:1.7) -- +(210:1.25)node[midway, above] {\footnotesize\textcolor{blue}{$\sigma_{1}$}};
   \draw[  blue] (270:1.7) -- +(270:0.75)node[near end, left] {\footnotesize\textcolor{blue}{$\sigma_{2}$}};
   \draw[  blue] (270:1.7) -- +(330:1.25)node[midway, above] {\footnotesize\textcolor{blue}{$\sigma_{3}$}};
   
   \fill[color=blue] (270:1.7)++(210:1.25) circle (0.3ex);
   \draw[  blue] (270:1.7)++(210:1.25) -- +(324:0.5)node[above] {\footnotesize\textcolor{blue}{$\sigma_{3}$}};
   \draw[  blue] (270:1.7)++(210:1.25) -- +(270:0.5)node[below] {\footnotesize\textcolor{blue}{$\sigma_{2}$}};
   \draw[  blue] (270:1.7)++(210:1.25) -- +(216:0.5)node[above] {\footnotesize\textcolor{blue}{$\sigma_{1}$}};
   \fill[color=blue] (270:1.7)++(210:1.25)++(324:0.5) circle (0.3ex);
   \fill[color=blue] (270:1.7)++(210:1.25)++(270:0.5) circle (0.3ex);
   \fill[color=blue] (270:1.7)++(210:1.25)++(216:0.5) circle (0.3ex);
   
   \fill[color=blue] (270:1.7)++(270:0.75) circle (0.3ex);
   \draw[  blue] (270:1.7)++(270:0.75) -- +(324:0.5)node[below] {\footnotesize\textcolor{blue}{$\sigma_{3}$}};
   \draw[  blue] (270:1.7)++(270:0.75) -- +(270:0.5)node[below] {\footnotesize\textcolor{blue}{$\sigma_{2}$}};
   \draw[  blue] (270:1.7)++(270:0.75) -- +(216:0.5)node[below] {\footnotesize\textcolor{blue}{$\sigma_{1}$}};
   \fill[color=blue] (270:1.7)++(270:0.75)++(324:0.5) circle (0.3ex);
   \fill[color=blue] (270:1.7)++(270:0.75)++(270:0.5) circle (0.3ex);
   \fill[color=blue] (270:1.7)++(270:0.75)++(216:0.5) circle (0.3ex);
  
  \fill[color=blue] (270:1.7)++(330:1.25) circle (0.3ex);
  \draw[  blue] (270:1.7)++(330:1.25) -- +(324:0.5)node[above] {\footnotesize\textcolor{blue}{$\sigma_{3}$}};
   \draw[  blue] (270:1.7)++(330:1.25) -- +(270:0.5)node[below] {\footnotesize\textcolor{blue}{$\sigma_{2}$}};
   \draw[  blue] (270:1.7)++(330:1.25) -- +(216:0.5)node[above] {\footnotesize\textcolor{blue}{$\sigma_{1}$}};
   \fill[color=blue] (270:1.7)++(330:1.25)++(324:0.5) circle (0.3ex);
   \fill[color=blue] (270:1.7)++(330:1.25)++(270:0.5) circle (0.3ex);
   \fill[color=blue] (270:1.7)++(330:1.25)++(216:0.5) circle (0.3ex);
 
\end{tikzpicture}
\caption[Example of a language as a tree]{Strings of length $3$ based on the alphabet \\ $\Sigma=\{\sigma_{1},\sigma_{2},\sigma_{3}\}$ in their geometric representation.}
\label{tree_language}
\end{figure}
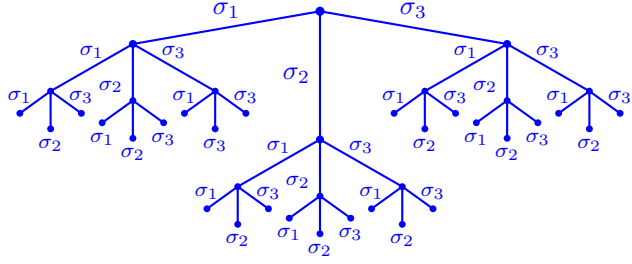

 We now want to define maps between Cantor sets by using transducers:
\begin{defn}
A map $\phi$ between two Cantor sets $\Sigma^{\omega}$ and $\Xi^{\omega}$ is called \textbf{rational} if there exists a transducer such that $w\phi=(\theta_{0},w)out$ for all $w \in \Sigma^{\omega}$.
\end{defn}
One can show that these maps form a category. Moreover, they are continuous with respect to the product topology and if they are bijective, then they are homeomorphisms (see \cite[Subsection 2.3]{GNS}).

\begin{defn}
A bijective rational map is called a \textbf{rational homeomorphism}.
\end{defn}
It is also true that the inverse of a rational homeomorphism is a rational homeomorphism. So, fixing an alphabet $\Sigma$, we can define the \textit{group of rational homeomorphisms} $\mathcal{R}$ over $\Sigma$. Since two different alphabets (with at least two elements) give two isomorphic groups, we refer to one of them simply as $\mathcal{R}$, without mentioning the underlying alphabet.\\

For a generalization of this setting to non-finite state machines and to better understand the topic see \cite{GNS}. Here, we just mention the following result as it will be useful later.

\begin{prop}[\cite{GNS}{, Proposition 2.11}]\label{characterization_rational}
A set $\mathcal{L} \subseteq \Xi^{\omega}$ is rational if and only if it is the image of a rational map $\phi: \overline{\Xi}^{\omega} \rightarrow \Xi^{\omega}$ with $\overline{\Xi}$ a finite alphabet. 
\end{prop}

A language is any subcollection of $\Sigma^{\omega}$, but we are interested in the rational ones. And we have geometric interpretation of $\Sigma^{*}$. We want to consider rooted trees that are not necessarily regular, but still have some nice properties. \\
Given a rooted tree $\mathcal{T}$ and a vertex $x$, we can consider the rooted subtree $\mathcal{T}_{x}$ such that the root is $x$ (see Figure~\ref{sub_tree})..
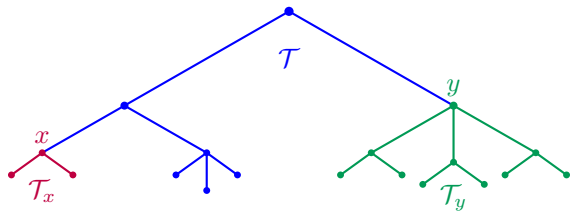
\begin{figure}
\centering
\begin{tikzpicture}[thick]

   \draw[  blue] (0,0) -- (210:2.5) node[midway, above] {} ;
   \draw[  blue] (0,0) -- (330:2.5) node[above, label={[xshift=0cm, yshift=-2cm]\textcolor{ForestGreen}{$\mathcal{T}_{y}$}}] {\textcolor{ForestGreen}{$y$}};
   \fill[color=blue] (0,0) circle (0.4ex) node[label={[xshift=0cm, yshift=-1cm]\textcolor{blue}{$\mathcal{T}$}}]{};

   \draw[  blue] (210:2.5) -- +(210:1.25)node[above, label={[xshift=0cm, yshift=-1.2cm]\textcolor{purple}{$\mathcal{T}_{x}$}}] {\textcolor{purple}{$x$}};
   \draw[  blue] (210:2.5) -- +(330:1.25)node[midway, above] {};
   \fill[color=blue] (210:2.5) circle (0.35ex);
   
   \fill[color=purple] (210:2.5)++(210:1.25) circle (0.3ex);
   \draw[  purple] (210:2.5)++(210:1.25) -- +(324:0.5) node[above] {};
   \draw[  purple] (210:2.5)++(210:1.25) -- +(216:0.5) node[above] {};
   \fill[color=purple] (210:2.5)++(210:1.25)++(324:0.5) circle (0.3ex);
   \fill[color=purple] (210:2.5)++(210:1.25)++(216:0.5) circle (0.3ex);
  
  \fill[color=blue] (210:2.5)++(330:1.25) circle (0.3ex);
  \draw[  blue] (210:2.5)++(330:1.25) -- +(324:0.5)node[above] {};
   \draw[  blue] (210:2.5)++(330:1.25) -- +(270:0.5)node[below] {};
   \draw[  blue] (210:2.5)++(330:1.25) -- +(216:0.5)node[above] {};
   \fill[color=blue] (210:2.5)++(330:1.25)++(324:0.5) circle (0.3ex);
   \fill[color=blue] (210:2.5)++(330:1.25)++(270:0.5) circle (0.3ex);
   \fill[color=blue] (210:2.5)++(330:1.25)++(216:0.5) circle (0.3ex);

   \fill[color=ForestGreen] (330:2.5) circle (0.35ex);
   \draw[  ForestGreen] (330:2.5) -- +(210:1.25)node[midway, above] {};
   \draw[  ForestGreen] (330:2.5) -- +(270:0.75)node[near end, left] {};
   \draw[  ForestGreen] (330:2.5) -- +(330:1.25)node[midway, above] {};

   \fill[color=ForestGreen] (330:2.5)++(210:1.25) circle (0.3ex);
   \draw[  ForestGreen] (330:2.5)++(210:1.25) -- +(324:0.5)node[above] {};
   \draw[  ForestGreen] (330:2.5)++(210:1.25) -- +(216:0.5)node[above] {};
   \fill[color=ForestGreen] (330:2.5)++(210:1.25)++(324:0.5) circle (0.3ex);
   \fill[color=ForestGreen] (330:2.5)++(210:1.25)++(216:0.5) circle (0.3ex);
   
   \fill[color=ForestGreen] (330:2.5)++(270:0.75) circle (0.3ex);
   \draw[  ForestGreen] (330:2.5)++(270:0.75) -- +(324:0.5)node[below] {};
   \draw[  ForestGreen] (330:2.5)++(270:0.75) -- +(216:0.5)node[below] {};
   \fill[color=ForestGreen] (330:2.5)++(270:0.75)++(324:0.5) circle (0.3ex);
   \fill[color=ForestGreen] (330:2.5)++(270:0.75)++(216:0.5) circle (0.3ex);
  
  \fill[color=ForestGreen] (330:2.5)++(330:1.25) circle (0.3ex);
  \draw[  ForestGreen] (330:2.5)++(330:1.25) -- +(324:0.5)node[above] {};
   \draw[  ForestGreen] (330:2.5)++(330:1.25) -- +(216:0.5)node[above] {};
   \fill[color=ForestGreen] (330:2.5)++(330:1.25)++(324:0.5) circle (0.3ex);
   \fill[color=ForestGreen] (330:2.5)++(330:1.25)++(216:0.5) circle (0.3ex);
\end{tikzpicture}
\caption[Examples of rooted subtrees.]{Two examples of rooted subtrees.}
\label{sub_tree}
\end{figure}

\begin{defn}
A \textbf{self-similar structure} on a rooted tree $\mathcal{T}$ is a partition of the vertices into finitely many classes together with a finite set of rooted tree isomorphisms between $\mathcal{T}_{x}$ and $\mathcal{T}_{y}$ for each couple of vertices $x$ and $y$ in the same class and where each isomorphism maps vertices of $\mathcal{T}_{x}$ to vertices of $\mathcal{T}_{y}$ in the same class.\\
Moreover, we want the isomorphisms to satisfy some natural conditions. They are closed under taking the inverse, composition and restriction. Namely, if $\varphi: \mathcal{T}_{x} \rightarrow \mathcal{T}_{y}$ and $\psi: \mathcal{T}_{y} \rightarrow \mathcal{T}_{z}$ belong to the self-similar structure, then $\varphi^{-1}$, $\varphi \psi$ and $\varphi_{|_{\mathcal{T}_{x'}}}: \mathcal{T}_{x'} \rightarrow \mathcal{T}_{x'\varphi} $ belong to the structure too. 
\end{defn}
\noindent A tree $\mathcal{T}$ with a self-similar structure is called \textbf{self-similar}. Note that the equivalence classes are originally called types, though we do not use that terminology, to avoid confusion.\

If we set $\Sigma$ to be the set of classes of a self-similar tree $\mathcal{T}$, 
it can be seen, but we will not give further details here, that there is a projection of the Gromov boundary $\partial \mathcal{T}$ 
onto a subset $\mathcal{L}$ of $\Sigma^{\omega}$. 
Furthermore, one can see that $\mathcal{L}$ is a rational subset and, in fact, that any rational subset can be characterized geometrically in this way (see e.g. \cite[Subsection 2.1]{BBM}). In order to make the projection bijective, which means that $\partial \mathcal{T}$ is a subset of $\Sigma^{\omega}$, we need to rely on a different coding that comes from \textit{rigid structures}. These are again  self-similar structures, but with some further hypothesis, and we will see that we can get a rigid one starting from a self-similar which is not rigid a priori. All the details are provided below.\\

Due to the following well-known property of Cantor sets

\begin{prop}[\cite{W}{, Theorem 30.7}]
Let $X$ be a compact metrizable space. Then there exists a continuous surjective map from a Cantor set onto $X$.
\end{prop}

Every compact metrizable space can be \textit{encoded} by a collection of infinite strings $\Sigma^{\omega}$ (it could be a proper subcollection of the all space as far as the map is still surjective). Namely, every element is represented by a (non-necessarily unique) infinite string. Such a string is called \textbf{coding}.\\

\begin{exm}
Set $\Sigma=\{ 0, 1, \ldots, 9\}$ and let $[0,1]$ be the unit interval. Then each element of $[0,1]$ has at most 2 codings due to the decimal expansion: exactly one if the number is irrational and exactly two if the number is rational. Namely, $\sigma_{1}\sigma_{2}\ldots\sigma_{n} \overline{0}$ and $\sigma_{1}\sigma_{2}\ldots(\sigma_{n}-1) \overline{9}$ where $\sigma_{n}-1$ is the difference \textit{mod 10} and $\overline{\sigma}=\sigma\sigma\ldots\sigma\ldots$.
\end{exm}

\begin{oss}
Despite the presence of a tree and choice of using the same word, to avoid ambiguity, it is worth mentioning that atom-coding are not coding in this sense. 
\end{oss}

In order to to construct a machine that can tell if two elements of $\partial_{h} G$, represented by their atom-codings, are in the same $\pi_{h}$-fiber, we give the following 
\begin{defn}
Let $\Sigma$ be a finite alphabet and $\Sigma^{\omega}$ its associate Cantor set of infinite strings. We say that an equivalence relation $\mathcal{G}$ on $\Sigma^{\omega}$ is \textbf{rational} if it is a rational subset of $\Sigma^{\omega} \times \Sigma^{\omega}$ in the sense of Definition~\ref{rational_subset}.
\end{defn}

Before proving that the gluing relation on atoms is rational, we want to present a few examples of the property and point out that the definition is in some way well-posed.

\begin{exm}
Gromov boundaries of hyperbolic groups can be seen as quotients of the Cantor set given by geodesic rays (see e.g. \cite{CP}).
The fact that the relation is rational is due to the fact that Gromov boundaries are \textit{semi-Markovian}. See \cite{CP} for definitions and for a proof in the torsion free case, and see \cite{P} for the connection between semi-Markovian and rational and for the groups with torsion.
\end{exm}

\begin{exm}
Limit spaces of contracting self similar groups (the gluing relation is given by the orbits of the action). See \cite{N2} for definitions and, in particular, Proposition 5.6 for the rationality.  
\end{exm}

\begin{exm}
Limit spaces of rearrangement groups of fractals (see \cite{BF}) seem to be natural candidates. We bring attention on the work of Donoven on the more general topic of invariant factors (\cite{D}). We are interested in Section 4.3, which is devoted to replacement systems. Even if his approach is similar, the question is still open.
\end{exm}

\begin{prop}
Let $\mathcal{G}$ be a rational equivalence relation on $\Sigma^{\omega}$. Then $\mathcal{G}$ is preserved by rational homeomorphisms of $\Sigma^{\omega}$.
\end{prop}
\noindent What we are going to do is to show that this is a direct consequence of Proposition~\ref{characterization_rational}.
\begin{proof}
To start, we observe that 
 $\Sigma^{\omega} \times \Sigma^{\omega}= (\Sigma \times \Sigma)^{\omega}$. Setting $\Xi=\Sigma \times \Sigma$, by virtue of Proposition~\ref{characterization_rational} we have a rational map $\phi: \overline{\Xi}^{\omega} \rightarrow \Xi^{\omega}$with image $\mathcal{G}$.
 Let $\psi: \Sigma^{\omega} \rightarrow \tilde{\Sigma}^{\omega}$  be a rational homeomorphism. We denote by $\psi \times \psi: \Xi^{\omega} \rightarrow \tilde{\Xi}^{\omega}$ with $\tilde{\Xi}= \tilde{\Sigma} \times \tilde{\Sigma}$ the map that is $\psi$ on each component. Since the composition of two rational maps is a rational map, it remains to prove that $\psi \times \psi$ is rational, hence the composition $\phi (\psi \times \psi)$ is rational as well. 

\[
\begin{tikzcd}[ampersand replacement=\&]
\Xi^{\omega} \arrow[r, "\psi \times \psi",rightarrow] \& \tilde{\Xi}^{\omega} \\
\overline{\Xi}^{\omega} \arrow[u, "\phi",rightarrow] \arrow[ur, "\phi(\psi \times \psi)"',dashrightarrow]
\end{tikzcd}
\]

\noindent If $(\Sigma, \tilde{\Sigma},\Theta, \rightarrow, out)$ is the defining transducer for $\psi$, we just create $$(\Xi, \tilde{\Xi}, \Theta \times \Theta, \rightarrow \times \rightarrow, out \times out)$$ such that if $(\sigma_{j},\theta_{j}^{i}) \rightarrow \theta_{j}^{o}$ with $j=1,2$, then $((\sigma_{1},\sigma_{2}), (\theta_{1}^{i},\theta_{2}^{i})) \rightarrow (\theta_{1}^{o},\theta_{2}^{o})$ and an adjusted condition holds for the output function.
\end{proof}

Now that we have set a connection with the literature, we can start working on our case. But before introducing the machine, we need to fix some notation. \\
First of all, we define a slightly more rigid version of a self-similar structure.
\begin{defn}
Let $\mathcal{T}$ be a self-similar rooted tree. We define a \textbf{rigid structure} as the subcollection of rooted tree isomorphims of the given self-similar structure that satisfy the following
\begin{itemize}
    \item[(a)]for each pair of vertices $x$ and $y$ of the same class, there exists a unique isomorphism $\varphi_{x,y}$ that maps $\mathcal{T}_{x}$ on $\mathcal{T}_{y}$;
    \item[(b)]it is closed under composition, that means $\varphi_{x,y}\varphi_{y,z}=\varphi_{x,z}$ for $x$, $y$ and $z$ of the same class;
    \item[(c)]if $x'$ is a child of $x$ and $y'=x'\varphi_{x,y}$, then $\varphi_{x',y'}$ is the restriction of $\varphi_{x,y}$ onto $\mathcal{T}_{x'}$.
    \end{itemize}
\end{defn}

It is always possible to retrieve a rigid structure starting from a self-similar one
\begin{prop}\label{rigid-structures-walk-among-us}
Every self-similar tree has a rigid structure.
\end{prop}
\noindent We are interested in the technique involved in the proof  of Proposition~\ref{rigid-structures-walk-among-us} (see \cite[Proposition 2.18]{BBM} for the complete version). More specifically, we need to define \textit{markings}. We take a set of vertices that contains exactly one element for each class and the root. We denote it with $\widehat{\Omega}$. Let $\Omega$ be the set of vertices which are children of elements in $\widehat{\Omega}$. For each $o \in \Omega$ we choose $\tau_{o}$ to be a rooted tree isomorphism between $o$ and the only vertex in $\widehat{\Omega}$ that belongs to the same class and we call it an \textbf{elementary marking}. Now, we take any vertex $x$ and we define its \textbf{marking} $\psi_{x}$ as follows:
\begin{itemize}
    \item[(1)]if $x$ is the root, then $\psi_{x}$ is the identity isomorphism of $\mathcal{T}$;
    \item[(2)]if $x$ is a vertex of $\mathcal{T}$ with marking $\psi_{x}$ and $y$ is a child of $x$, denote by $o=y\psi_{x}$ the corresponding child of $x\psi_{x}$ inside $ \widehat{\Omega}$  and define $\psi_{y}=\psi_{x}\tau_{o}$.
\end{itemize}
\noindent Note that the composition is \textit{partial}, which means that actually we are considering $\psi_{x}\tau_{o}$ with $\psi'_{x}$ the restriction of $\psi_{x}$ onto $\mathcal{T}_{y}$.

\[
\begin{tikzcd}[ampersand replacement=\&]
\mathcal{T}_{x} \arrow[rr, "\psi_{x}",rightarrow] \& \& \mathcal{T}_{x\psi} \\
\mathcal{T}_{y} \arrow[u, "",hookrightarrow] \arrow[r, "\psi'_{x}",rightarrow] \& \mathcal{T}_{o} \arrow[ur, "\tau_{o}"',rightarrow]
\end{tikzcd}
\]
One can see that the rooted tree isomorphism defined as $\varphi_{x,y}:=\psi_{x}\psi_{y}^{-1}$ with $x$ and $y$ two vertices of the same class are in fact a rigid structure.

We now apply these notions to our context. First of all, we observe the following.
\begin{oss}
The tree of atoms is self-similar with respect to the structure given by morphisms and types.
\end{oss}
For each atom $a$ we associate a marking $\psi_{a}$ and we have a collection $\{ \tau_{i} \mid i \in I \}$ for some finite set of indices $I$. These help us defining our coding in the following sense: we take an alphabet $R$ and we call it the set of \textbf{rigid types} that is in a 1 to 1 correspondence with the set of elementary markings.  We will usually have $r_{i} \leftrightarrow \tau_{i}$. An atom $\overline{a}$ is of rigid type $r_{i}$ if its marking is of the form $\psi_{\overline{a}}=\psi_{a}\tau_{i}$ with $\overline{a}$ child of $a$. Finally, given a horofunction $(u_{n})_{n=1}^{\infty}$, we get a corresponding string based on $R$.\\
A first consequence of this coding is that we can construct the \textbf{type automaton} (see Example 2.5 in \cite{BBM} for a general treatment on the subject): the set of states are the types of atoms and the number of transitions between two types are the number of children that an atom of the first type has of the second type. It is easy to see that this number does not depend on the choice of the atoms and that can be labeled by the rigid types.\\ 

\begin{exm}
\label{type_automaton_2D}
We consider again the uniform tiling of the hyperbolic plane made of squares such that each vertex has degree 5. And the groups of isometries given by $$G= \langle g,h \mid g^{5}=1, \ h^{2}=1, \ (gh)^{4}=1 \rangle$$ with $g$ the rotation by 72° around the center and $h$ the 180° rotation around the middle point of one of the edges starting from the center. By looking at the first levels (see Figure~\ref{atoms_tiling}.(\subref{atoms_tiling_A})), we can argue that there are four types of atoms. 
\begin{itemize}
    \item The root, that is the only $0$-level atom; we call it $A$.
    \item The first type, we call it $B$, there is one of them for each edge starting from $x_{0}$, so at the first level there are five.
    \item The second type, we call it $C$, these are the intersections between two red regions; again we have five of them at the first level.
\end{itemize}  
So, we will need ten letters $\{0,1, \ldots, 9 \}$ to codify the first level. At this point, we notice that every type $B$ has three children, two of type $B$ and one of type $C$. Hence, we will use the letters $B_{0},B_{1},B_{2}$. While for the type $C$, there are three children, two of type $C$ and 
\begin{itemize}
    \item The fourth type, we call it $D$, it is the middle child of a type $C$ and it has just one child of type $B$.
\end{itemize}
To conclude, we need three letters $C_{0},C_{1},C_{2}$ and a letter $D_{0}$ to fully encode the elementary markings. See Figure~\ref{type_automata_tiling}, for the type automaton.
\end{exm}
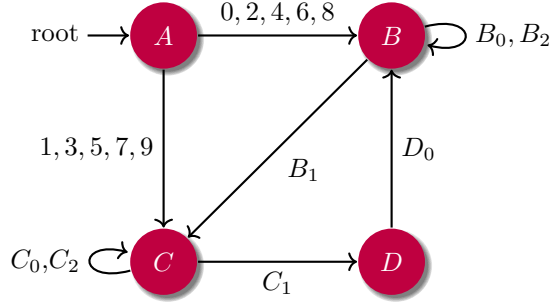
\begin{figure}
\centering
 \begin{tikzpicture}[node distance=3cm,on grid,auto, thick] 
   \node[state] (A)  [fill=purple,draw=none,circular drop shadow,text=white] {$A$}; 
   \draw[<-] (A) -- node[label={[xshift=-0.7cm, yshift=-0.2cm]root}] {} ++(-1cm,0);
   \node[state] (B) [fill=purple,draw=none,circular drop shadow,text=white] [ right=of A] {$B$};
   \node[state] (C) [fill=purple,draw=none,circular drop shadow,text=white] [ below=of A] {$C$}; 
   \node[state](D)[fill=purple,draw=none,circular drop shadow,text=white]  [ right=of C] {$D$};
  
    \path[->] 
    (A)   edge [] node {$0,2,4,6,8$} (B)
          edge [] node  [left] { $1,3,5,7,9$ } (C)
          
    (B) edge [loop right] node  {$B_{0},B_{2}$} ()
   	edge [] node  [below right] {$B_{1}$} (C)	 
    
    (C) edge  node [below]{$C_{1}$} (D)
        edge [loop left] node  {$C_{0}$,$C_{2}$} ()
    	
    (D) edge [] node  [right] {$D_{0}$} (B);
          
\end{tikzpicture}
\caption[Example of type automaton]{The type automaton for the $1$-skeleton of the hyperbolic tiling.}
\label{type_automata_tiling}
\end{figure}

\begin{defn}\label{gluing_automata_def}
We define the automaton $\mathcal{M}=(\Delta_{/ \sim}, R^2, \rightarrow , (a_{0},a_{0})_{\sim})$ in the following way:
\begin{description}
\item[States] The set $\Delta:=\{(a,b) \mid a,b \in \mathcal{A}_{n} \text{ for some } n\in \mathbb{N},  \  \dG(a,b) \leq \lambda \}$ with $\lambda$ the gluing constant and the quotient is on the relation $(a,b) \sim (c,d)$ there exists $g\in G$ such that $g=\psi_{a}\psi_{c}^{-1}=\psi_{b}\psi_{d}^{-1}$. The elements of the quotient are denoted by $(a,b)_{\sim}$ for some representative $(a,b) \in \Delta$.
\item[Alphabet] We consider the cartesian product of the set $R$ of rigid types with itself, that correspond to the cartesian product of the set of elementary markings.
\item[Transition function] We put an arrow $(r_{1},r_{2}) \in R^{2}$ from $(a,b)_{\sim}$ to $(c,d)_{\sim}$ whenever there exist two elements $(\overline{a},\overline{b}) \in (a,b)_{\sim}$ and $(\overline{c},\overline{d}) \in (c,d)_{\sim}$ such that $\overline{c}$ is a child of $\overline{a}$ and in particular $\psi_{\overline{c}}=\psi_{\overline{a}}\tau_{1}$ with $\tau_{1}$ the elementary marking associated to $r_{1}$. The same holds for $\overline{d},\overline{c}$ and $r_{2}$.
\item[Initial state] We denote with $a_{0}$ the atom of level 0. The notation follows from the description of the states.
\end{description}
\end{defn}
\begin{oss}
Note that the rigid structure does not define new types, so the requirement in the definition of $\sim$ is the same as asking for a $g \in G$ that is a morphism between $a$ and $c$, but also between $b$ and $d$.
\end{oss}
\begin{oss}
The automaton contains a copy of the type automaton. Indeed, given a type there is a state $(a,a)_{\sim}$ which collects all $a$ of that type paired with themselves. And the transitions between two of these states are labeled precisely by $(r,r)$ with $r$ ranging in the set of all transition labels of the type automaton.\\
\end{oss}

We now proceed in the following way: first we need to verify that the machine is actually doing what we expect on horofunctions and then we will prove that it is finite-state.\

\begin{prop}
Let $\mathcal{M}$ be the automaton described above and let $(u_{n})_{n=1}^{\infty}$ and $(v_{n})_{n=1}^{\infty}$ be two horofunctions described by their codings. We have that the horofunctions are the same element in $\partial G$ if and only if there exists an infinite transition through the states $\{(a_{n},b_{n})_{\sim}\}_{n=1}^{\infty}$ in $\mathcal{M}$.
\end{prop}
\begin{proof}
We first suppose that two horofunctions glue together. Then by Theorem~\ref{Conj1}, we know that $\dG(u_{n},v_{n}) \leq \lambda$ holds for every $n$. Hence, by definition of the automaton, we can take $a_{n}=u_{n}$ and $b_{n}=v_{n}$. We then know that, by passing through these states, we can read the string that is the sequence of double elementary markings associated to the horofunctions.\\

On the other hand, suppose we can read the two horofunctions on $\mathcal{M}$. Suppose, also, that for $k \leq n$ we have $(u_{k},v_{k}) \in (a_{k},b_{k})_{\sim}$. In particular, $\dG(u_{k},v_{k}) = \dG(u_{k}g,v_{k}g) =\dG(a_{k},b_{k}) \leq \lambda$. We want to show that $(u_{n+1},v_{n+1}) \in (a_{n+1},b_{n+1})_{\sim}$.\\
By hypothesis we know that $g=\psi_{u_{n}}\psi_{a_{n}}^{-1}=\psi_{v_{n}}\psi_{b_{n}}^{-1}$. Since we read $(r_{1},r_{2})$ to reach the state $(a_{n+1},b_{n+1})_{\sim}$ from the state $(a_{n},b_{n})_{\sim}$, we have
$$\psi_{u_{n+1}}=\psi_{u_{n}}\tau_{1} \text{ and } \psi_{v_{n+1}}=\psi_{v_{n}}\tau_{2},$$
with $r_{i}$ the digit associated to the elementary marking $\tau_{i}$ for $i=1,2$. \\
By the same token, we can choose the representative $(a_{n+1},b_{n+1})$ such that 
$$\psi_{a_{n+1}}=\psi_{a_{n}}\tau_{1} \text{ and } \psi_{b_{n+1}}=\psi_{b_{n}}\tau_{2}.$$

To conclude, $a_{n+1}$ and $u_{n+1}$ are of the same type (they have the same marking), the same holds for $b_{n+1}$ and $v_{n+1}$. So there exist $h_{1},h_{2} \in G$ such that
$$h_{1}=\psi_{u_{n+1}}\psi_{a_{n+1}}^{-1} \text{ and } h_{2}=\psi_{v_{n+1}}\psi_{b_{n+1}}^{-1}.$$

If we put everything together, we get
$$h_{1}=\psi_{u_{n+1}}\psi_{a_{n+1}}^{-1}=\psi_{u_{n}}\tau_{1}\tau_{1}^{-1}\psi_{a_{n}}^{-1}=g=\psi_{v_{n}}\tau_{2}\tau_{2}^{-1}\psi_{b_{n}}^{-1}=\psi_{v_{n+1}}\psi_{b_{n+1}}^{-1}=h_{2}.$$

This means that $\dG(u_{n+1},v_{n+1})=\dG(a_{n+1}g,b_{n+1}g)=\dG(a_{n+1},b_{n+1})$. In particular, $(u_{n+1},v_{n+1})$ belongs to $\Delta$ and to $(a_{n+1},b_{n+1})_{\sim}$ as desired.
\end{proof}

A couple of remarks about efficiency and geometric interpretation of the machine are needed.

\begin{oss}
\begin{itemize}\
\item[(1)]Notice that the number of steps before stopping is not optimal. This occurs since we rely on a gluing constant and so we may create some states without any possible transition from there, in other words the automaton is not reduced. 
\item[(2)]The automaton roughly gives an estimate about the distance between two points in the Gromov boundary by looking at their codings. Indeed, by Theorem~\ref{SConj1} and by Definition~\ref{gluing_automata_def} 
the (discrete) amount of time at which the machine stops is not far from the Gromov product between the two elements in input.
\end{itemize}

\end{oss}

In order to leave the previous proposition as clean as possible, we collect here the properties of the automaton needed to get the rationality.

\begin{cor}
The machine $\mathcal{M}$ is deterministic and recognizes the horofunctions.
\end{cor}
\begin{proof}\

\textit{Determinism.} Suppose that the following situation occurs in the automaton: 

\[ \begin{tikzpicture}[node distance=2.5cm,on grid,auto,el/.style = {inner sep=2pt, align=left, sloped},] 
   \node[ellipse, fill=purple, text=white, circular drop shadow] (ab)   {$(a,b)_{\sim}$}; 
   \node[ellipse, fill=purple, text=white,circular drop shadow] (1) [above right=of ab,xshift=1cm] {\small{$(a_{1},b_{1})_{\sim}$}};
   \node[ellipse, fill=purple, text=white,circular drop shadow] (2) [below right=of ab, xshift=1cm] {\small{$(a_{2},b_{2})_{\sim}$}}; 
  
   \path[->] 
    (ab)  edge[bend left=20]  node[el, above] {$(r,s)$} (1)
          edge[bend right=20]  node[el, below] {$(r,s)$} (2);
\end{tikzpicture}
\]
we want to prove that $(a_{1},b_{1})_{\sim}=(a_{2},b_{2})_{\sim}$. This is a consequence of the rigid structure, namely there exist 
$$(a,b) \in (a,b)_{\sim} \text{ and } (a_{1},b_{1}) \in (a_{1},b_{1})_{\sim} \text{ with } \psi_{a_{1}}=\tau_{r}\psi_{a}, \ \psi_{b_{1}}=\tau_{s}\psi_{b}$$
and
$$(\overline{a},\overline{b}) \in (a,b)_{\sim} \text{ and } (a_{2},b_{2}) \in (a_{2},b_{2})_{\sim} \text{ with } \psi_{a_{2}}=\psi_{\overline{a}}\tau_{r}, \ \psi_{b_{2}}=\psi_{\overline{b}}\tau_{s};$$
moreover there exists $g \in G$ such that $g=\psi_{\overline{a}}\psi_{a}^{-1}=\psi_{\overline{b}}\psi_{b}^{-1}$. So we have
$$ \psi_{a_{2}}\psi_{a_{1}}^{-1}=\psi_{\overline{a}}\tau_{r}\tau_{r}^{-1}\psi_{a}^{-1}=g=\psi_{\overline{b}}\tau_{s}\tau_{s}^{-1}\psi_{b}^{-1}=\psi_{b_{2}}\psi_{b_{1}}^{-1},$$
that yields the claim.\\

\textit{Recognizer.} Suppose 
$$(a_{0},a_{0})_{\sim} \rightarrow (a_{1},b_{1})_{\sim} \rightarrow \ldots \rightarrow (a_{n},b_{n})_{\sim} \rightarrow \ldots$$
is a transition of states on $\mathcal{M}$. We recall $(a_{0},a_{0})_{\sim}$ is the initial state, related to the word
$$(r_{1},s_{1})(r_{2},s_{2}) \ldots (r_{n},s_{n}) \ldots$$
This means that there exist $a_{n-1} \xrightarrow[]{r_{n}} a_{n}$ and $\overline{a}_{n} \xrightarrow[]{r_{n+1}} \overline{a}_{n+1}$, but $a_{n}$ and $\overline{a}_{n}$ have the same type and so $r_{n+1}$ must be an allowed rigid type for $a_{n}$ too, hence we can provide $a_{n+1}$ such that $a_{n} \xrightarrow[]{r_{n+1}} a_{n+1}$. 
\end{proof}

This last part is devoted to showing that $\mathcal{M}$ is a finite state machine, namely we want a bound for the cardinality of $\Delta_{/ \sim}$. For this purpose, we need the key definition introduced in \cite{BBM} to prove that the number of types are finite.\\

Before that, we recall that if two elements $x$ and $y$ belong to the same $k$-level atom $a$, then $\overline{d}_{x}=\overline{d}_{y}$ over $B_{k}$. Hence, $\overline{d}_{a}$ is well defined over $B_{k}$.\\ 
If $\Gamma_{0}$ is a subset of vertices of $\Gamma$,  $f$ is a function from $\Gamma_{0}$ to $\mathbb{Z}$ and $g \in G$, then we define $fg$ to be the function $yfg:=yg^{-1}f$ for all $y \in \Gamma_{0}g$. Note also that if $f_{1}$ and $f_{2}$ differ by a constant, that $f_{1}g$ and $f_{1}g$ also differ by a constant. Putting these two facts together leads to the definition of $\overline{d}_{a}g$ on $B_{k}g$. 
\begin{defn}
Let $a \in \mathcal{A}_{m}$ and $b \in \mathcal{A}_{n}$. We say that an element $g \in G$ induces a \textbf{geometric equivalence} between $a$ and $b$ if
\begin{itemize}
    \item[(1)]$P(a,S_{m})g=P(b,S_{n})$;
    \item[(2)]$\overline{d}_{a}g=\overline{d}_{b}$ over $P(b,S_{n})$;
    \item[(3)]$C(p)g=C(pg)$ for all $p \in P(a,S_{m})$.
\end{itemize}
\end{defn}

\noindent The main result concerning this definition is the following
\begin{prop}\label{geometric_implies_morphism}
If $g \in G$ induces a geometric equivalence between two atoms $a$ and $b$, then it induces a morphism. Hence $a$ and $b$ are of the same type.
\end{prop}
\noindent Most of Section 3.5 of \cite{BBM} consists of a proof for this Proposition. The following proof is taken from Corollary 3.28 in \cite{BBM} and we show it here because it will be useful to understand our case.\

\begin{proof}[Proof(Theorem~\ref{finite-types})] \label{proof_finite_types}
By virtue of Proposition~\ref{geometric_implies_morphism}, it suffices to prove that geometric equivalence classes are finite.\\
Since the action of $G$ onto $\Gamma$ is cocompact, there exists a compact, hence finite, subset $K$ of vertices such that 
$$KG=\{ Kg \mid g \in G \}$$
is the whole graph. Now take $p \in P(a,S_{m})$, then there exists an element $h \in G$ such that $ph \in KG$; exploiting Proposition~\ref{visible_proximal_b} we have that $P(a,S_{n})$ is contained in a $8\delta+
2$-neighborhood of $K$. This means that there are finitely many possibilities for $P(a,S_{m})$ modulo the action of $G$. Moreover, since the action of $G$ is properly discontinuous and by Proposition~\ref{cone_type_finite} there are finitely many cone types, we have only finitely many choices for $C(p)$ for each $p \in P(a,S_{m})$, and there are only finitely many choices for the restriction of $\overline{d}_{a}$ to $P(a,S_{m})$.
\end{proof}
We are going to study a slight refinement of geometric equivalences and types. In order to do that, we consider a $\lambda$-neighborhood of an atom with respect to $\dG$. 
The fact that the collection of atoms in the neighborhood of an atom is finite is due to the fact that $\dG \leq \T\dG \leq \dG + \lambda$ with $\lambda$ the gluing constant together with the fact that the tips are finite (see Proposition~\ref{diam}).\\ 
We denote the set of all $n$-level atoms within a distance $\lambda$ to an $n$-level atom $a$ with $\mathcal{A}_{\lambda}(a)$ and we call $a$ the \textbf{center} of the neighborhood. 
\begin{defn}
Two atoms $a,b \in \mathcal{A}$ have the same \textbf{$\lambda$-type} if there exists an element $g \in G$ that induces a bijection between $\mathcal{A}_{\lambda}(a)$ and $\mathcal{A}_{\lambda}(b)$ and such that $g_{|_{a_{\lambda}}}$ 
is a morphism of types between $a_{\lambda}$ and $b_{\lambda}$ 
for all $a_{\lambda} \in \mathcal{A}_{\lambda}(a)$.
\end{defn}
Before giving the corresponding definition of geometric equivalence, we want to notice the following.
\begin{oss}
\begin{itemize}\
\item[(1)]There can be two atoms in a $\lambda$-neighborhood with the same  type. 
\item[(2)]Since $g$ is an isometry, we have that the distance between two atoms in a $\lambda$-neighborhood depends just on the $\lambda$-type of its center.
\end{itemize}
\end{oss} 

\begin{defn}
Two atoms $a,b \in \mathcal{A}$ are said to be \textbf{geometric $\lambda$-equivalent} if there exists $g \in G$ such that for all $a_{\lambda} \in \mathcal{A}_{\lambda}(a)$ and $b_{\lambda} \in \mathcal{A}_{\lambda}(b)$ the following hold
\begin{itemize}
\item[(1)]$P(a_{\lambda},S_{n})g=P(b_{\lambda},S_{m})$;
\item[(2)]$\overline{d}_{a_{\lambda}}g$ agrees with $\overline{d}_{b_{\lambda}}$ on $P(b_{\lambda},S_{m})$;
\item[(3)]$C(p)g=C(pg)$ for all $p \in P(a_{\lambda},S_{n})$; 
\end{itemize}
for suitable positive integers $n$ and $m$.
\end{defn}

\noindent We have the following version of Proposition~\ref{geometric_implies_morphism}.

\begin{lem}
\label{geometric_lambda}
If two atoms are geometric $\lambda$-equivalent, than they have the same $\lambda$-type.
\end{lem}
\begin{proof}
By definition, the element $g$ that induces the geometric $\lambda$-equivalence also induces a geometric equivalence on each atom that belongs to $\mathcal{A}_{\lambda}(a)$. By Proposition~\ref{geometric_implies_morphism}, we have that $g$ induces a morphism on each atom. Hence the claim.
\end{proof}

\noindent All that is left to do is prove that the number of geometric ${\lambda}$-equivalence classes is finite. But again this follow almost immediately by \cite{BBM}.

\begin{lem}
\label{geometric_lambda_finite}
The number of equivalence classes with respect to the geometric $\lambda$-equivalence is finite.
\end{lem}
\begin{proof}
One can argue as in the \hyperref[proof_finite_types]{proof of Theorem~\ref{finite-types}} and by noticing that the union of all the proximal sets of atoms in $\mathcal{A}_{_{\lambda}}(a)$ has a finite diameter by virtue of the \hyperref[Conj3]{Hooking Lemma}.
\end{proof}

\noindent We are now ready to prove the following

\begin{prop}
The set $\Delta_{/\sim}$ is finite.
\end{prop}
\begin{proof}
The key idea is that there exists a way to cover $\Delta$ by $\mathcal{A}_{\lambda}(a)$ as $a$ ranges in $\mathcal{A}$, or, more explicitly, for each $(a,b) \in \Delta$, we have that $b \in \mathcal{A}_{\lambda}(a)$. \\
By combining Lemma~\ref{geometric_lambda} and Lemma~\ref{geometric_lambda_finite}, we have that there are finitely many $\lambda$-types. Set $C_{\lambda}$ to be the finite number of $\lambda$-types. We also know that $|\mathcal{A}_{\lambda}(a)|$ is finite, and in particular there are finitely many pairs $(a,b)$ as $b \in \mathcal{A}_{\lambda}(a)$.\\
Finally, if $(a,b) \sim (c,d)$ then there exists $g \in G$ that induces a map $g:\mathcal{A}_{\lambda}(a) \rightarrow \mathcal{A}_{\lambda}(c)$ and such that $(a,b)g=(c,d)$. So there are at most $C_{\lambda}(|\mathcal{A}_{\lambda}(a)|-1)$ elements in $\Delta_{/\sim}$.
\end{proof}

To summarize what we achieved in this section, we explicit give this
\begin{thm}\label{rational_gluing_main_thm}
The quotient map $\pi_{h}: \partial_{h} \Gamma \twoheadrightarrow \partial \Gamma$ defines a rational equivalence relation.
\end{thm}

As a final remark, we point out that since two atoms in a $\lambda$-neighborhood may have the same rigid type, we cannot conclude that the gluing relation is semi-Markovian as for other tree structures on the Gromov boundary. But something more can be said about $\lambda$-types.

\begin{prop}
The $\lambda$-types are a self-similar structure for the tree of atoms.
\end{prop}
\begin{proof}
The claim follows easily from the fact that $\lambda$-types are finite and the following argument: the element $g \in G$ that maps one $\lambda$-neighborhood into another is an isometry and it is a morphism on each atom of the neighborhood. Hence, it preserves the $\lambda$-neighborhoods of chidren and the types of the atoms contained in them.
\end{proof}

\section{Example} \label{8}

In this section, we will deal with the group

$$\langle g_{1},g_{2},g_{3},g_{4} \mid g_{i}^{2}, \ \  (g_{i}g_{j})^{6}, \ \ i \in \{ 1,2,3,4\}, \ j>i \rangle.$$

Geometrically, we can represent each of the relations with an hexagon of edge $g_{i}g_{j}$ (see Figure~\ref{generator_relator_fractal}.(\subref{Relator_fractal})). Since we have four generators and all of them are involutions, we can imagine the situation depicted in Figure~\ref{generator_relator_fractal}.(\subref{generators_fractal}), that is the vertex $g_{i}$ coincide with its inverse and we have six “hexagonal'' relations.
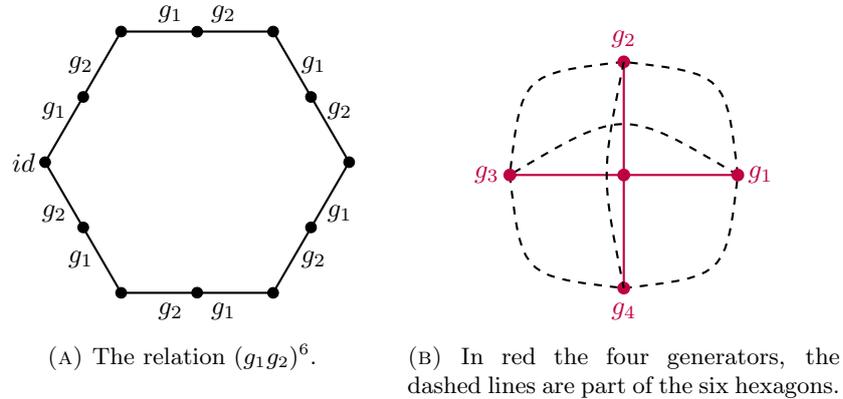
\begin{figure}
    \centering
\begin{subfigure}[t]{0.45\textwidth}
\centering
    \begin{tikzpicture}
    \draw[thick] (180:2) node[left] {$id$} ;
    \foreach \a in {160,220,280,340,100,40} { 
\draw[thick] (\a:2) node[] {$g_{1}$} ;}
\foreach \a in {320,20,80,140,200,260} { 
\draw[] (\a:2) node[] {$g_{2}$} ;}

        \draw[thick] (2,0)
  \foreach \x in {0,60,120,180,240,300,360}
    { --  (\x:2)};

     \foreach \a in {0,60,120,180,240,300} { 
\draw[fill,color=black] (\a:2cm) circle (2pt);} 

    \begin{scope}[rotate=30]
        \foreach \a in {0,60,120,180,240,300} { 
\draw[fill,color=black] (\a:1.73cm) circle (2pt);

}
    \end{scope}
    \end{tikzpicture}
    \subcaption{The relation $(g_{1}g_{2})^{6}$.}
    \label{Relator_fractal}
\end{subfigure}
\begin{subfigure}[t]{0.45\textwidth}
\centering
    \begin{tikzpicture}
     \draw[thick,fill,purple] (0:0) circle (2pt);
    \draw[thick,black, dashed] 
  \foreach \x in {0,90,180,270}
    { (\x:1.5) .. controls (\x+45:1.9)  ..   (\x+90:1.5)};
    \draw[thick, black,dashed]
    (0:1.5) .. controls (90:0.9) ..  (180:1.5);
    
    \draw[thick,fill,purple] 
    \foreach \x in {0,90,180,270}
    { (0,0) --   (\x:1.5) circle (2pt)};
    \draw[thick,fill,purple]   (0:1.8) node {$g_{1}$};
    \draw[thick,fill,purple]   (90:1.8) node {$g_{2}$};
    \draw[thick,fill,purple]   (180:1.8) node {$g_{3}$};
    \draw[thick,fill,purple]   (270:1.8) node {$g_{4}$};
    \draw[thick, black,dashed]
    {(90:1.5) .. controls  (90+90:0.3) ..  (90+180:1.5)};
    \end{tikzpicture}
    \subcaption{In red the four generators, the dashed lines are part of the six hexagons.}
    \label{generators_fractal}
\end{subfigure}
\caption{Generators and relations of the group.}
\label{generator_relator_fractal}
\end{figure}

Now, we can consider two types of atoms at the first level: a wide one (in Figure~\ref{atoms_fractal} there are two of them, outlined by blue lines) and a narrow one (there is one of them in Figure~\ref{atoms_fractal} and it is the green one). The narrow type, unlike the wide, will not split for the next four levels, this is due to the fact that that we have to wait until $B_{6}(id)$ to intersect the atom. So, the only child has a different type at each step, but all of them are homeomorphic.\\
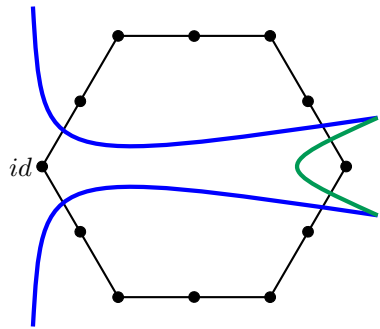
\begin{figure}
    \centering
    \begin{tikzpicture}
    \draw[thick] (180:2) node[left] {$id$} ;

        \draw[thick] (2,0)
  \foreach \x in {0,60,120,180,240,300,360}
    { --  (\x:2)};

     \foreach \a in {0,60,120,180,240,300} { 
\draw[fill,color=black] (\a:2cm) circle (2pt);} 

\draw[ultra thick,blue] (135:3) .. controls (180:2) .. (15:2.5);

\draw[ultra thick,blue] (225:3) .. controls (180:2) .. (-15:2.5);

\draw[ultra thick,ForestGreen] (-15:2.5) .. controls (0:1) .. (15:2.5);

    \begin{scope}[rotate=30]
        \foreach \a in {0,60,120,180,240,300} { 
\draw[fill,color=black] (\a:1.73cm) circle (2pt);

}
    \end{scope}
    \end{tikzpicture}
    \caption{The two types of atom at the first level seen on a single relation $(g_{i}g_{j})^{6}$.}
    \label{atoms_fractal}
    \end{figure}

\begin{figure}
\centering

\begin{subfigure}[t]{0.45\textwidth}
\centering
    \begin{tikzpicture}
        \draw[thick] (0,0) -- (-15:2);
        \draw[thick,dashed] (0,0) -- (20:2);
        \draw[thick] (20:2) -- (-15:2);
       % \draw[thick] (20:2) -- +(-87.5:1.2);
        \draw[thick] (0,0) -- (58:2);
        \draw[thick] (-15:2) --  (58:2);
        %\draw[thick] (-15:2) --  +(111.5:2.3);
        \draw[thick] (20:2) -- (58:2);
       % \draw[thick] (20:2) -- +(127.5:1.3);

        \draw[fill,blue] (0,0) circle (2pt);
        \draw[fill,blue] (-15:2) circle (2pt);
        \draw[fill,blue] (20:2) circle (2pt);
        \draw[fill,blue] (58:2) circle (2pt);
        \draw[fill,ForestGreen] (-15:1) circle (2pt);
        \draw[fill,ForestGreen] (20:1) circle (2pt);
        \draw[fill,ForestGreen] (58:1) circle (2pt);
        \draw[fill,ForestGreen] (-15:2)++(111.5:1.3) circle (2pt);
        \draw[fill,ForestGreen] (20:2)++(127.5:0.65) circle (2pt);
        \draw[fill,ForestGreen] (20:2)++(-87.5:0.6) circle (2pt);
        
    \end{tikzpicture}
    \subcaption{}
    \label{HG_fractal}
\end{subfigure}
\begin{subfigure}[t]{0.45\textwidth}
\centering
    \begin{tikzpicture}

    \draw[thick,fill,black] 
    \foreach \x in {-30,90,210}
    { (0,0) --   (\x:1.5) };

\foreach \x in {-30,90,210}
{
\draw[thick,fill,ForestGreen]  (\x:1.5) node {} circle (2pt);}

    \draw[thick,fill,blue] (0:0) circle (2pt);

    \end{tikzpicture}
    \subcaption{}
    \label{HG_portion}
\end{subfigure}
    \caption[The first horizontal graph]{The first horizontal graph and a portion of it.}
    \label{HG_fractal_portion}
\end{figure}
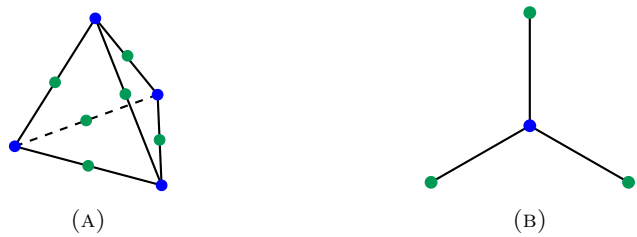
To construct the horizontal graph (again here we suppose $\lambda_{e}=1$ for the sake simplicity, as in the previous example), we have to imagine the four wide type atoms as vertices of a tetrahedron, while the six narrow type atoms are the middle point of the edges (see Figure~\ref{HG_fractal_portion}.(\subref{HG_fractal})).\\

Due to the self-similar nature of the atoms, we can focus only to one portion of the tetrahedron: we consider a wide atom and its three adjacent narrow atoms (see Figure~\ref{HG_fractal_portion}.(\subref{HG_portion})). Then we just focus on this portion and it can be seen that the sequence of horizontal graphs is the one depicted in Figure~\ref{fractal_approx}. In particular, in these four steps the narrow type atom remains a vertex, while the expansion is made by the wide type atom. This process leads to an Apollonian gasket.

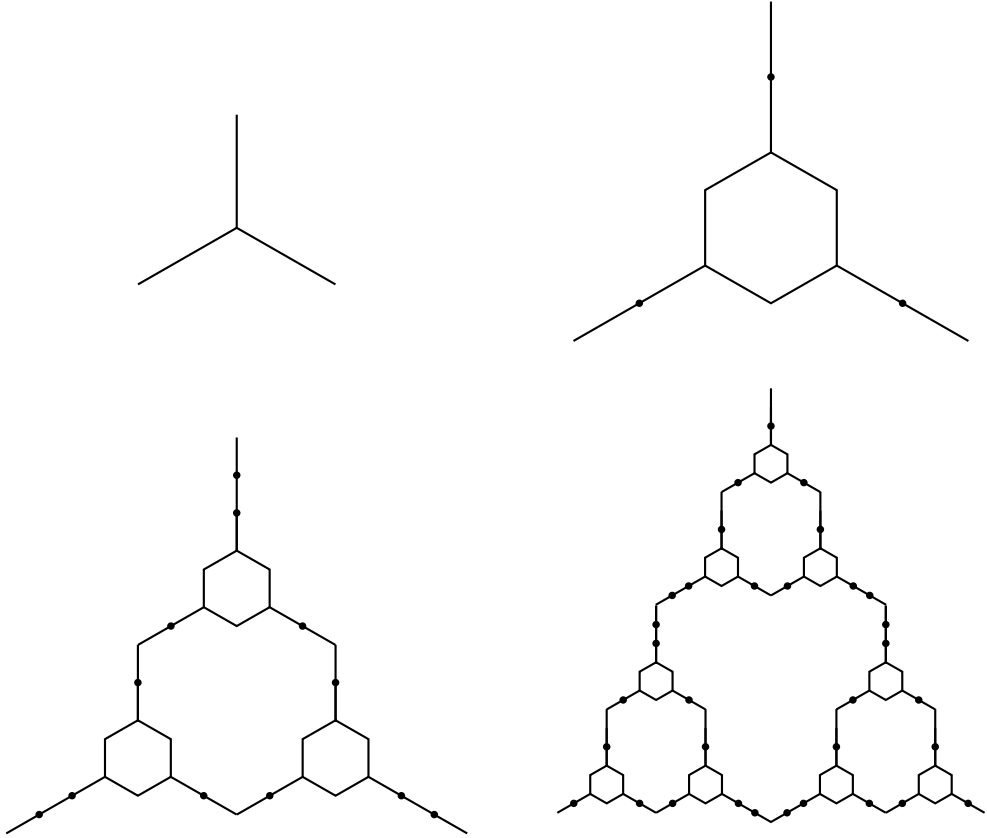
\begin{figure}
    \begin{tikzpicture}

    \draw[thick,fill,black] 
    \foreach \x in {-30,90,210}
    { (0,0) --   (\x:1.5) };

\begin{scope}[xshift=200]
\draw[thick] (0,1)
  \foreach \x in {90,150,210,270,330,30,90}
    { --  (\x:1)};
    \foreach \a in {90,210,330}
    \draw[thick] (\a:1) -- (\a:3);
    \foreach \a in {90,210,330}
    \draw[thick,fill]  (\a:2) node {} circle (1pt);
\end{scope}

\begin{scope}[yshift=-200]
\begin{scope}[xshift=37]
\draw[thick] (0,1)
  \foreach \x in {90,150,210,270,330,30,90}
    { --  (\x:.5)};
    \foreach \a in {90,210}
    \draw[thick] (\a:.5) -- (\a:1.5);
    \foreach \a in {90,210,330}
    \draw[thick,fill]  (\a:1) node {} circle (1pt);
    \draw[thick] (330:.5) -- (330:2);
    \draw[thick,fill]  (330:1.5) node {} circle (1pt);
\end{scope}
\begin{scope}[xshift=-37]
\draw[thick] (0,1)
  \foreach \x in {90,150,210,270,330,30,90}
    { --  (\x:.5)};
    \foreach \a in {90,330}
    \draw[thick] (\a:.5) -- (\a:1.5);
    \foreach \a in {90,210,330}
    \draw[thick,fill]  (\a:1) node {} circle (1pt);
     \draw[thick] (210:.5) -- (210:2);
    \draw[thick,fill]  (210:1.5) node {} circle (1pt);
\end{scope}
\begin{scope}[yshift=64]
\draw[thick] (0,1)
  \foreach \x in {90,150,210,270,330,30,90}
    { --  (\x:.5)};
    \foreach \a in {210,330}
    \draw[thick] (\a:.5) -- (\a:1.5);
    \foreach \a in {90,210,330}
    \draw[thick,fill]  (\a:1) node {} circle (1pt);
     \draw[thick] (90:.5) -- (90:2);
    \draw[thick,fill]  (90:1.5) node {} circle (1pt);
\end{scope}
\end{scope}

\begin{scope}[xshift=200,yshift=-210]
\begin{scope}[xshift=24.5]
\begin{scope}[xshift=37]
\draw[thick] (0,1)
  \foreach \x in {90,150,210,270,330,30,90}
    { --  (\x:.25)};
    \foreach \a in {90,210,330}
    \draw[thick] (\a:.25) -- (\a:.75);
    \foreach \a in {90,210,330}
    \draw[thick,fill]  (\a:.5) node {} circle (1pt);
    % \draw[thick] (330:.25) -- (330:1);
    % \draw[thick,fill]  (330:.75) node {} circle (1pt);
\end{scope}
\begin{scope}[xshift=0]
\draw[thick] (0,1)
  \foreach \x in {90,150,210,270,330,30,90}
    { --  (\x:.25)};
    \foreach \a in {90,330}
    \draw[thick] (\a:.25) -- (\a:.75);
    \foreach \a in {90,210,330}
    \draw[thick,fill]  (\a:.5) node {} circle (1pt);
     \draw[thick] (210:.25) -- (210:1);
    \draw[thick,fill]  (210:.75) node {} circle (1pt);
\end{scope}
\begin{scope}[yshift=39,xshift=18.5]
\draw[thick] (0,1)
  \foreach \x in {90,150,210,270,330,30,90}
    { --  (\x:.25)};
    \foreach \a in {210,330}
    \draw[thick] (\a:.25) -- (\a:.75);
    \foreach \a in {90,210,330}
    \draw[thick,fill]  (\a:.5) node {} circle (1pt);
     \draw[thick] (90:.25) -- (90:1);
    \draw[thick,fill]  (90:.75) node {} circle (1pt);
\end{scope}
\end{scope}

\begin{scope}[xshift=-61.5]
\begin{scope}[xshift=37]
\draw[thick] (0,1)
  \foreach \x in {90,150,210,270,330,30,90}
    { --  (\x:.25)};
    \foreach \a in {90,210}
    \draw[thick] (\a:.25) -- (\a:.75);
    \foreach \a in {90,210,330}
    \draw[thick,fill]  (\a:.5) node {} circle (1pt);
    \draw[thick] (330:.25) -- (330:1);
    \draw[thick,fill]  (330:.75) node {} circle (1pt);
\end{scope}
\begin{scope}[xshift=0]
\draw[thick] (0,1)
  \foreach \x in {90,150,210,270,330,30,90}
    { --  (\x:.25)};
    \foreach \a in {90,210,330}
    \draw[thick] (\a:.25) -- (\a:.75);
    \foreach \a in {90,210,330}
    \draw[thick,fill]  (\a:.5) node {} circle (1pt);
    %  \draw[thick] (210:.25) -- (210:1);
    % \draw[thick,fill]  (210:.75) node {} circle (1pt);
\end{scope}
\begin{scope}[yshift=39,xshift=18.5]
\draw[thick] (0,1)
  \foreach \x in {90,150,210,270,330,30,90}
    { --  (\x:.25)};
    \foreach \a in {210,330}
    \draw[thick] (\a:.25) -- (\a:.75);
    \foreach \a in {90,210,330}
    \draw[thick,fill]  (\a:.5) node {} circle (1pt);
     \draw[thick] (90:.25) -- (90:1);
    \draw[thick,fill]  (90:.75) node {} circle (1pt);
\end{scope}
\end{scope}

\begin{scope}[xshift=-18.5,yshift=82]
\begin{scope}[xshift=37]
\draw[thick] (0,1)
  \foreach \x in {90,150,210,270,330,30,90}
    { --  (\x:.25)};
    \foreach \a in {90,210}
    \draw[thick] (\a:.25) -- (\a:.75);
    \foreach \a in {90,210,330}
    \draw[thick,fill]  (\a:.5) node {} circle (1pt);
    \draw[thick] (330:.25) -- (330:1);
    \draw[thick,fill]  (330:.75) node {} circle (1pt);
\end{scope}
\begin{scope}[xshift=0]
\draw[thick] (0,1)
  \foreach \x in {90,150,210,270,330,30,90}
    { --  (\x:.25)};
    \foreach \a in {90,330}
    \draw[thick] (\a:.25) -- (\a:.75);
    \foreach \a in {90,210,330}
    \draw[thick,fill]  (\a:.5) node {} circle (1pt);
     \draw[thick] (210:.25) -- (210:1);
    \draw[thick,fill]  (210:.75) node {} circle (1pt);
\end{scope}
\begin{scope}[yshift=39,xshift=18.5]
\draw[thick] (0,1)
  \foreach \x in {90,150,210,270,330,30,90}
    { --  (\x:.25)};
    \foreach \a in {90,210,330}
    \draw[thick] (\a:.25) -- (\a:.75);
    \foreach \a in {90,210,330}
    \draw[thick,fill]  (\a:.5) node {} circle (1pt);
    %  \draw[thick] (90:.25) -- (90:1);
    % \draw[thick,fill]  (90:.75) node {} circle (1pt);
\end{scope}
\end{scope}
\end{scope}
    \end{tikzpicture}
    \caption{The first four portions of horizontal graphs.}
\label{fractal_approx}
\end{figure}

\subsection{Gluing automata} We want to discuss the construction of the gluing automaton for this example. We do not exhibit the full automaton, but we provide a sketch of how to build it. \

We start by saying that the type automaton is depicted in Figure~\ref{fractal_automaton}, where $\lambda$ is set to $1$ as for the other examples and we highlight the blue and green states corresponding respectively to the wide and the narrow type at the first level (the initial state is $A$ as always). Names for the other states are given by following a geometric intuition that we prefer to omit since it is not useful for for this description.\

We proceed by levels and we only show the first one. We need to list all the elements in $\Delta$ at the current level (we can look at the tetrahedron in Figure~\ref{HG_fractal_portion}.(\subref{HG_fractal})). So, in this case we have
\begin{center}
\begin{tabular}{ c|cccc } 
  & $w_{1}$ & $w_{2}$ & $w_{3}$ & $w_{4}$  \\ 
  \hline
 $n_{1}$ & 1 & 0 & 0 & 1\\ 
 $n_{2}$& 1 & 0 & 1 & 0\\
  $n_{3}$& 1 & 1 & 0 & 0\\ 
   $n_{4}$ & 0 & 1 & 0 & 1\\ 
    $n_{5}$ & 0 & 1 & 1 & 0\\ 
    $n_{6}$ & 0 & 0 & 1 & 1\\ 
\end{tabular}
\end{center}
and note that each pair need to be counted twice $(w_{i},n_{j})$ and $(n_{j},w_{i})$.\

In order to compute the states of $\mathcal{M}$, we should provide elements of the group that send pairs to pairs. Note that this does not mean that whenever we have a pair of atoms $(a,b)$ and another one $(a',b')$ such that $a'$ shares the same type with $a$ and $b'$ with $b$, there exists such element. Once completed the first level, we pass to the second having in mind that now we already have some states in $\mathcal{M}$ and hence new elements in $\Delta$ may be in the same equivalence class of an element of the previous level. Moreover, we have to add all the possible transitions according to the rigid structure. The procedure ends when we are sure that all possible rigid types have been processed and this can be done by looking at the type automaton. 

\begin{figure}
\resizebox{0.9\textwidth}{!}{
\begin{tikzpicture}[node distance=5cm,on grid,auto, thick] 
\tikzset{%
    in place/.style={
      auto=false,
      fill=white,
      inner sep=2pt,
    },
  }
   \node[state,accepting] at (-4,12) (A)  [fill=purple,draw=none,circular drop shadow,text=white] {$A$}; 
   \node[state,accepting] at (-4,9) (W)  [fill=blue,draw=none,drop shadow,text=white] {$W$}; 
   \node[state,accepting] at (2,12) (N)  [fill=ForestGreen,draw=none,drop shadow,text=white] {$N$}; 
   \node[state,accepting] at (3, 14) (N1)  [fill=purple,draw=none,drop shadow,text=white]{$N1$}; 
   \node[state] at (5,14) (N2) [fill=purple,draw=none,circular drop shadow,text=white] {$N2$};
   \node[state]  at (6,12) (N3) [fill=purple,draw=none,circular drop shadow,text=white]  {$N3$}; 
   \node[state] at (5,10) (N5)   [fill=purple,draw=none,circular drop shadow,text=white]  {$N5$};
   \node[state]  at (4,12) (N4) [fill=purple,draw=none,circular drop shadow,text=white] {$N4$}; 
   \node[state,accepting] at (3,10) (Nt)  [fill=purple,draw=none,circular drop shadow,text=white]   {$\tilde{N}$};
   \node[state,accepting] at (-4,6) (W')[fill=purple,draw=none,circular drop shadow,text=white]   {$W'$};
   \node[state,accepting] at (-4,3) (W'')[fill=purple,draw=none,circular drop shadow,text=white]   {$W''$};
   \node[state,accepting] at (-1,6) (N')[fill=purple,draw=none,circular drop shadow,text=white]   {$N'$};
    \node[state,accepting] at (2,6) (N'1)[fill=purple,draw=none,circular drop shadow,text=white]   {$N'1$};
    \node[state,accepting] at (-1,3) (N'')[fill=purple,draw=none,circular drop shadow,text=white]   {$N''$};
  
    \path[->] 
    (A) edge  node[in place] {$4$} (W)
       edge  node [in place]{$6$} (N) 

     (N)  edge node[in place] {$1$} (N1)
    (N1)  edge node[in place] {$1$} (N2)
    (N2)  edge node[in place] {$1$} (N3)
    (N3)  edge node[in place] {$1$} (N4)
    (N4)  edge node[in place] {$1$} (N5)
    (N5)  edge node[in place] {$1$} (Nt)
        edge [loop right] node[in place] {$4$} ()
    (N'')  edge node[in place] {$2$} (W'')
           %edge node[in place] {$2$} (N2)
        
    (W')  edge[loop left]  node[in place] {$2$} ()
           edge  node[in place] {$2$} (N')
            edge  node[pos=.25,in place] {$1$} (N'')
           edge  node[in place] {$1$} (W'')
           edge  node[in place] {$3$} (N)
           
      (W)  edge  node[in place] {$3$} (W')
          edge  node[in place] {$3$} (N)
          edge  node[in place] {$3$} (N')

       (W'')  edge  node[pos=.25,in place] {$2$} (N')
             edge[bend left=60]  node[in place] {$2$} (W)
             %edge  node[in place] {$1$} (N2)
         
     (N') edge  node[in place] {$1$} (N'1)
     (N'1) edge  node[in place] {$1$} (N'')
     (Nt) edge  node[in place] {$2$} (W')
          edge  node[in place] {$4$} (N');

 \draw[->]
    (N'') to (3,3) to node[in place] {$2$} (7,7) to (7,14) to (N2);

\draw[->]
    (W'') to (-6,3) to (-6,13) to node[in place] {$1$} (2,15) to (5,15) to (N2);
\end{tikzpicture}
}
\caption[The type automaton for the fractal example]{The type automaton: labels denote the number of arrows of that kind, the blue state and the letter $W$ mean the wide type and the same holds for the green state and the letter $N$ (narrow). }
\label{fractal_automaton}
\end{figure}

\section*{Acknowledgements}
This work is part of the author’s Ph.D. thesis for the joint Ph.D. program in Mathematics of University of Milano-Bicocca, University of Pavia and INdAM. The author would like to thank his advisors James Belk and Francesco Matucci for proposing the topic to him and for all of their insightful guidance and suggestions. The author gratefully acknowledges Collin Bleak and Matthew C. B. Zaremsky
for helpful corrections and comments. The author is also grateful to James Belk, Collin Bleak and Francesco Matucci for kindly providing the images of the atoms of the hyperbolic disk tiling (Figure~\ref{atoms_tiling}.(\subref{atoms_tiling_A}) and Figure~\ref{4-5_tiling}) from their work \cite{BBM}.

\bibliographystyle{alphaurl}
\bibliography{references}

\begin{thebibliography}{BBM21}

\bibitem[Bar18]{B}
B.~J. Barrett.
\newblock {\em Detecting topological properties of boundaries of hyperbolic
  groups}.
\newblock PhD thesis, University of Cambridge, 2018.
\newblock \href {https://doi.org/10.17863/CAM.32926}
  {\path{doi:10.17863/CAM.32926}}.

\bibitem[BBI01]{BBI}
D.~Burago, Y.~Burago, and S.~Ivanov.
\newblock {\em A Course in Metric Geometry}, volume~33 of {\em Graduate Studies
  in Mathematics}.
\newblock AMS, 2001.
\newblock \href {https://doi.org/10.1090/gsm/033} {\path{doi:10.1090/gsm/033}}.

\bibitem[BBM21]{BBM}
J.~Belk, C.~Bleak, and F.~Matucci.
\newblock Rational embeddings of hyperbolic groups.
\newblock {\em J. Comb. Algebra}, 5(2):123–183, 2021.
\newblock \href {https://doi.org/10.4171/JCA/52} {\path{doi:10.4171/JCA/52}}.

\bibitem[Bel19]{mathoverflow}
J.~Belk.
\newblock Two definitions of horofunction for gromov hyperbolic spaces.
\newblock MathOverflow, 2019.
\newblock URL: \url{https://mathoverflow.net/q/328741}.

\bibitem[BF19]{BF}
J.~Belk and B.~Forrest.
\newblock Rearrangement groups of fractals.
\newblock {\em Trans. Amer. Math. Soc.}, 372(7):4509--4552, 2019.
\newblock \href {https://doi.org/10.1090/tran/7386}
  {\path{doi:10.1090/tran/7386}}.

\bibitem[BGN03]{BGN}
L.~Bartholdi, R.~Grigorchuk, and V.~Nekrashevych.
\newblock From fractal groups to fractal sets.
\newblock In {\em Fractals in Graz 2001}, Trends in Mathematics, page 25–118,
  2003.
\newblock \href {https://doi.org/10.1007/978-3-0348-8014-5_2}
  {\path{doi:10.1007/978-3-0348-8014-5_2}}.

\bibitem[BGS85]{BGS}
W.~Ballmann, M.~Gromov, and V.~Schroeder.
\newblock Manifolds of nonpositive curvature.
\newblock {\em Progress in Mathematics}, 61, 1985.

\bibitem[BH13]{BH}
M.~R. Bridson and A.~Haefliger.
\newblock {\em Metric spaces of non-positive curvature}, volume 319 of {\em
  Grundlehren der mathematischen Wissenschaften}.
\newblock Springer Berlin, Heidelberg, 2013.
\newblock \href {https://doi.org/10.1007/978-3-662-12494-9}
  {\path{doi:10.1007/978-3-662-12494-9}}.

\bibitem[Can84]{C}
J.~W. Cannon.
\newblock The combinatorial structure of cocompact discrete hyperbolic groups.
\newblock {\em Geom Dedicata}, 16:123–148, 1984.
\newblock \href {https://doi.org/10.1007/BF00146825}
  {\path{doi:10.1007/BF00146825}}.

\bibitem[CP93]{CP}
M.~Coornaert and A.~Papadopoulos.
\newblock {\em Symbolic Dynamics and Hyperbolic Groups}, volume 1539 of {\em
  Lecture Notes in Mathematics}.
\newblock Springer Berlin, Heidelberg, 1993.
\newblock \href {https://doi.org/10.1007/BFb0092577}
  {\path{doi:10.1007/BFb0092577}}.

\bibitem[CP01]{CP2}
M.~Coornaert and A.~Papadopoulos.
\newblock Horofunctions and symbolic dynamics on gromov hyperbolic groups.
\newblock {\em Glasgow Mathematical Journal}, 43(3):425–456, 2001.
\newblock \href {https://doi.org/10.1017/S0017089501030063}
  {\path{doi:10.1017/S0017089501030063}}.

\bibitem[DK03]{DK}
R.~Diestel and D.~Kühn.
\newblock Graph-theoretical versus topological ends of graphs.
\newblock {\em Journal of Combinatorial Theory, Series B}, 87(1):197--206,
  2003.
\newblock \href {https://doi.org/10.1016/S0095-8956(02)00034-5}
  {\path{doi:10.1016/S0095-8956(02)00034-5}}.

\bibitem[Don16]{D}
C.~R. Donoven.
\newblock {\em Fractal, group theoretic, and relational structures on Cantor
  space}.
\newblock PhD thesis, University of St Andrews, 2016.
\newblock URL: \url{http://hdl.handle.net/10023/11370}.

\bibitem[GMS19]{GMS}
D.~Groves, J.F. Manning, and A.~Sisto.
\newblock Boundaries of dehn fillings.
\newblock {\em Geometry and Topology}, 23(6):2929--3002, 2019.
\newblock \href {https://doi.org/10.2140/gt.2019.23.2929}
  {\path{doi:10.2140/gt.2019.23.2929}}.

\bibitem[GNS00]{GNS}
R.~I. Grigorchuk, V.~Nekrashevych, and V.~Sushchanski\v{i}.
\newblock Automata, dynamical systems, and groups.
\newblock {\em Proc. Steklov Inst. Math}, 231(4):128–203, 2000.

\bibitem[Gro87]{G}
M.~Gromov.
\newblock Hyperbolic groups.
\newblock {\em Essays in group theory}, 8, 1987.

\bibitem[Kai03]{K}
V.~A. Kaimanovich.
\newblock Random walks on sierpi{\'{n}}ski graphs: Hyperbolicity and stochastic
  homogenization.
\newblock In {\em Fractals in Graz 2001}, Trends in Mathematics, pages
  145--183, 2003.
\newblock \href {https://doi.org/10.1007/978-3-0348-8014-5_5}
  {\path{doi:10.1007/978-3-0348-8014-5_5}}.

\bibitem[KB02]{KB}
I.~Kapovich and N.~Benakli.
\newblock Boundaries of hyperbolic groups.
\newblock In {\em Combinatorial and geometric group theory}, volume 296 of {\em
  Contemporary Mathematics}, page 39–93, 2002.
\newblock \href {https://doi.org/10.1090/conm/296}
  {\path{doi:10.1090/conm/296}}.

\bibitem[Kur66]{Ku}
K.~Kuratowski.
\newblock {\em Introduction a la theorie des ensembles et a la topologie}.
\newblock L'einseignement mathematique, Geneve, 1966.

\bibitem[LW09]{LW}
K-S. Lau and X-Y. Wang.
\newblock Self-similar sets as hyperbolic boundaries.
\newblock {\em Indiana University Mathematics Journal}, 58(4):1777--1795, 2009.
\newblock \href {https://doi.org/10.2307/24903289}
  {\path{doi:10.2307/24903289}}.

\bibitem[Lö17]{L}
C.~Löh.
\newblock {\em Geometric Group Theory: An Introduction}.
\newblock Universitext. Springer Cham, 2017.
\newblock \href {https://doi.org/10.1007/978-3-319-72254-2}
  {\path{doi:10.1007/978-3-319-72254-2}}.

\bibitem[Nek03]{N}
V.~Nekrashevych.
\newblock Hyperbolic spaces from self-similar group actions.
\newblock {\em Algebra and Discrete Mathema}, 2(1):77–86, 2003.

\bibitem[Nek07]{N2}
V.~Nekrashevych.
\newblock Self-similar groups and their geometry.
\newblock {\em Sao Paulo Journal of Mathematical Sciences}, 1(1):41–95, 2007.

\bibitem[Paw15]{P}
D.~Pawlik.
\newblock Gromov boundaries as markov compacta, 2015.
\newblock URL: \url{https://arxiv.org/abs/1503.04577}, \href
  {https://doi.org/10.48550/ARXIV.1503.04577}
  {\path{doi:10.48550/ARXIV.1503.04577}}.

\bibitem[Rus14]{R}
B.~Rushton.
\newblock Classification of subdivision rules for geometric groups of low
  dimension.
\newblock {\em Conform. Geom. Dyn.}, 18:171--191, 2014.
\newblock \href {https://doi.org/10.1090/S1088-4173-2014-00269-0}
  {\path{doi:10.1090/S1088-4173-2014-00269-0}}.

\bibitem[Rus17]{R2}
B.~Rushton.
\newblock Subdivision rules for all gromov hyperbolic groups, 2017.
\newblock URL: \url{https://arxiv.org/abs/1708.02366}, \href
  {https://doi.org/10.48550/arXiv.1708.02366}
  {\path{doi:10.48550/arXiv.1708.02366}}.

\bibitem[Vä05]{V}
J.~Väisälä.
\newblock Gromov hyperbolic spaces.
\newblock {\em Expositiones Mathematicae}, 23(3):187--231, 2005.
\newblock \href {https://doi.org/10.1016/j.exmath.2005.01.010}
  {\path{doi:10.1016/j.exmath.2005.01.010}}.

\bibitem[Wil70]{W}
S.~Willard.
\newblock {\em General Topology}.
\newblock Addison-Wesley Series in Mathematics. Addison-Wesley, 1970.

\bibitem[WW05]{WW}
C.~Webster and A.~Winchester.
\newblock Boundaries of hyperbolic metric spaces.
\newblock {\em Pac. J. Math.}, 221(1):147--158, 2005.
\newblock \href {https://doi.org/10.2140/pjm.2005.221.147}
  {\path{doi:10.2140/pjm.2005.221.147}}.

\bibitem[WW06]{WW2}
C.~Webster and A.~Winchester.
\newblock Busemann points of infinite graphs.
\newblock {\em Transactions of the American Mathematical Society},
  358(9):4209–4224, 2006.
\newblock \href {https://doi.org/10.1090/S0002-9947-06-03877-3}
  {\path{doi:10.1090/S0002-9947-06-03877-3}}.

\end{thebibliography}

\end{document}